\documentclass[aap]{imsart}
\RequirePackage{amsthm,amsmath,amsfonts,amssymb}
\RequirePackage[numbers]{natbib}

\RequirePackage[normalem]{ulem}

\RequirePackage{mathrsfs}
\RequirePackage{amssymb}
\RequirePackage{srcltx} 
\RequirePackage{dsfont}
\RequirePackage{hyperref}
\RequirePackage{color}
\RequirePackage{enumerate}
\RequirePackage{tikz}
\RequirePackage{comment}
\RequirePackage{lscape}
\RequirePackage{graphicx}
\RequirePackage{tabularx}
\RequirePackage{pdfpages}    % Felix: used in Figure fig:THEproof-illustration to draw with tikz above an imported picture

\startlocaldefs

\renewcommand{\d}{{\rm d}} 
 
\newcommand{\eps}{\varepsilon}

\newcommand\numberthis{\addtocounter{equation}{1}\tag{\theequation}}

\definecolor{Red}{rgb}{1,0,0}

\def\emptyset{\varnothing} 
\def\d{{\rm d}}

\def\cA{{\mathcal A}}
\def\cD{{\mathcal D}}

\def\cS{{\mathcal S}}

\newfam\Bbbfam 
\font\tenBbb=msbm10 
\font\sevenBbb=msbm7 
\font\fiveBbb=msbm5 
\textfont\Bbbfam=\tenBbb 
\scriptfont\Bbbfam=\sevenBbb 
\scriptscriptfont\Bbbfam=\fiveBbb

\def\2{\mathbf 2}

\newcommand{\R}     {\mathbb{R}} 
 
\newcommand{\N}     {\mathbb{N}} 
\renewcommand{\P}   {\mathbb{P}} 
 
\newcommand{\E}     {\mathbb{E}}

\newcommand{\smfrac}[2]{\textstyle{\frac {#1}{#2}}}

\def\1{{\mathchoice {1\mskip-4mu\mathrm l}      % Blackboard bold 1 
{1\mskip-4mu\mathrm l} 
{1\mskip-4.5mu\mathrm l} {1\mskip-5mu\mathrm l}}} 
 
\def\comment#1{} 
\newtheoremstyle{thm}{2ex}{2ex}{\itshape\rmfamily}{} 
{\bfseries\rmfamily}{}{1.7ex}{} 
 
\newtheoremstyle{rem}{1.3ex}{1.3ex}{\rmfamily}{} 
{\itshape\rmfamily}{}{1.5ex}{}

\numberwithin{equation}{section}
\renewcommand{\theequation}{\thesection.\arabic{equation}} 
 
\newtheorem{theorem}{Theorem}[section] 
\newtheorem{lemma}[theorem]{Lemma} 
\newtheorem{prop}[theorem] {Proposition} 
\newtheorem{cor}[theorem]  {Corollary}

\theoremstyle{remark}
\newtheorem{defn}[theorem] {Definition}

\newtheorem{remark}[theorem]{Remark} 

\endlocaldefs

\begin{document}

\begin{frontmatter}
\title{From clonal interference to Poissonian interacting trajectories}
%\title{A sample article title with some additional note\thanksref{T1}}
\runtitle{Clonal interference}
%\thankstext{T1}{A sample of additional note to the title.}

\begin{aug}
%%%%%%%%%%%%%%%%%%%%%%%%%%%%%%%%%%%%%%%%%%%%%%%
%% Only one address is permitted per author. %%
%% Only division, organization and e-mail is %%
%% included in the address.                  %%
%% Additional information can be included in %%
\author[A]{\fnms{Felix}~\snm{Hermann}\ead[label=e1]{hermann@math.uni-frankfurt.de}},
\author[B]{\fnms{Adrián}~\snm{González Casanova}\ead[label=e2]{agonz591@asu.edu}},
\author[D]{\fnms{Renato}~\snm{Soares dos Santos}\ead[label=e4]{rsantos@mat.ufmg.br}},
\author[E]{\fnms{András}~\snm{Tóbiás}\ead[label=e5]{tobias@cs.bme.hu}},
\and
\author[A]{\fnms{Anton}~\snm{Wakolbinger}\ead[label=e6]{wakolbinger@math.uni-frankfurt.de}}

%%%%%%%%%%%%%%%%%%%%%%%%%%%%%%%%%%%%%%%%%%%%%%
%% Addresses                                %%
%%%%%%%%%%%%%%%%%%%%%%%%%%%%%%%%%%%%%%%%%%%%%%
\address[A]{Goethe-Universität Frankfurt am Main, FB 12, Institut für Mathematik, 60629 Frankfurt, Germany\printead[presep={,\ }]{e1,e6}}

\address[B]{School of Mathematical and Statistical Sciences, and Biodesign Institute, Arizona State University.  797 E Tyler St, Tempe, AZ 85281, United States\printead[presep={,\ }]{e2}}

\address[D]{Departamento de Matem\'atica, Universidade Federal de Minas Gerais, Av.~Antônio Carlos 6627,  31270-901
Belo Horizonte, Brazil\printead[presep={,\ }]{e4}}

\address[E]{Department of Computer Science and Information Theory, Faculty of Electrical Engineering and Informatics, Budapest University of Technology and Economics, Műegyetem rkp. 3., H-1111 Budapest, Hungary, and HUN-REN Alfréd Rényi Institute of Mathematics, Reáltanoda utca 13--15, H-1053 Budapest, Hungary \printead[presep={,\ }]{e5}}

\end{aug}

\begin{abstract}
We consider a population whose size $N$ is fixed over the generations, and in which random beneficial mutations arrive at a rate of order $1/\log N$ per generation. In this so-called Gerrish--Lenski regime, typically a finite number of contending mutations are present together with one resident type.
These mutations compete for fixation, a phenomenon addressed as clonal interference. We introduce and study a Poissonian system of  interacting trajectories (PIT), and prove that it arises as a large population scaling limit of the logarithmic sizes of the contending clonal subpopulations  in a continuous-time Moran model with strong selection.   We show that the PIT  exhibits an almost surely positive asymptotic rate of fitness increase (speed of adaptation), which turns out to be finite if and only if fitness increments have a finite expectation. We relate this speed to heuristic predictions from the literature. Furthermore, we derive a functional central limit theorem for the fitness of the resident population in the PIT. 
\end{abstract}



\end{frontmatter}

\setcounter{tocdepth}{3}

% \tableofcontents

\setcounter{section}{0}

\bigskip

%\centerline{\small(\version)} 
%\vspace{.5cm} 
\section{Introduction}\label{Intro}
Clonal interference~\cite{GL98,G01,PK07,BGPW19} is the interaction between multiple beneficial mutations that compete for fixation in a population. 
In this paper we introduce a {\em Poissonian system of interacting trajectories} (PIT) that in an appropriate parameter regime  emerges as a scaling limit of clonal subpopulation sizes and thus
captures important  features of clonal interference.
The sources of randomness in the PIT as well as the deterministic interactive dynamics of the trajectories are defined at the beginning of Section 2, and a cut-out of a realisation of the PIT is displayed in the right panel of Figure~\ref{fig-moran_sim}. As we will explain shortly, the PIT arises naturally in the context of population genetics, but we believe that it is of interest in its own right.
Consequently, part of the present work is devoted to a first study of its properties, and the corresponding sections (2.1,  2.3, 5 and 6)  can be read without background in population genetics. A substantial part of our work, however, is devoted to showing that the PIT arises as a scaling limit (as the total population size diverges) in a multitype Moran model with recurrent beneficial mutations.
Here, the Moran model was chosen for convenience, but we believe that the PIT is universal in the sense  that an analogous limiting result holds e.g.\ also for a large class of Cannings models. 

\textbf{Heuristics and scaling regime.}
Let us now give a brief description of how the PIT appears in a population-genetic framework. Consider a population whose size $N$ is large and constant over the generations.
Beneficial mutations arrive in the population at rate $\mu_N$ per generation, and each of these mutations induces a random fitness increment, where the fitness increments (denoted by $A_i$) are assumed to be independent and identically distributed. Individuals carrying the same type form a {\em (clonal) subpopulation}.
Figure~\ref{fig-moran_sim} (left) illustrates how relative subpopulation sizes evolve over time, approximating logistic curves for large~$N$. Logarithmic size-scaling transforms the exponential growth and decline phases of these logistic curves to linear trajectories while the  competition phases shrink to points where a trajectory reaches height 1 with a certain positive slope (and is kinked to slope 0), and at the same time another trajectory leaves height 1 with the opposite
slope (Figure~\ref{fig-moran_sim}, mid). As it turns out, a scaling limit of this picture leads to a system of piecewise linear interacting trajectories depicted
in Figure~\ref{fig-moran_sim} (right).

\begin{figure}
  \includegraphics[width=4.7cm,trim={0.5cm, 1cm, 0.8cm, 2cm},clip]
   {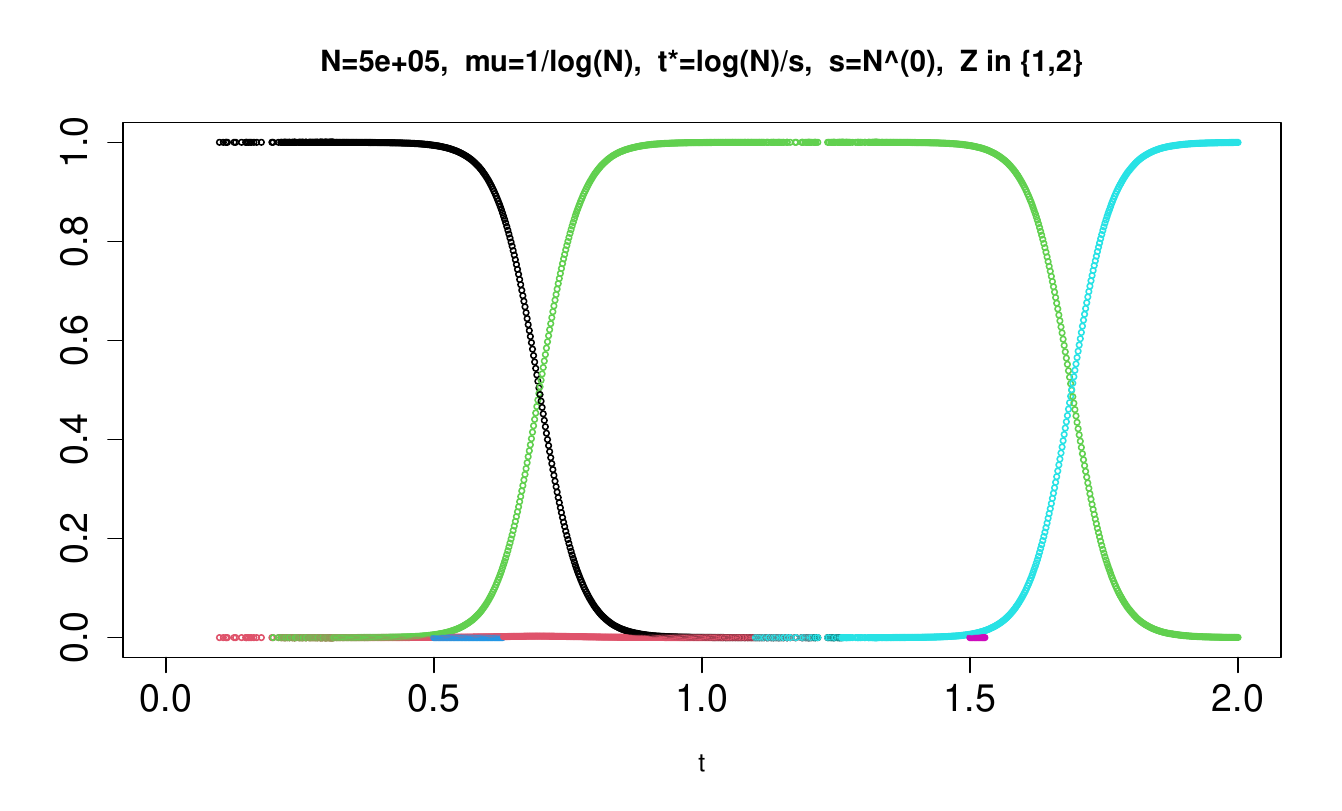}
  \includegraphics[width=4.7cm,trim={0.5cm, 1cm, 0.8cm, 2cm},clip]
   {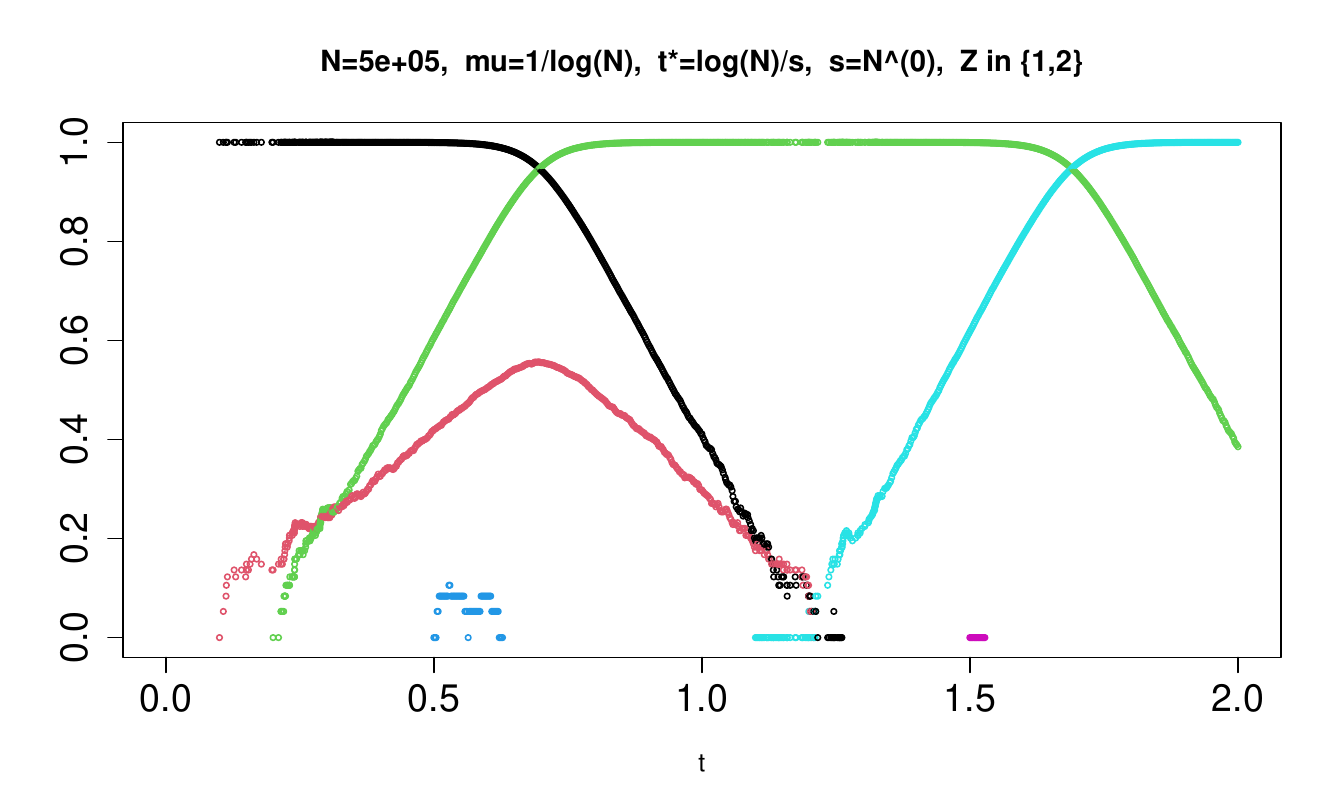}
  \includegraphics[width=4.7cm,trim={0.5cm, 1cm, 0.8cm, 2cm},clip]
   {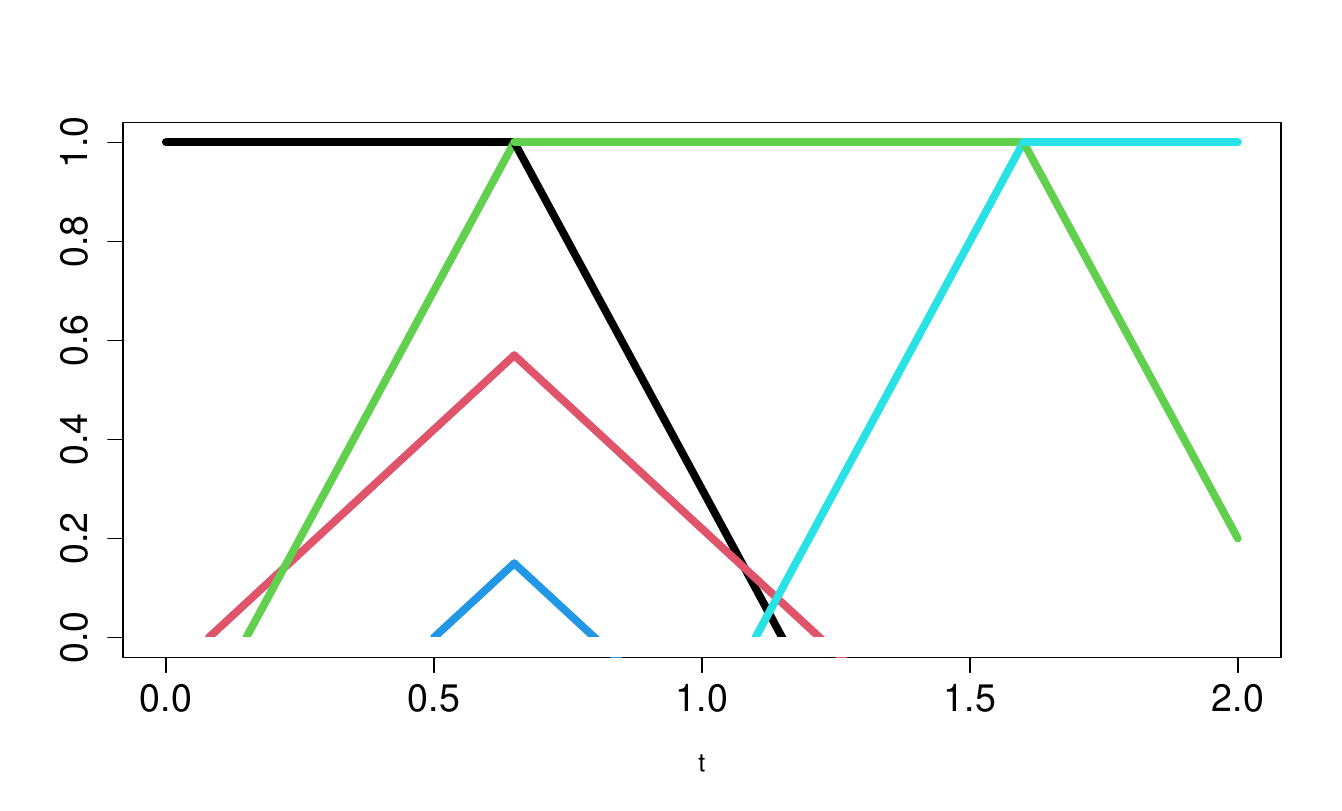} 
   \caption{\label{fig-moran_sim}
     This figure depicts a simulation of a Moran model with mutation and selection
     (cf.\ Section~\ref{sec-model}) in the Gerrish--Lenski regime
     with population size $N=500\,000$ and fitness increment distribution
     $\gamma = \frac12\delta_{\{1\}}+\frac12\delta_{\{2\}}$.
     Left: sub-population sizes divided by $N$, approximating logistic curves;
     middle: logarithmic sub-population sizes divided by $\log N$, approximately giving piecewise linear trajectories and making effects of clonal interference on the first (red) mutation visible;
     right: stylized version of these trajectories providing a good guess of the scaling limit -- i.e.\ the PIT.
   }
\end{figure}

In this study we focus on \emph{strong} selection, i.e.\ where the distribution of the $A_i$ does not scale with $N$.
Then, the linear growth of logarithmic subpopulation sizes appears on a timescale of $\log N$ generations per unit.

For the mutation rate we consider the case where $\mu_N$ is of order $1/\log N$. We refer to this as the \emph{Gerrish--Lenski regime} since it was proposed by the authors of \cite{GL98}. Indeed in the setting at hand, this regime is characteristic for a non-trivial finite number of subpopulations contending for residency and fixation, which is the hallmark of clonal interference.

\textbf{Main results.}
In the framework of the Moran model (that is briefly described in Section~\ref{sec-THEtheorem} and formally specified in Section~\ref{sec-model}), our main result, Theorem~\ref{theorem-THE}, asserts the joint distributional convergence (as $N\to \infty$) of four relevant functionals to the corresponding functionals of the PIT: (i) the rescaled logarithmic frequencies, (ii) the fitness values of clonal subpopulations, (iii) the average population fitness, and (iv) the ancestral tree of mutations.

Our main result on the PIT itself concerns the existence of a speed of adaptation, i.e., the average increase of fitness:   
Denote by $F(t)$ the fitness of the resident at time $t >0$. 
We show in Theorem~\ref{theorem-speed} that $F(t)/t$ converges almost surely as $t \to \infty$ to a deterministic limit. The limit is positive and finite whenever the distribution of fitness increments has a finite first moment, and infinite otherwise. In the special case when the fitness increments are deterministic and constant, Proposition~\ref{prop-1case} provides an explicit expression for the speed.
In general, obtaining a precise numerical value or an explicit formula for the speed seems difficult, and we postpone investigations in this direction to future work. Furthermore, we derive a functional central limit theorem for the fitness of the resident population in the PIT.

\textbf{Related work.} Gerrish and Lenski~\cite{GL98} were particularly interested in a prediction of the slowing down of the speed of adaptation caused by clonal interference.
Their heuristics consisted in eliminating contending mutations that are outcompeted by a
fitter mutation that is born before they reach residency.
This heuristics was refined by Baake~et~al.~\cite{BGPW19}, where it was further considered  that mutations from the past can also affect the fate of a contending mutation, using a framework which already carried certain features of the PIT. We will elaborate more on this in Section~\ref{speedheur}.

In the setting of adaptive dynamics, the effects of clonal interference were analysed by Billiard and Smadi~\cite{BS17, BS19}. These authors studied the case of three competing types rigorously. In [5] they discuss three parameter regimes for the mutation rate, where one of these (being intermediate between the regimes of rare and frequent mutations) corresponds to the Gerrish-Lenski regime. In this regime, 
for
most of the time there is a unique resident subpopulation, 
and new mutations typically happen in the resident population.

In the case of logistic competition, the regime of rare mutations was investigated in the seminal paper by Champagnat~\cite{C06},
see also the references therein. Scaling time by the mutation rate, the durations of ``selective sweeps'' vanish, and the process of the fitness of the
population converges to a pure jump process called the trait substitution sequence of adaptive
dynamics, as it was shown in~\cite{C06}. The case where coexistence is possible was first studied
by Champagnat and Méléard~\cite{CM11}.

Selective sweeps in population-genetic models (with constant population size $N$) were already studied earlier, see e.g.~\cite{KHL89,DS04}.

A rare mutation regime with $\mu_N\sim N^{-a}$ and selective advantages of mutants scaling like $s_N\sim N^{-b}$ was considered by Gonz\'alez Casanova~et~al.~\cite{GKWY16}.
There, conditions were imposed on $a$ and $b$ that guarantee that with high probability as $N\to \infty$ no mutant family is present in the population beside the resident type. In particular, these conditions implied that the times at which a new resident is established converge to a homogeneous Poisson process on the timescale whose unit is $(\mu_N s_N)^{-1}$ generations.  Recently, it was shown by Udomchatpitak and Schweinsberg~\cite{US24} that the same convergence remains true in the mutation regime $\mu_N = o \big( \frac 1{\log N} \big)$ for $s_N \sim N^{-b}$ with any $0<b<1$.

We also point out that piecewise linear trajectories describing the scaling limits of logarithmic frequencies of mutant families appear already in the paper \cite{DM11} by Durrett and Mayberry,
which is an important source of inspiration for this manuscript.
These authors consider polynomial (and thus much faster than inverse logarithmic) mutation rates per generation, leading to a regime where large numbers of ``mutations on mutations'' occur already in the growth phase of a mutant family, such that random genetic drift plays asymptotically no role in the large-population limit. Another difference to our setting is that the authors of~\cite{DM11} consider deterministic fitness increments. Altogether this lead to deterministic limiting systems. 
This polynomial (a.k.a.\ power-law) mutation regime has also been studied in various models of adaptive dynamics~\cite{BCS19,CMT21,CKS21,EK21,BPT23,P23,EK23} and branching processes~\cite{B24}.
These models typically come with a fixed mutation graph; the possible types/traits of individuals form a countable (often finite) set, and mutations between some of these types are possible. The scaling limit does not feature clear parent$\to$child relations anymore since mutations do not appear as a point process but rather as a piecewise constant influx,  even between mesoscopic (size $\Theta(N^\beta)$, $\beta<1$) subpopulations.

A two-type model with logistic competition and with back-and-forth mutations between a wildtype and a strongly beneficial type was studied by Smadi~\cite{S17} for various mutation regimes, including the regime analogous to~$\mu_N \asymp 1$.

\textbf{Structure of the paper.} 
In Section~\ref{sec-modelresults} we describe the limiting as well as the prelimiting model, and present our main results. 
 Specifically, in Section~\ref{PIT} we introduce the dynamics of interacting trajectories and the PIT, in Section~\ref{sec-THEtheorem} we state the corresponding large-population limit result (Theorem~\ref{theorem-THE}), and in Section~\ref{sec-speedresult} we present our results on properties of the PIT (speed of adaptation, functional CLT). As a preparation for the proof of Theorem~\ref{theorem-THE} given in Section~\ref{sec-convergenceproof}, Section~\ref{sec-IT} discusses aspects of the deterministic interactive dynamics that underlies the PIT, and Section~\ref{sec-model} gives a short summary of relevant concepts of the Moran model with recurrent beneficial
mutations and random fitness effects.
The proofs of the results related to the speed of adaptation in the PIT are given in Section~\ref{sec-speedproof}, which can be read independently of Sections~\ref{newsec4} and~\ref{sec-convergenceproof}. Section~\ref{sec:fixandGLh}, which can be read independently of Sections~\ref{newsec4}--\ref{sec-convergenceproof}, discusses the concept of fixation  of mutations within the PIT, and   puts the heuristics from~\cite{GL98} and~\cite{BGPW19} for the speed of adaptation into the framework of the PIT. 
 Section~\ref{modext} gives a glimpse on possible model extensions, including an outlook on regimes of moderate and ``nearly strong'' selection.

\section{Model and main results}\label{sec-modelresults}
\subsection{A Poissonian system of interacting trajectories (PIT)}\label{PIT}
With the picture in mind that was described in the paragraph {\it Heuristics and scaling regime} of Section~\ref{Intro}, we are now going to define in the present section a system of continuous, piecewise linear $[0,1]$-valued trajectories $(H_i)_{i\in \N_0}$ whose interactive dynamics, given a random input, is deterministic. For the sake of proving our scaling limit result Theorem~\ref{theorem-THE}, the random input will be replaced by a deterministic one in Section~\ref{sec-IT}.

  The model parameters for the random input are a positive real number $\lambda$ and a  probability distribution $\gamma$ on $(0,\infty)$.
Let  $T_1 < T_2 < \cdots$ be the points of a Poisson process with intensity measure  $\lambda\,\d t$, $t \in \R_+$. Given $(T_i)_{i\in \N}$ let $A_1,A_2,\ldots$ be iid with distribution $\gamma$, and conditionally on $(T_i,A_i)_{i\in \N}$, let the random variables $B_i$, $i=1,2,\ldots$, be independent and Bernoulli-distributed with
\begin{equation}\label{survprob}
\P(B_i =1) = \frac {A_i}{1+A_i}.
\end{equation}
The intuition behind~\eqref{survprob} is as follows: For $a>0$ the quantity $\frac{a}{1+a}$ is the survival probability of a binary, continuous-time Galton--Watson process with birth rate $1+a$ and death rate $1$, see e.g. page~109 of Athreya and Ney \cite{athreya1972branching}. Likewise, $\frac{a}{1+a}$ is the fixation probability of a  mutant with (strong) selective advantage $a$ in a standard Moran$(N)$-model as $N \to \infty$, see e.g. \cite[Section 2.4]{B21ii}.
In this sense, the $B_i$ provide a ``thinning by survival''.

We will address the $T_i$, $i \in \N$, as \emph{immigration times} (or \emph{birth times}), and we use the convention \mbox{$T_0=0$}. We denote the space of continuous and piecewise linear trajectories from $[0,\infty)$ to $[0,1]$ by~$\mathcal C_{\rm PL}$. Each $h\in \mathcal C_{\rm PL} $ has at time $t$ a {\em height} $h(t)$ and a (right) {\em slope}
\begin{equation}\label{defvh}
   v_{h}(t):=  \lim\limits_{\delta \downarrow 0} \tfrac 1\delta(h(t+\delta)-h(t)). 
\end{equation}
\begin{defn}[Dynamics of the PIT] \label{PITdyn}The \emph{Poissonian system of  interacting trajectories} with parameters $(\lambda,\gamma)$, or briefly the $\mathrm{PIT}(\lambda,\gamma)$, is a $(\mathcal C_{\mathrm{PL}})^{\N_0}$-valued random variable $\mathscr H=(H_i)_{i \in \N_0}$ resulting from the following interactive dynamics (where we abbreviate $v_{H_i}(t) =: V_i(t)$ and write $V_i(t-)$ for the left limit of $r \mapsto V_i(r)$ at time $t$).
\begin{itemize}
        \item $H_0(0) = 1$, $H_1(0) = H_2(0)= \cdots = 0$, and all trajectories $H_i$, $i\in \N_0$, initially have slope~$0$.
        \item At the immigration time $T_i$ the slope of trajectory $H_i$ jumps from $0$ to $A_i$ if $B_i=1$ and stays~$0$ otherwise.
        \item Trajectories continue with constant slope until the next immigration time is reached or one of the trajectories reaches either 1 from below or 0 from above.
        \item Whenever a trajectory   at some time $t$ reaches height $0$ from above, its slope is instantly set to~$0$, and this trajectory then stays at height $0$ forever. 
        \item  {[{\em Kinking rule.}]} Whenever at some time $t$ a trajectory $H_j$ reaches height $1$ from below, then the  slopes of all trajectories whose height is in $(0,1]$ at time $t$ are simultaneously  reduced by $$v^\ast := \max\{V_i(t-) \mid i \in \N_0 \mbox { such that } H_i(t) = 1\},$$
         i.e. for all $H_i \in \mathscr H$ with $H_i(t) > 0$
         \begin{eqnarray} \label{newslope}
          V_i(t) = V_i(t-) - v^\ast .
         \end{eqnarray}  
           \end{itemize}
          \end{defn}
          In the light of the Moran model the intuition for the ``transmittal of kinks''~\eqref{newslope} to all other trajectories that currently have positive height is as follows: As soon as a trajectory reaches height 1, say with slope $v^*$, then, due to the logarithmic scaling, its slope drops instantly to 0 and this trajectory  corresponds to the ``new macroscopic'' subpopulation, so that the mean fitness of the population also makes an upward jump of size $v^*$. All the other contemporary subpopulations experience a decrease of their relative fitness with respect to the dominant subpopulation, and hence the slopes of the corresponding trajectories are reduced by~$v^*$.
          
           The dynamics specified in Definition~\ref{PITdyn} is illustrated by Figure~\ref{fig-limitingexample}.
             Obviously the above stated rules of the interactive dynamics allow to construct  $(H_i)$ from the random input
\begin{equation}\label{Psi}
  \Psi
    := ((T_i, A_i \cdot B_i))_{i \in \mathbb N}.
\end{equation}
By the Poisson colouring theorem, $\Psi$ represents a Poisson process on $\R_+^2$; the intensity measures of its restriction to $\R_+\times (0,\infty)$ is $\lambda^*\d t\,\gamma^*(\d a)$, $t\ge 0$, $a > 0$,  where $\lambda^*$ and $\gamma^*$ are defined in the following remark.
\begin{remark}[Discarding the trajectories of initial slope $0$]\label{remark-discarding0slopes} 
    Let $\mathscr H$ be a PIT with random input~$\Psi$. The trajectories $H_i$ for which $B_i=0$ remain at height $0$ forever and thus will never be contending for reaching height 1.  We define the sequence  $T^*_1<T^*_2<\cdots$ of immigration times of trajectories 
    with initially positive slopes by    $\{T^*_1, T^*_2, \ldots\} = \{T_i\mid B_i=1\}$.
    This thinning reduces the immigration rate $\lambda$ to
  \begin{equation}\label{lambdatilde}
\lambda^*:=\lambda\int\frac {a}{1+a}\gamma(\d a),
 \end{equation}
 and the random variables $(A^*_1, A^*_2, \ldots)$ that come along with the $T^*_j$'s have a biased distribution, being i.i.d.\ copies of a random variable $A^*$ with    $$\P(A^* \in \d a)=\frac\lambda{\lambda^*}\cdot\frac{a}{1+a}\gamma(\d a)=:\gamma^*(\d a).$$
We call the trajectory immigrating at time $T^*_j$ the \emph{$j$-th contending trajectory} (or simply \emph{$j$-th contender}). In view of the intuition coming from the Moran model (and the scaling limit result proved in Section~\ref{sec-THEtheorem}) we will address $T_i$ and $T^*_j$ also as the times of the $i$-th mutation and the $j$-th contending mutation, respectively.
\end{remark}

The next lemma  states properties of the PIT which also play a role in the subsequent definition.
\begin{lemma} \label{ITunique} With probability 1 for any time $t \ge 0, $ 
\begin{itemize}
    \item[(a)] there is no pair $i\neq j$ with $(H_i(t), V_i(t))=(H_j(t), V_j(t))$, 
    \item[(b)] there is exactly one $i\in \N_0$ with $(H_i(t), V_i(t)) = (1,0)$,
    \item[(c)] there is no more than one trajectory reaching height 1 at time $t$.
\end{itemize}
\end{lemma}
\begin{proof}
    (a) Assume there exists a time $t> 0$ and $i\neq j$ for which $(H_i(t), V_i(t)) = (H_j(t), V_j(t)))$ with $H_i(t) > 0$. 
    Then, from the kinking rule, which asserts that at every change of slope the update is the same for
every trait, one derives that $H_i$ and $H_j$ are equal in the past, and in particular one can trace back the
two trajectories up to the time $T := T_i = T_j$. The latter equality, however, occurs with probability $0$.\\
(b) While the uniqueness assertion is a direct consequence of part (a), the fact that there is {\em at least} one $i\in \N_0$ with $(H_i(t), V_i(t)) = (1,0)$ is immediate from Definition~\ref{PITdyn}.
\\ (c) This can be shown by induction along the increasing sequence of times at which some trajectory reaches height 1. (Since we will not make use of assertion (c) in the sequel, we content ourselves with this hint.)
\end{proof}
\begin{defn}[Resident type, fitness of types, resident fitness, resident change times]\label{defrchPIT}\ \\ Let   $\mathscr H=(H_i)_{i \in \N_0}$ be a $\mathrm{PIT}(\lambda,\gamma)$. With probability 1 the  following objects are well-defined:

\begin{itemize}
       \item  For $t\ge 0$, with $i$ as in Lemma~\ref{ITunique} (b), we call $\rho(t)=i$ 
the {\em resident type} at time $t$.
\item The {\em fitness of type $i$} (relative to type $0$) is defined recursively as
\begin{equation}\label{defMi}
M_0 := 0, \quad M_i := M_{\rho(T_i)} + A_i, \quad i =1,2,\ldots 
\end{equation}
\item The {\em resident fitness} at time $t$ (relative to type $0$) is defined as
\begin{equation}\label{finc}
    F(t) := M_{\rho(t)}.   \end{equation}
     \item The times  at which one of the trajectories $H_i$, $i \geq 0$, reaches height~$1$ from below will be called the {\em resident change times},
     denoted $R_1<R_2<\ldots$ 
\end{itemize}
\end{defn}

\begin{figure}[!ht]
\begin{tikzpicture}
 % axis
  \draw[scale=2.9, black, ->] (0.9, 0) -- (5.4, 0) node[right] {};
  \draw[scale=2.9, black, ->] (0.9, -0.1) -- (0.9, 1.2) node[above] {};
  \draw[scale=2.9, black] (0.87,1) node[left] {$1$};
  \draw[scale=2.9, black] (0.87,0) node[left] {$0$};

 % R_i
    \draw[scale=2.9, red] (2.65,-0.1) node[below] {$R_1$};
    \draw[scale=2.9, dotted, black, thick] (2.6, 0) -- (2.6, 1);
    \draw[scale=2.9, blue] (3.4,-0.1)  node[below] {$R_2$};
    \draw[scale=2.9, dotted, black, thick] (3.4, 0) -- (3.4, 1);
    \draw[scale=2.9, orange] (4.533,-0.1) node[below] {$R_3$};
    \draw[scale=2.9, dotted, black, thick] (4.533, 0) -- (4.533, 1);
 
 % dead mutants
    \draw[scale=2.9,gray] (1.4,-0.1) node[below] {$T_2$};
    \draw[scale=2.9, black!20!pink] (2.9,-0.1) node[below] {$T_5$};

 % black mutant
    \draw[scale=2.9, black] (1.75,1.02) node[above] {$0$};
    \draw[scale=2.9, domain=0.9:2.6, smooth, variable=\x, black, thick] plot ({\x}, {1});
    \draw[scale=2.9, domain=2.6:3.4, smooth, variable=\x, black, thick] plot ({\x}, {1-1*(\x-2.6)});
    \draw[scale=2.9, domain=3.4:3.5, smooth, variable=\x, black, thick] plot ({\x}, {0.2-2*(\x-3.4)});

 % green mutant
    \draw[scale=2.9, black!40!green] (1.2,-0.1) node[below] {$T_1$};
    \draw[scale=2.9, black!40!green] (1.3,0.06)  node[above] {$A_1=0.2$};
    \draw[scale=2.9, domain=1.2:2.6, smooth, variable=\x, black!40!green, thick] plot ({\x}, {0.2*(\x-1.2)});
    \draw[scale=2.9, domain=2.6:2.95, smooth, variable=\x, black!40!green, thick] plot ({\x}, {0.28-0.8*(\x-2.6)});
 
 % red mutant    
    \draw[scale=2.9, red] (3.0,1.02) node[above] {$1$};
    \draw[scale=2.9, red] (1.6,-0.1) node[below] {$T_3$};
    \draw[scale=2.9, red] (1.9,0.45)   node[above] {$A_3=1$};
    \draw[scale=2.9, domain=1.6:2.6, smooth, variable=\x, red, thick] plot ({\x}, {\x-1.6});
    \draw[scale=2.9, domain=2.6:3.4, smooth, variable=\x, red, thick] plot ({\x}, {1});
    \draw[scale=2.9, domain=3.4:4.4, smooth, variable=\x, red, thick] plot ({\x}, {1-1*(\x-3.4)});
    % \draw[scale=2.9, domain=4.233:4.828, smooth, variable=\x, red, thick] plot ({\x}, {0.833-1.4*(\x-4.233)});

 % blue mutant
    \draw[scale=2.9, blue] (3.966,1.02)    node[above] {$2$};
    \draw[scale=2.9, blue] (2.325,0.0) node[above] {$A_4=2$};
    \draw[scale=2.9, blue] (2.45,-0.1)  node[below] {$T_4$};
    \draw[scale=2.9, domain=2.5:2.6, variable=\x, blue, thick] plot ({\x}, {2*(\x-2.5)});
    \draw[scale=2.9, domain=2.6:3.4, variable=\x, blue, thick] plot ({\x}, {0.2+1*(\x-2.6)});
    \draw[scale=2.9, domain=3.4:4.533, smooth, variable=\x,   blue, thick] plot ({\x}, {1});
    \draw[scale=2.9, domain=4.533:5.2, smooth, variable=\x, blue, thick] plot ({\x}, {1-0.6*(\x-4.533)});
    \draw[scale=2.9, domain=5.2:5.4, smooth, variable=\x, blue, thick, dashed] plot ({\x}, {1-0.6*(\x-4.533)});

 % orange mutant
    \draw[scale=2.9, orange] (4.966,1.02)  node[above] {$2.6$};
    \draw[scale=2.9, orange] (3.05,0.08) node[above] {$A_6=1.6$};
    \draw[scale=2.9, orange] (3.2,-0.1)  node[below] {$T_6$};
    \draw[scale=2.9, domain=3.2:3.4, smooth, variable=\x, orange, thick] plot ({\x},{1.6*(\x-3.2)}); 
    \draw[scale=2.9, domain=3.4:4.533,smooth, variable=\x, orange, thick] plot ({\x},{0.32+0.6*(\x-3.4)}); 
    \draw[scale=2.9, domain=4.533:5.2, smooth, variable=\x, orange, thick] plot ({\x}, {1});
    \draw[scale=2.9, domain=5.2:5.4, smooth, variable=\x, orange, thick, dashed] plot ({\x}, {1});
\end{tikzpicture}

\caption{
This is an example of a realisation of a PIT with $(T_1, \ldots, T_6) = (1.2,1.4,1.6,2.5,2.9,3.2)$.
The Bernoulli variables $B_1$, \ldots, $B_6$ have realisations $1,0,1,1,0,1$ and the initial slopes of $H_1,H_3,H_4,H_6$ are $(A_1,A_3, A_4, A_6) =(0.2, 1, 2, 1.6)$. Among $H_1, H_2, \dots$, the trajectory $H_3$ is the first one to reach height~$1$. At the time $R_1$ at which this happens, the slope of $H_0$ jumps from $0$ to $-1$, the slope of $H_1$ jumps from $0.2$ to $-0.8$ and the slope of $H_4$ jumps from $1.5$ to $0.5$.  The numbers in the top line of the figure are the current values of the resident fitness $F(t)$. In particular,
$F(R_2)
  %= \max(A_3 + F(T_3), A_4 +  F(T_4))
  = 2$, and
$F(R_3)
  %= \max(A_4 + F(T_4), A_6 +  F(T_6))
  = 2.6.$
}\label{fig-limitingexample} 
\end{figure}

The following lemma (whose proof will be given in Section~\ref{sec-prooflem}) connects the jumps of the resident fitness with the jumps of the slopes of those trajectories whose height is positive at the corresponding (resident change) time.
\begin{lemma}\label{lemFM}a) The resident fitness has the representation
\begin{equation}\label{represF}
  F(t) = \sum_{j: R_j \le t} V_{\rho(R_j)}(R_j-). 
\end{equation}
b) For all $i \in \N_0$ and all $t<t'$ for which $T_i \le t$ and $H_i(t') > 0$, 
\begin{equation}\label{Vjumpsum}
V_i(t')-V_i(t) = F(t)-F(t').
\end{equation}
\end{lemma}

Formula~\eqref{defMi} implicitly decrees that every new type arises through a mutation from the currently resident type. This suggests a ``genealogy of mutations'' whose prelimit (described in the  next subsection) is reminiscent of the ``tree of alleles'' analysed for multitype branching processes in Bertoin~\cite{bertoin2010limit} (albeit there in a neutral situation and in a different scaling regime).

\begin{defn}[Genealogy of mutations in the PIT]\label{genmut}
  For $i\in \N$ we call $\rho(T_i)$ (as specified in Definition~\ref{defrchPIT}) the {\em parent} of type $i$ in the PIT $\mathscr H$. In this case, we call type $i$ a \emph{child} of $\rho(T_i)$.
  This induces an (almost surely defined) random rooted tree $\mathscr G$ with vertex set $\N_0$, edge set $\{(\rho(T_i), i) \,| \, i\in \N\}$ and root~$0$, which we call the {\em ancestral tree of mutations} in the PIT $\mathscr H$. Type $j $ is called an \emph{ancestor} of type $i$ if there is a directed path from $j$ to $i$ in this tree. In that case, we also say that type $i$ is a~\emph{descendant} of type $j$. (See Figure~\ref{fig-genealogyexample} for an illustration.)
  \end{defn}

We see in Figure~\ref{fig-limitingexample} that a resident is not necessarily a descendant of the previous resident. Hence, just observing the sequence of residents provides an incomplete and thus false picture of population ancestry. In order to describe the genealogy of mutants, one needs a finer description involving mesoscopic types, which is one of the main reasons for introducing the PIT.

\begin{figure}[!ht]
\begin{tikzpicture}[bn/.style={circle,fill,draw,minimum size=6mm},every node/.append style={bn}, scale=0.8]

  % initial resident / root
  \node[text=white] (n0) at (0,1.5) {$0$};

  % first generation
  \node[color=black!40!green, text=white] (n1) at (2.5,0)   {$1$};
  \node[color=gray, text=white] (n2) at (2.5,1)   {$2$};
  \node[color=red, text=white] (n3) at (2.5,2)   {$3$};
  \node[color=blue, text=white] (n4) at (2.5,3)   {$4$};

  % second generation
  \node[color=black!20!pink, text=white] (n5) at (5,1.2)   {$5$};
  \node[color=orange, text=white] (n6) at (5,2.8)   {$6$};

  % links to first generation
  \draw[->,thick] (n0) -- (n1);
  \draw[->,thick,dotted] (n0) -- (n2);
  \draw[->,very thick] (n0) -- (n3);
  \draw[->,thick] (n0) -- (n4);

  % links to second generation
  \draw[->,thick,dotted] (n3) -- (n5);
  \draw[->,very thick] (n3) -- (n6);

   % \Vertex[x=0]{$0$}
\end{tikzpicture}
\caption{
  Illustration of the genealogy corresponding to Figure~\ref{fig-limitingexample}.
  Dotted edges mark mutations lost to genetic drift, i.e.~leading to non-contenders. Types 1, 3, 4 and 6 are contenders. Type 1 never becomes resident. Type 4 becomes resident but afterwards is ``kinked to extinction'' without having any descendant. Type 3, after becoming resident in a non-solitary way, spawns  type 6 as its descendant, which then becomes resident in a solitary way.
  The bold arrows highlight mutations that will be present in the ancestral line of all future individuals.
  The depicted situation points to concepts that will be defined and discussed in Section~\ref{compfix}.
}\label{fig-genealogyexample} 
\end{figure}
\subsection{The PIT as a scaling limit}\label{sec-THEtheorem}
    Now, for our first main result, consider a family of Moran models with mutation and selection, and population size $N$, $N\geq1$. (Readers not familiar with these concepts may find a concise introduction to the Moran model in this context in Section~\ref{sec-model}.)
Genetic types are numbered by $i\in \N_0$ and are distinguished by a numerical pair $(M^N_i,T^N_i)$, where $M^N_i\in\R$ marks the \emph{fitness} and $T^N_i\geq 0$ the first time of arrival of type $i$ to the population, assuming $T^N_i<T^N_{i+1}$ for all $i\geq0$.
    New types $j$ arise at times $T^N_j$ of a Poisson process of rate $\lambda/\log N$, $\lambda>0$, whereby a single randomly chosen individual, say of type $i$, mutates and its fitness increases in an iid fashion, i.e.\ $M^N_j:=M^N_i+A_j$, where $(A_i)_{i\geq1}$ are iid and follow the distribution $\gamma$. For $i>0$ we denote by $\mathfrak p_N(i) $ the type of the individual which mutated into type $i$. Like in Definition~\ref{genmut} this induces a random rooted tree $G^N$ with vertex set $\N_0$, edge set $\{(\mathfrak p_N(i), i) \,| \, i\in \N\}$ and root~$0$. Further, denoting by $X^N_i(t)$ the number of type $i$ individuals present at time $t$, let an individual of type $j$ replace an individual of type $i$ by its (identical) offspring at rate $(1+(M^N_j-M^N_i)^+)X^N_j(t)X^N_i(t)$, i.e.\ neutral reproduction happens at rate $1$, while a higher fitness value confers a linear bonus to the reproduction rate. Initially, we assume a homogeneous population, i.e. $X^N_0(0)=N, T^N_0=0$ and $M^N_0=0$. We call
    $$
      H^N_i(t)
        := \frac{\log^+(X^N_i(t\log N))}{\log N}\,, \quad t\ge 0,\, i \in \N, \qquad\text{where }
        \log^+(x) := \log(1+x), 
    $$
    the {\em logarithmic frequencies} of the types (also known as stochastic exponent or stochastic Hopf--Cole transform). Putting $H^N:= (H_i^N)_{i\in \N_0}$, the sequence of random paths $H^N=(H^N_i)$ takes values in $\mathcal D^{\N_0}$, where
\begin{equation}\label{defD}
  \cD
   := \cD(\R_+,[0,1])
\end{equation}
is the space of c\`adl\`ag functions from $\R_+$ to $[0,1]$, equipped with a metric that induces the Skorokhod $J_1$-convergence on all bounded time intervals, see \cite{ethier2009markov} Sec.~3.5. Writing
    $$\overline F^N(t)=\frac1N\sum_{i\geq0}X^N_i(t\log N)M^N_i$$
    for the 
    average population fitness in the prelimiting Moran model, we can now state our first main result.
  \begin{theorem}\label{theorem-THE}
    For all $i=1,2,\ldots$, as $N\to \infty$, 
    \begin{align} \label{convHNdist}
       & H^N \xrightarrow{\,\,d\,\,} \mathscr H \, \, 
             \mbox { as random elements of the product space }\mathcal D^{\N_0}, \\[.5em]
     & (M_1^N,\ldots, M_i^N) \xrightarrow{\,\,d\,\,} (M_1,\ldots, M_i) \quad \mbox{ in } \, \R^i,    \label{convMNdist}
     \\[.5em]
     &\overline F^N\xrightarrow{d}F
    \quad \mbox{ in } \mathcal D(\R_+, \R_+) \mbox{ with respect to the Skorokhod } M_2\mbox{-topology,}
    \label{convFN}
     \\[.5em] 
     & G^N|_{\{0,\ldots, i\}} \xrightarrow{\,\,d\,\,} \mathscr G|_{\{0,\ldots, i\}}, \label{convGN}
    \end{align}
    where the set of trees with vertex set $\{0,1,\ldots,i\}$ that are rooted in $0$ is endowed with the discrete topology.
  \end{theorem}
The proof of Theorem~\ref{theorem-THE} will be carried out in Section~\ref{sec-convergenceproof}, starting with a short outline in Section~\ref{sec-outline}. 
From the proof of Theorem~\ref{theorem-THE} it is apparent that for any fixed $i$, the convergences~\eqref{convHNdist}--\eqref{convGN} occur jointly in distribution (and not only separately for the four prelimiting objects).
In~\eqref{convFN} we use the $M_2$-topology (which is the weakest of the four Skorokhod topologies) because 
the average fitness at times of a resident change can take any value between the former and the new resident fitness.
In~\eqref{convGN}, the restrictions of $G^N$ and $\mathcal G$ to $\{0,\ldots, i\}$ are trees rooted in $0$  because for all $1 \leq j \leq i$  the parent of $j$ (i.e. $\mathfrak p_N(j)$ resp. $\rho(T_j)$)  is an element of $\{0,\ldots, j-1\}$ by construction of $G^N$ and $\mathcal G$.

We expect that the proof of Theorem~\ref{theorem-THE} extends with minor modifications to the case when the birth times $(T_i)$ do not form a Poisson process but any renewal process whose inter-arrival times have a continuous distribution.

\subsection{Speed of adaptation in the PIT}\label{sec-speedresult} 

For $\lambda > 0$ and $\gamma \in \mathcal M_1((0,\infty))$ let $F(t)$ denote the resident fitness at time $t \geq 0$ in the PIT$(\lambda,\gamma)$ (see Definition~\ref{defrchPIT}). In case the limit of $\frac1tF(t)$, as $t\to\infty$, exists in $[0,\infty]$, we call this the \emph{speed of increase of the resident fitness}, or briefly the \emph{speed of adaptation}. All results stated in this subsection will be proved in Section~\ref{sec-speedproof}.
\begin{theorem}\label{theorem-speed}
\begin{enumerate}[(i)]
\item If $\int_{0}^{\infty} a \gamma(\d a)<\infty$, then, as $t\to\infty$, $\frac1tF(t)$ converges almost surely to a constant~\mbox{$\overline v\in(0,\infty)$}.
\item If $\int_{0}^{\infty} a \gamma(\d a) = \infty$, then $\frac1tF(t)\to\infty$ almost surely.
\end{enumerate}
\end{theorem} 
The proof of Theorem~\ref{theorem-speed} will rely on a renewal structure of the process
$$\Pi(t) := \sum_{i\in I(t)} \delta_{(H_i(t), V_i(t))}, \quad t \ge 0,$$
where $I(t) := \{i \in \N_0:H_i(t) > 0, V_i(t) \ge 0\}$
is the index set of those trajectories that have positive height and nonnegative slope at time $t$. Thanks to the PIT dynamics, $(\Pi(t))_{t\ge 0}$ is a Markov process (taking its values in the finite counting measures on $(0,1] \times [0,\infty)$), and 
those resident change times (recall  Definition~\ref{defrchPIT}) for which  $\Pi(t)$ enters its ``bottleneck state'' $\delta_{(1,0)}$ turn out to be renewal times. These times can also be described as follows:
\begin{defn} \label{srch}We say that a resident change time $R$ is {\em solitary} if $V_i(R)\le 0$ for all $i=1,2,\ldots$, and denote the sequence of solitary resident change times by $L_1 < L_2 < \ldots$.
\end{defn}
 Thus, the solitary resident change times are those ones among the resident change times at which all trajectories of positive height have nonpositive slope. As an example, note that $L_1= R_3$ in  Figure~\ref{fig-limitingexample}.

The process $(F(L_n))_{n\in \N_0}$ thus turns out to have i.i.d. increments, which makes LLN and CLT results for renewal reward processes available. Indeed, the speed of adaptation $\overline v$ will be expressed in Proposition~\ref{renrewprop} as the ratio $\frac {\E[F(L_1)]}{\E[L_1]}$.
This is similar in spirit to the quantity that was found by H.\ Guess as the asymptotic fitness increase in a Wright-Fisher model with multiplicative fitness (\cite[Theorems 4 and 5]{guess1974limit}). We conjecture that an analogue of Guess' result also holds for the Moran model specified in Section~\ref{sec-THEtheorem}, leading to a speed of adaptation $\overline v_N$ in the $N$-th prelimiting model. The task to prove this as well as (criteria for) the convergence $\overline v_N \to \overline v$ is left to future research.

If $\gamma$ is in the normal domain of attraction of an $\alpha$-stable law $\nu$ with $0<\alpha < 1$ (and thus has infinite first moment), then the renewal argument in the proof of Theorem~\ref{theorem-speed} suggests the conjecture that $t^{-1/\alpha} F(t)$ converges in distribution as $t \to \infty$. If this conjecture is true, then an interesting question (again connected with clonal interference) is  how far this limit differs from $\nu$. 

The next proposition is in the spirit of the \emph{thinning heuristics} introduced in~\cite[Section~3.1]{BGPW19}.
\begin{prop}\label{prop-1case}
In the case of deterministic and constant fitness increments, i.e.\ if $\gamma$ is the point mass $\delta_{c}$ in  some $c>0$, we have a.s.\
\begin{equation}\label{fixedc}
  \lim_{t\to\infty} \frac{F(t)}{t}
      = \frac{\lambda c^2}{1+c+\lambda}. \end{equation}
\end{prop}

For $\lambda \to \infty$ the r.h.s.\ of~\eqref{fixedc} converges to the finite value $c^2$, reflecting the fact that a high mutation rate leads to a strong effect of clonal interference, i.e.\ despite being advantageous and surviving the random genetic drift, most mutations are lost by clonal interference.
The next proposition has a similar spirit; roughly spoken it states that, at least for fitness increment distributions $\gamma$ with bounded support and with high mutation rates, the fitness increment  over a fixed time interval is dictated by the mutations whose fitness increment is ``essentially maximal''. 
\begin{prop}\label{highmut}
  Let the support of $\gamma$ be bounded, with $b$ denoting its supremum. For $\lambda > 0$ let $F_\lambda$ be the resident fitness in the {\rm PIT}$(\lambda, \gamma)$.  Then, for any $t>0$
  $$
    F_\lambda(t)
      \xrightarrow{\lambda\to\infty} b(\lceil bt\rceil - 1)
      \qquad\text{in probability.}
  $$
\end{prop}
An application of a functional central limit theorem for renewal reward processes (discussed in Appendix~\ref{s:RenRew}) yields the following functional central limit theorem for the population fitness $F(t)$ in case of finite variance of $\gamma$.

\begin{theorem}\label{theorem-speedCLT}
If $\int_0^{\infty} a^2 \gamma(\d a)<\infty$, then there exists $\sigma > 0$ such that
\begin{equation}\label{FCLT}
\Bigg(\frac{F(nt)-\overline v n t}{\sigma \sqrt{n}}\Bigg)_{t \geq 0} \overset{d}{\longrightarrow} W \quad\mbox{ as } n\to \infty, 
\end{equation}
in the space of c\`adl\`ag functions from $[0,\infty)$ to $\R$ with respect to the Skorokhod $J_1$-topology, where $\overline v$ is as in Theorem~\ref{theorem-speed}(i) and $W = (W_t)_{t \geq 0}$ is a standard Brownian motion.
\end{theorem}

The standard deviation $\sigma$ figuring in \eqref{FCLT}  will be expressed in \eqref{defsigma} with the help of the renewal structure addressed after~Theorem~\ref{theorem-speed}. Readers who are interested mainly in the proof of the results stated in Section~\ref{sec-speedresult} may proceed directly to Section~\ref{sec-speedproof}.

\section{A system of interacting trajectories and its Moran prelimit}\label{newsec4}
This section prepares for the  proof   of the scaling limit result Theorem~\ref{theorem-THE}, which will be given in Section~\ref{sec-convergenceproof}.  While  Section~\ref{sec-IT} extends the PIT dynamics from stochastic to deterministic inputs, Section~\ref{sec-model} gives a concise presentation of the Moran prelimit.
\subsection{A system of interacting trajectories}\label{sec-IT}
\ 
In the inductive concatenation arguments in Sections~\ref{sec:newmut} and~\ref{sec:completion} we will work with deterministic inputs for the PIT dynamics, including more general initial configurations than just one trajectory at height 1 and with slope 0.   This will be prepared in what follows.

Recalling  the space $\mathcal C_{\rm PL}$ from Section~\ref{PIT} and the definition of $v_h$ from~\eqref{defvh}, we observe that for all $h \in \mathcal C_{\rm PL}$ and $t\ge 0$ the pair $(h(t), v_h(t))$ belongs to
$$\mathcal S:= \big(\{0\}\times [0,\infty)\big) \cup \big((0,1)\times \R\big) \cup  \big(\{1\}\times (-\infty, 0]\big).$$
Figure~\ref{fig-limitingexample} displays a few trajectories in $\mathcal C_{\rm PL}$.
Assume that for $k\in \N$ and $\overline \iota \in \N_0 \cup\{\infty\}$  we are given a configuration
\begin{equation}\label{defaleph}
  \aleph:= ((\eta_i,c_i))_{-k < i\le 0}\in {\mathcal S}^{\{-k+1,\ldots,0\}}  
\end{equation}
with $(\eta_0,c_0) := (1,0)$ and  $(\eta_i,c_i)\neq (\eta_{i'},c_{i'})$ for $i\neq i'$, and a configuration
\begin{equation}\label{defbeth}
   \beth:= ((t_i,c_i))_{1\le i < \overline \iota}\in ([0,\infty)^2)^{\{i\, \mid 1\le i< \overline \iota\}}, 
\end{equation}  with $0<t_1<  t_2\cdots$ and $t_i\uparrow \infty$ as $i\to \infty$ in case  $\overline \iota = \infty$. 

We view $\aleph$ as a {\em starting configuration} specifying the height and slope of trajectory $h_i$, $-k < i\le 0$, at time $0$, and $\beth$ as an {\em immigration configuration}, specifying that the trajectory $h_i$, $1\le i < \overline \iota$, has height $0$ for $t \in [0,t_i]$ and right slope $c_i$ at its {\em immigration time}~$t_i$. The symbols $\aleph$ (aleph) and $\beth$ (beth) are reminiscent of ``being present at time 0'' and ``born later''. 
\begin{defn}[Interactive dynamics]\label{defdyn}
For a starting configuration  $\aleph$ as in \eqref{defaleph} and an immigration configuration $\beth$ as in \eqref{defbeth}, let
\begin{equation}\label{defsystH}
  \mathbb H  
    = \{h_i\mid -k < i < \overline \iota\}
    \in \big(\mathcal C_{\rm PL}\big)^{\aleph\, \cup\, \beth}
\end{equation}
result from the  deterministic interactive dynamics on $\mathcal S^{\{i |-k<i<\overline \iota\}}$ described in 
Definition~\ref{PITdyn},
with the only differences being that the index set of the trajectories, which previously was~$\N_0$, now is $\{i |-k<i<\overline \iota\}$, the starting configuration (of heights and slopes at time $0$), which previously was $((1,0))$, is now $\aleph$, and the immigration configuration, which previously was the random  $\Psi= (T_i, A_i B_i)_{i\in \N}$,  is now the (possible finite) deterministic configuration $\beth$.
 We call this $\mathbb H$ {\em the system of interacting trajectories initiated by $\aleph$ and $\beth$}, and denote it by $\mathbb H(\aleph, \beth)$.
\end{defn}
With regard to the kinking rule and in accordance with Definition~\ref{defrchPIT} we define the resident type at time $t$ as
\begin{equation}\label{defrestypeIT}
    \rho(t)
     :=\arg \max \{ v_{h_i}(t-) \colon -k < i < \overline \iota, h_i(t)=1 \}, \quad t> 0;  \quad \rho(0) := 0.
\end{equation}
As an analogue to the definition of the resident fitness in~\eqref{finc}, we put $t_i:=0$ for $-k< i \le 0$, and set
\begin{equation}\label{falternative}
        f(t)
         := c_{\rho(t)} + f(t_{\rho(t)}), \quad t \ge 0,
\end{equation}
for some arbitrarily prescribed value $f(0)\in\R$.

With the resident change times $r_1<r_2<\cdots$ specified as in Definition~\ref{defrchPIT} (now for $(h_i)$ in place of $(H_i)$), the resident fitness has a representation analogous to that in Lemma~\ref{lemFM}:
\begin{equation}\label{represf}
f(t)-f(0) = \sum_{j\ge 1 \colon r_j \le t} v_{\rho_j}(r_j-), \quad t\ge 0.
\end{equation}

\begin{remark}
The dynamics specified in  Definition~\ref{defdyn} and a view on~\eqref{represf} suggest to consider the following system of equations (where we set $\eta_i := 0$ for $i \ge 1$):
\begin{eqnarray}\label{evolh}
h_i(t) &=& \mathds 1_{\{t \geq t_i\}} \Big( \eta_i + \int_{t_i}^t  \big( c_i + f(t_i)-f(s) \big) \mathrm d s \Big)^+,\qquad -k<  i< \overline \iota,\\
\label{resfit}
f(t)&=& \max \{c_j+f(t_j) \mid -k< j < \overline \iota, \, h_j(t) =1\}.
\end{eqnarray}

Indeed, working in a piecewise manner (up to the next immigration or resident change time) one checks readily that (for any prescribed value $f(0)$)  the system (\eqref{evolh}, \eqref{resfit}) has a unique solution $(\mathbb H, f)$, with $\mathbb H= (h_i)_{-k< i < \overline \iota}$ following the dynamics specified in Definition~\ref{defdyn}, and $f$ being the resident fitness defined by~\eqref{defrestypeIT} and~\eqref{falternative}  (recall Figure~\ref{fig-limitingexample} for an illustration).
\end{remark}

    We call the population \emph{initially monomorphic} or \emph{initially homogeneous} when $k=1$ and hence $\aleph=((1,0))$.
    Finally, to relate back to the description of the PIT in Section~\ref{PIT}, recall the Poissonian sequence $\Psi$ given by~\eqref{Psi} and observe that $\mathscr H:=\mathbb H(((1,0)),\Psi)$ is the PIT with random input $\Psi$. In particular, the resident fitness $F$ specified in Definition~\ref{defrchPIT} arises as the $f$ from  \eqref{resfit} with the same random input.
    The construction in this section makes it feasible to define the PIT with general initial conditions and generalize Theorem~\ref{theorem-THE} in a corresponding way; we refrained from this in order to ease the presentation in Section~\ref{sec-modelresults}. However, note that in the proofs in Section~\ref{sec-convergenceproof} we will use restart arguments from non-monomorphic initial conditions, where the construction of the present section  will be instrumental.
\begin{remark}\label{remintuit}
The intuition behind Definition~\ref{defdyn} (which will be justified by the large-population limit of the Moran model defined in Section~\ref{sec-model}, see Theorem~\ref{theorem-THE}) is as follows:

Consider a population of size $N$ which at time $0$ consists of $k$ subfamilies, the one indexed with $i=0$ being  ``macroscopic'', i.e. having logarithmic frequency $\eta_0=1$ in the limit $N\to \infty$) and the others (indexed with $i=-1, ... ,-k+1$) being ``mesoscopic'', i.e. having logarithmic frequency $\eta_i \in (0,1)$ in the limit $N\to \infty$. (Note that the initial configuration in~\eqref{defaleph} is slightly more general: There, with a view towards the initial configurations in Lemma~\ref{lem:multitype-without-kinks} and the induction argument in Proposition~\ref{fixedk}, we also allow for additional macroscopic subpopulations which then are required to have negative  growth rate.)

The initially resident type has fitness $f(0)$, while  for $i< 0$ type $i$ has fitness $c_i + f(0)$.
Thus, as long as all the types with negative indices are mesoscopic, their sizes grow (in an appropriate timescale) exponentially with  growth rate $c_i$,  which is equivalent to saying that their ``logarithmic sizes'' (corresponding to $h_i$) grow linearly with slope $c_i$. At each time $t_i$ contained in the configuration~$\beth$, a~new type arises by a mutation on top of the currently resident type. The fitness increment of this new type is $c_i$,  which is thus also its relative fitness with respect to the resident type at time $t_i$. Assume the first of all types apart from the initial resident that reaches a macroscopic size is type $j$, and assume this happens (under an appropriate time rescaling) at time $r>0$. Then at this time the resident fitness jumps from $f(0)$ to $f(0)+c_{j}$, and all the other types whose growth is still ongoing find themselves in an environment in which competition is more difficult: e.g.\ while the relative fitness (w.r.t.\ the resident) of type ~$-1$ was $c_{-1}$ before time $r$, it jumps to $c_{-1}-c_{j}$ at time $r$, assuming its subpopulation has not been absorbed at $0$ yet, i.e.\ gone extinct.
This establishes the link between \eqref{falternative} and \eqref{resfit}, and explains the ``kinks'' of the trajectories that happen at resident change times, illustrating the last bullet point in Definition~\ref{defdyn}.  In this way, the notions ``relative fitness of type $i$ with respect to the currently resident type'' and ``current slope of the trajectory $i$'' become equivalent.
\end{remark}

\subsection{A Moran model with clonal interference}\label{sec-model}
The prelimiting model which will figure in Theorem~\ref{theorem-THE}  is a Moran model with population size $N$ and infinitely many types. We now define its type space and its Markovian dynamics on the type frequencies. At time $0$ finitely many  types (numbered 
by $-k+1, -k+2,\ldots, 0$) are present, and after time $0$ new types (numbered by $1,2,\ldots$ in the order of their appearance) arrive in the population via mutations at the jump times of a Poisson counting process $\mathcal I^N(t)$, $t\ge 0$,  with rate $\mu_N = \tfrac \lambda {\log N}$. For given numbers $f(0)$ and $c_i$, $-k< i <0$, as well as $c_0:=0$, the fitness levels of the  types $-k+1, \cdots, -1, 0$ that are present at time $t=0$ are defined as
$$m_i^N := c_i +f(0), \qquad -k< i \le 0.$$
For $t\ge 0$ and $i > -k$ we denote the number of type-$i$ individuals at time~$t$ by~$\mathcal X^N_i(t)$, and write  $\mathcal X^N(t) = (\mathcal X^N_i(t))_{i > -k}$. We specify the joint Markovian dynamics of the process $(\mathcal X^N, \mathcal M^N, \mathcal I^N)=(\mathcal X^N(t),  \mathcal M^N(t), \mathcal I^N(t))_{t\ge 0}$, where $\mathcal M^N(t) = (M^N_i)_{-k<i\le \mathcal I^N(t)}$ is the vector of fitness levels of the types that came into play up to time~$t$.

The state space of $(\mathcal X^N, \mathcal M^N, \mathcal I^N)$ is
$$E_N:= \bigcup_{\iota \in \N_0} \Big\{(x_{-k+1}, x_{-k+2}, \ldots, x_\iota, 0, 0, \ldots)\Big | x_i\in \N_0, \sum_{-k<i\le \iota}x_i=N\Big\} \times\R^{k+\iota}\times\{\iota\}.$$
Writing 
$$
  e_i
    = (e_{i,\ell})_{\ell > -k}
    = (0,0,\ldots, 0,1,0,0, \ldots)
$$
for the sequence that has 1 in component $i$ and $0$ in all other components, we can write the transition rates as follows:
\begin{itemize}
     \item \emph{Mutation}: For $(x,m, \iota)\in E_N$, for $-k<j\le \iota$ and $a\in \R_+$, the jump rate of the process  $(\mathcal X^N, \mathcal M^N, \mathcal I^N)$ from $(x,m, \iota)$ to 
       $(x-e_j+e_{\iota+1}, m+(m_j+a)e_{\iota+1}, \iota+1)$ is $\frac {\lambda}{\log N} \frac{x_j}N \gamma(\d a)$. 
 \\[-.5em] 
    \item \emph{Resampling}: For $(x,m, \iota)\in E_N$ and  for $-k<j,\ell\le \iota$, the jump rate of the process  $(\mathcal X^N, \mathcal M^N, \mathcal I^N)$ from $(x,m, \iota)$ to
          $(x+e_j-e_\ell,m, \iota)$ is $x_jx_\ell(1+(m_j-m_\ell)^+)\frac1{N}$.    
   \end{itemize}
In order to pass to a timescale in which one unit of time corresponds to $\log N$ generations, we define the process $(\mathscr X^N, \mathscr M^N, \mathscr I^N)$ by
\begin{equation}\label{xitoX}
 (\mathscr X^N(t), \mathscr M^N(t), \mathscr I^N(t)) := (\mathcal X^N(t\log N), \mathcal M^N(t\log N), \mathcal I^N(t\log N)).   
\end{equation}
The process $\mathscr I^N$ is thus for all $N$ a Poisson counting process with intensity $\lambda$. This allows to couple the sequence of processes $(\mathscr X^N, \mathscr M^N)$, $N\in \N$, via  ingredients which we encountered already in Section~\ref{PIT}:
\begin{itemize}
\item Let $(T_i)_{i\in\N}$ be the times of a Poisson process of rate $\lambda>0$. 
\item Let $(A_i)_{i\in\N}$ be an i.i.d.\ sequence of $\gamma$-distributed random variables, independent of $(T_i)_{i}$.
\end{itemize}
\begin{remark}\label{repres} a) Like $(\mathcal X^N,\mathcal M^N, \mathcal I^N)$, the 
process $(\mathscr X^N,\mathscr M^N, \mathscr I^N)$ defined by~\eqref{xitoX} is a Markovian jump process. Its dynamics  may be specified  using $(T_i)$ and $(A_i)$ as follows:
  \begin{itemize}
     \item \emph{Mutation}: At time $T_i$ the process  $(\mathscr X^N, \mathscr M^N, \mathscr I^N)$ jumps from state $(x,m, \iota)$ to state
          $(x-e_j+e_i, m+e_i\cdot(A_i+m_j), \iota +1)$ with probability
          $x_j/N$,\, $j>-k$.\\
    (Note that, when a mutation event occurs as above, necessarily $i = \iota+1$.)
 \\[-.5em] 
    \item \emph{Resampling}:
      %The process  $(\mathscr X^N, \mathscr M^N, \mathscr I^N)$ jumps from state $(x,m,\iota)$ to state $(x+e_j-e_\ell,m,\iota)$ at rate $x_jx_\ell(1+(m_j-m_\ell)^+)\frac{\log N}{N}$, \quad $-k<j, \, \ell\le \iota$.
      The jump rate of the process $(\mathscr X^N, \mathscr M^N, \mathscr I^N)$ from state $(x,m,\iota)$ to state $(x+e_j-e_\ell,m,\iota)$\ \ is\ \ $x_jx_\ell(1+(m_j-m_\ell)^+)\frac{\log N}{N}$, \ $-k<j, \, \ell\le \iota$.
   \end{itemize}
   b) The just described dynamics on the type frequencies can also be obtained via a graphical representation with three types of transitions:
   \begin{enumerate}
       \item a \emph{mutation} occurs at each time $T_i$ on an individual that is randomly sampled from the population, resulting in the founder of a new type with fitness increment $A_i$ relative to its parent;
       \item \emph{neutral reproduction} occurs with rate proportional to $\log N/N$ for each ordered pair of individuals and leads to the first one reproducing (i.e., giving rise to another individual with the same type and fitness) and the second one dying; 
       \item \emph{selective reproduction} occurs for each pair of individuals with rate proportional to $\log N/N$ times their fitness difference, and leads to the fitter individual reproducing and the less fit one dying. 
   \end{enumerate}
\end{remark}

\begin{defn}\label{defGHN}
  \begin{enumerate}[a)]
    \item 
    Using the just described graphical representation we can trace back the individual ancestral lineages and in particular define a genealogy of mutations:
    For $i >0$ we say that $j < i$ is the {\em parent} of type $i$ if type $i$ originated via a mutation of a type $j$-individual. We then write $j=: \mathfrak p_N(i)$. 
    In the case $k=1$ (i.e.\ if all individuals at time $0$ carry the same type) this induces a random tree $G^N$ with vertex set~$\N_0$, edge set $\{(\mathfrak p_N(i), i) \,| \, i\in \N\}$ and root~$0$.
    \item 
    The {\em logarithmic frequency} (or briefly {\em the height}) {\em of the $i$-th mutant family at time $t$} is defined as
    \begin{align}
      H^N_i(t)
       &:= \frac{\log^+\big(\mathscr X^N_i(t)\big)}{\log N}.
    \label{eq:logplus}
    \end{align}
    The sequence of random paths $H = (H_i^N)$ takes values in $\mathcal D^\N$ equipped with the product topology, with the space $\mathcal D$ defined in~\eqref{defD}. 
  \end{enumerate}
\end{defn}

\newpage

\section{The PIT as a scaling limit: Proof of Theorem~\ref{theorem-THE}}\label{sec-convergenceproof}

\subsection{Outline of the proof}\label{sec-outline}

As the very basis for the proof, in Section~\ref{sec:bin-bran-pros} we state results on super- and subcritical
binary branching processes (Lemmas~\ref{lem:gw-supercritical} and ~\ref{lem:gw-subcritical} resp.), regarding
convergence to piecewise linear functions under logarithmic scaling.
In Section~\ref{sec:multitype-nomut},
these findings will first be transferred to the Moran model without mutations in the presence of finitely many 
mutant families of mesoscopic size, to show linear growth of these \emph{mesoscopes} between resident changes, cf.\ Lemma~\ref{lem:multitype-without-kinks}, using stochastic ordering via comparison of jump rates.
Similarly, using time-change and stochastic ordering for bundles of mesoscopes, 
Lemma~\ref{lem:multitype-sweep} zooms into the resident changes,
showing that their time span indeed vanishes on the logarithmic timescale, providing a kink in the PIT.
These evolutionary phases will then be pieced together by concatenating applications of 
Lemmas~\ref{lem:multitype-without-kinks} and~\ref{lem:multitype-sweep} along resident change times. In Sections~\ref{sec:newmut} and~\ref{sec:completion}, the proof of Theorem~\ref{theorem-THE} will be completed by an induction along the times of mutations. Here, an essential intermediate step is Proposition~\ref{convprob}, in which the logarithmic type frequencies of the $N$-th prelimiting system are coupled with a system of interacting trajectories as defined in Section~\ref{PIT}, but now with the Bernoulli random variables $B_i$ replaced by indicators $B^N_i$ which predict whether the $i$-th mutant becomes a contender in the prelimit.

Possible generalizations of this methodology will be discussed in Section~\ref{sec:THE-genar}.

\subsection{Auxiliary results from branching processes}\label{sec:bin-bran-pros}

The two lemmata in this subsection reflect the well-known fact that the logarithm of sped-up sub- or supercritical Galton--Watson processes scale to linear functions. In order to state them, 
let $Z=(Z_t)_{t \geq 0}$ be a continuous-time binary Galton--Watson process with individual birth and death rates $b, d \geq 0$ respectively, and let $s:=|b-d|$.
We denote by $T_0:=\inf \{t \geq 0 \mid Z_t = 0\}$ the extinction time, 
and by $\P_z$ the law of $Z$ started at $z \in \N_0$.

The next lemma follows from Theorem~\ref{thm:superbranch},
which gathers some useful facts in the supercritical case.

\begin{lemma}\label{lem:gw-supercritical}
Assume $b>d$, abbreviate $\{Z\not\to 0\} := \{Z_t \geq 1, \forall t \geq 0\}$ and let $z_N \in \N$, $N>1$.
Then
\begin{enumerate}
    \item 
    $\displaystyle
      \sup_{t \geq 0}\Big|\frac{\log^+(Z_{t\log N})}{\log N} - \Big(\frac{\log z_N}{\log N} + st \Big)\1_{\{Z\not \to 0\}}\Big| 
        \xrightarrow{N\to\infty} 0
    $
    in probability under $\P_{z_N}$;
    \item If $t^N \to \infty$ then $\P_{z_N}(T_0 \geq t^N) \sim \P_{z_N}(Z \not \to 0) = 1-(d/b)^{z_N}$ as $N\to\infty$.
\end{enumerate}
\end{lemma}

The proof of the next lemma can be obtained as in \cite[Lemma A.1]{CMT21}.
\begin{lemma}\label{lem:gw-subcritical}
Assume $b<d$ and let $z_N \in \N$, $N>1$ with $\frac{\log z_N}{\log N} \to h \in [0,\infty)$ as $N\to\infty$.
Then
  \begin{enumerate}
  \item
    $\displaystyle
      \sup_{0\leq t\leq t_0}\Big|\frac{\log^+(Z_{t\log N})}{\log N}- (h - s t)^+\Big|
       \xrightarrow{N\to\infty} 0
    $ in probability under $\P_{z_N}$ for any $t_0>0$;
  \item
    $\frac{T_0}{\log N} \xrightarrow{N\to\infty} \frac hs$ in probability under $\P_{z_N}$.
  \end{enumerate}
\end{lemma}

Another useful property of binary Galton--Watson processes
is the following: For \mbox{$z < g\in\N$}, 
\begin{equation}
  \label{gamblersruin}
\begin{aligned}
\begin{array}{ll}
  \P_z(Z\text{ reaches }g\text{ before it reaches }0)  \leq \frac zg & \text{ when } b \leq d,\\
  \P_g(Z\text{ reaches }z\text{ in finite time}) = (d/b)^{g-z} & \text{ when } b > d.
\end{array}
\end{aligned}
\end{equation}
Indeed, the discrete-time embedding of $Z$ is a simple random walk on $\mathbb N_0$
with probability $p=b/(d+b)$ to jump to the right,
so \eqref{gamblersruin} follows from the well-known gambler's ruin formula.
For the first item, note that the subcritical case can be compared to the critical one.
See also eq.~(27) in~\cite{C06}.

% ----------------------------------------------------------------------------------------------------------
\subsection{Selective sweeps in the presence of multiple mesoscopic types}\label{sec:multitype-nomut}
Throughout this section we fix a $k \in \mathbb N$ (the number of types in the Moran model) and an $m \in \mathbb R_+^k$ (the vector of fitnesses). For $N\in\N$ let $X^N(t)=(X^N_1(t),\ldots,X^N_k(t))$, $t \ge 0$, be a process on
  $$
    \cS_N^k
     := \Big\{x= (x_1,\ldots, x_k): x_\ell \in \mathbb N_0, \sum_{\ell=1}^k x_\ell = N \Big\}
  $$
  whose generator $L$ acts on functions $f:\cS_N^k \to\R$ as
\begin{equation}\label{genL}
    Lf(x)
      = \tfrac{\log N}N\sum_{i\neq j}x_ix_j(1+(m_j-m_i)^+)(f(x+e_j-e_i)-f(x)).
 \end{equation}
 This is the generator of a Moran model with selection, with time accelerated by a factor $\log N$.
 In accordance with \eqref{eq:logplus} we put for all $t \ge 0$ and $\ell = 1, \ldots, k$
 \begin{equation}\label{defH}
   H^N_\ell(t)
     := \frac{\log^+\big(X^N_\ell(t)\big)}{\log N}.
 \end{equation}
 Throughout Section~\ref{sec:multitype-nomut} we make the following assumptions:
 \begin{equation}\label{asinh1}
 (\eta_\ell,m_\ell) \in [0,1]\times \mathbb R, \, \ell=1,\ldots, k,\, \mbox{ are pairwise distinct } \mbox{ with } \eta_1 = 1,
 \end{equation}
 \begin{equation}\label{asinh2}
 H_\ell^N(0) \to \eta_\ell \quad\mbox{in probability as } N\to \infty, \quad \ell = 1,\ldots, k.
 \end{equation}
   An important role will be played by the indicators
   \begin{equation} \label{defBN}
     B^N_\ell
      := \1_{\{X_\ell^{N}(t^N) \geq \log N\}}, \quad \mbox{ where } t^N:= \frac 1{\sqrt {\log N}}.
   \end{equation} 
    
  Throughout this section we fix a sequence $(\overline h_N)$ in $(0,1)$ with the properties
  \begin{equation}\label{condhN}
    \lim_{N\to\infty}\overline h_N
      = 1 \quad \mbox{ and } \quad (1-\overline h_N) \log N\to\infty \,\mbox{ as } N \to \infty
  \end{equation}
   and we observe that \, ${\overline g}_N:= N^{\overline h_N} - 1$ satisfies
   $\log^+(\overline g_N) = \overline h_N$ and
  \begin{equation}\label{mocogN}
    \tfrac{\overline g_N}{N} \to 0 \quad\mbox{as } N\to \infty.
  \end{equation}
For further use we will also require that
\begin{equation}\label{boundforhN}
\frac  N{\sqrt{\log N}} = o(\overline g_N) \quad \mbox{as } N\to \infty.
 \end{equation}
 (A concrete choice for $\overline h_N$ and $\overline g_N$ which satisfies \eqref{condhN} and \eqref{boundforhN} is 
 $\overline h_N := 1- \frac{\log\log N}{3\log N}$, $\overline g_N = \frac N{(\log N)^{1/3}}-1$.)

We begin with some rough linear bounds on $H^N=(H^N_\ell)_{1 \leq \ell \leq k}$, which imply an asymptotic stochastic continuity of $t\mapsto H_\ell^N(t)$ in $t=0$. 

\begin{lemma}\label{lem:linearbounds}
Assume~\eqref{asinh1}, \eqref{asinh2} and set $m_\star:=2\max_{1 \leq \ell \leq k} |m_\ell|$. Then, for any $T,\varepsilon>0$,
  \begin{equation}\label{e:linearbounds}
    \lim_{N\to \infty}
      \P \Big( \exists t \in [0,T] \colon \max_{1 \leq \ell \leq k} \big| H^N_\ell(t) - \eta_\ell\big|\geq (1+m_\star) t + \varepsilon \Big)
      = 0
  \end{equation}
  As a consequence, for any sequence $(\mathscr T_N)$ of random times converging to $0$ as $N\to \infty$,
  $$ \sup_{0\le t \le \mathscr T_N}|H_\ell^N(t)-\eta_\ell| \to 0 \mbox{ in probability as }N\to \infty.$$
\end{lemma}

\begin{proof}
By the union bound, it is enough to consider a fixed $1 \leq \ell \leq k$.
Let us compare $X^N_\ell$ with two continuous-time Galton--Watson processes $Y$ and $Z$
with individual birth/death rates $0$/$\log N
(1+m_\star)$ and $\log N(1+m_\star) $/$0$ respectively, both started from $X^N_\ell(0)$, and note that the terms $(m_j-m_1)^+$ appearing in~\eqref{genL} are all bounded by $m_\star$.
Setting $\varphi:\mathcal{S}^k_N \to \N_0$, $\varphi(x) = x_\ell$,
$E_0=\mathcal{S}^k_N$
and using \eqref{genL}, it is straightforward to verify condition \eqref{e:stochdomcond} with $A$ the generator of $X$ and $B$ the generator of $Y$ or $Z$, implying that $X^N_\ell$ can be coupled with $Y$ and $Z$ so that $Y(t) \leq X^N_\ell(t) \leq Z(t)$ for all $t \geq 0$; note that Theorem~\ref{thm:stochdom} only couples two processes, but the three processes can be coupled using regular conditional probabilities and taking $Y,Z$ e.g.\ conditionally independent given $X^N_\ell$.
The claim \eqref{e:linearbounds} then follows from Lemmas~\ref{lem:gw-supercritical}--\ref{lem:gw-subcritical}
once we note that $Y,Z$ are sped-up versions (with time sped-up by $\log N$) of processes treated therein.
\end{proof}

Via comparison of jump rates as in Lemma~\ref{lem:linearbounds}, we will show in the next lemma that, in a multitype Moran process with a single macroscopic component, all the other components are close enough to independent branching processes so that their logarithmic frequencies on the $\log N$ timescale converge to linear functions, as long as none of these components become close to macroscopic.

\begin{lemma}[Until the next resident change]\label{lem:multitype-without-kinks}
Suppose that \\
\phantom {Suppose}
(i) \,\,\,$X_1^N(0)\,/\,N \to 1$ in probability as $N\to \infty$ (and consequently $\eta_1 = 1$), \\
\phantom {Suppose}
(ii) \,\,$X_\ell^N(0)\,/\, \overline g_N \to 0$ in probability as $N\to \infty$ for $\ell \ge 2$, \\
\phantom {Suppose}
(iii) \,\,If $\eta_\ell = 1$ for some $\ell \ge 2$, then $m_\ell < m_1$. 
\\
Define 
\begin{align}
&h_\ell:t\mapsto (\eta_\ell + (m_\ell-m_1)t)^+, \quad \ell = 1,\ldots, k,\notag \\
    &\overline T^N
      := \inf\{t \geq 0\mid\max_{\ell\geq2}X_\ell^N(t)\ge \overline g_N\}
       = \inf\{t \geq 0\mid\max_{\ell\geq2}H_\ell^N(t)\ge \overline h_N\}
      \in (0,\infty],\label{defbarT}
\end{align}
where $\inf \emptyset = \min \emptyset = \infty$.
  Then the following hold.
  \begin{enumerate}
  \item[(A)] Let $\ell \ge 1$ be such that $X_\ell^N(0)\to\infty$ (which is implied by \ $0< \eta_\ell\le 1$). 
    Then for all $t_0 > 0$
  \begin{equation}\label{caseA}
    \sup_{0 \leq t \leq \overline{T}^N \wedge t_0}\big| H_\ell^N(t)
          - h_\ell(t)\big|
        \xrightarrow{N\to\infty} 0 \text{ in probability.}
   \end{equation}
   Moreover, $\sup_{0 \leq t < \overline{T}^N}\big|H_1^N(t) - 1\big|\to 0$ in probability as $N\to\infty$.
  \item[(B)]
      Let $\ell \ge 2$ be such that $m_\ell>m_1$ and $X_\ell^N(0) = 1$ for $N\geq1$, which implies $\eta_\ell=0$. Then the random variables $B^N_\ell$ defined in~\eqref{defBN} satisfy
    \begin{equation} \label{survival}
\P(B^N_\ell=1) \sim \P(X_\ell^N(t^N) >0) \to \tfrac{m_\ell-m_1}{1+m_\ell-m_1}=:\pi_\ell \,\,
\text{ as } N \to \infty    \end{equation} 
and, for each $t_0>0$,
   \vspace{-0.5em}
   \begin{equation}\label{straightline}
    \sup_{0 \leq t \leq \overline{T}^N \wedge t_0}\big| H_\ell^N(t)
          - B^N_\ell (m_\ell-m_1)t\big|
        \xrightarrow{N\to\infty} 0 \,\,\text{ in probability.}
   \end{equation}
  \item[(C)]Define
  \[\begin{array}{rcccl}
     \tau
     &:=& \min\limits_{\ell:\,m_\ell > m_1} \big\{\inf\{t\ge 0 \mid h_\ell(t)=1\}\big\}
      &=  \min\big\{\tfrac{1-\eta_{\ell}}{m_\ell-m_1}\,\big|\, \ell = 2,\ldots, k; \, m_\ell>m_1\big\},
      \\[1em]
    \tau^N
     &:=& \min\limits_{\ell:\,m_\ell > m_1} \big\{\inf\{t\ge 0 \mid B_\ell^N h_\ell(t)=1\}\big\}
      &=  \min\big\{\tfrac{1 -\eta_{\ell}}{m_\ell-m_1} \,\big|\, m_\ell>m_1 \,  \mbox{ and } B_\ell^N=1\big\},
  \end{array}\]
    where $\min\emptyset:=\infty$. Then
  \begin{equation}\label{mainassertionC}
      \big | \overline T^N - \tau^N\big | \, \mathds 1_{\{\tau^N < \infty\}}
       + \big | \mathds 1_{\{\tau^N = \infty\}}-  \mathds 1_{\{\overline T^N = \infty\}}\big |
           \xrightarrow{N\to\infty}0 \quad \mbox{in probability}.
  \end{equation}
    Moreover, if $\eta_{j}>0$ for all $j = 1,\ldots, k$, then both  $\overline{T}^N$ and $\tau^N$ converge to $\tau$ in probability.
  \end{enumerate}
\end{lemma}
\begin{proof}  
 Because of~\eqref{mocogN}, for $t\le \overline T^N$,
  \[
    1\ge \tfrac{\log(X^N_1(t))}{\log N} \ge \tfrac{\log (N-k\overline g_N)}{\log N} = \tfrac{\log(1-k\overline g_N/N)+\log N}{\log N} \to 1\quad\mbox { as } N\to \infty,
       \]
 which proves assertion (A) in case of $\ell=1$.\\
 Before proving the remaining assertions of the lemma, we make a few preparations. 
 The generator \eqref{genL} tells that, when $X^N$ is in state $x$, 
  the total rate for its $\ell$-coordinate to increase by one is $x_\ell\cdot \beta_{\ell,N}(x)\log N$ and the total rate to decrease is $x_\ell\cdot \delta_{\ell,N}(x)\log N$, where
  \begin{align*}
    \beta_{\ell,N}(x)
     &:= \frac1N\sum_{i\neq\ell}x_i\big(1+(m_\ell-m_i)^+\big)
    \quad\text{and}\quad
    \delta_{\ell,N}(x)
      := \frac1N\sum_{i\neq\ell}x_i\big(1+(m_i-m_\ell)^+\big).
  \end{align*}
To prepare for a comparison argument, we note that for all $x \in \mathcal S_N^k$ satisfying the inequality
\begin{equation}\label{mesoscopes}
  \max_{2\le j \le k} x_j < \overline g_N
  \end{equation}
  one has, (with a view on \eqref{mocogN}) for $K>0$ and $N\in \N$ large enough, the estimates
 \begin{align}\notag
   &\beta_{\ell,N}(x)  
      \ge \tfrac1N(N-k\overline g^{}_N)(1+(m_\ell-m_1)^+) \geq (1-\tfrac 1K)(1+(m_\ell-m_1)^+) =: \beta^-_{\ell,K},
    \\[.5em]\label{lowerestdelta}
  &\delta_{\ell,N}(x)
      \ge  \tfrac1N(N-k\overline g^{}_N)(1+(m_1-m_\ell)^+)
      \ge  (1-\tfrac1K)(1+(m_1-m_\ell)^+) 
      =:  \delta^-_{\ell,K},
    \\[.5em]\label{upperestbeta}&\beta_{\ell,N}(x) 
     \le 1+(m_\ell-m_1)^+ + \tfrac1Nk\overline g^{}_N(1+\max_{i\geq2}(m_\ell-m_i)^+))\\
     \notag
      &\phantom{AAAAAA}\le 1+(m_\ell-m_1)^+ + \tfrac1K(1+\max_{i\geq2}(m_\ell-m_i)^+))
     \quad=: \,\beta^+_{\ell,K},\\[-0.5em]\notag
%\intertext{and similarly,}
   &\delta_{\ell,N}(x)
     \leq 1+(m_1-m_\ell)^+  + \tfrac1Nk\overline g^{}_N(1+\max_{i\geq2}(m_i-m_\ell)^+))
     \\\notag
      &\phantom{AAAAAA}\le
      1 + (m_1-m_\ell)^+ + \tfrac1K(1+\max_{i\geq2}(m_i-m_\ell)^+))\quad=:\,\delta^+_{\ell,K}.
  \end{align}
Setting $\varphi: \mathcal{S}^k_N \to \N_0$, $\varphi(x) = x_\ell$
and \mbox{$E_0= \{x \in \mathcal{S}^k_N \colon\, x_i < \overline{g}_N~ \forall i \geq 2 \}$},
we can reason as in the proof of Lemma~\ref{lem:linearbounds} 
to define processes $Y_\ell^{K,N}$ and $Z_\ell^{K,N}$ on the  same probability space as $X^N$ such that $Y^{K,N}_\ell$, $Z^{K,N}_\ell$ are continuous-time Galton--Watson processes with individual birth/death rates 
$\beta^-_{\ell,K}$ / $\delta^+_{\ell,K}$ and 
$\beta^+_{\ell,K}$ / $\delta^-_{\ell,K}$ respectively and obey 
\begin{align}\notag
& Y^{K,N}_\ell(0) = X^N_\ell(0) = Z^{K,N}_\ell(0)\\
   \label{sandwich}
   & Y^{K,N}_\ell(t\log N) \le X^N_\ell(t) \le Z^{K,N}_\ell(t \log N), \quad   0 \leq t \leq \overline T_N. 
\end{align}
Moreover, by restarting $Y_\ell^{K,N}$, $Z_\ell^{K,N}$ at time $\overline{T}_N \log N$
and coupling them afterwards via Theorem~\ref{thm:stochdom}, 
we can make sure that  $Y_\ell^{K,N}(t) \leq Z_\ell^{K,N}(t)$ for all $t\geq 0$. 

For further use we define for $j\ge 2$
\begin{equation}\label{defhatTN}
    \overline{T}^N_j:= \inf\big\{ t \geq 0 \mid X^N_j(t) \geq \overline{g}_N \big\}.
    \end{equation}
After these preparations we turn to the proof of assertion (A).

\smallskip
Consider first the case of an $\ell \ge 2$ with $\eta_\ell = 1$ and (as enforced by assumption (iii)) \mbox{$m_\ell < m_1$}.
Then, as $N\to \infty$, the first bound in~\eqref{lowerestdelta} converges to $1$, while the first bound in~\eqref{upperestbeta} converges to $1+m_1-m_\ell$. This implies that, for $N$ large enough, and all $x \in \{y\in\mathcal S_N^k|y\text{ satisfies~\eqref{mesoscopes}}\}$, we have the inequality $\beta_{\ell,N}(x) < \delta_{\ell,N}(x)$. Consequently, for these $x$ the function $\phi(x):=x_\ell$  fulfills, with $L$ as in~\eqref{genL}, the inequality $L\phi(x) \le 0$, i.e., $\phi$ is superharmonic inside this set. This shows that $X^N_\ell(t \wedge T^N)$, $t\ge 0$, 
is a supermartingale.  Hence for all $t\ge 0$
$$
  X_\ell^N(0)
    \ge \E[X_\ell^N(t\wedge \overline T_\ell^N\wedge \overline T^N)\mid X^N(0)]
    \ge \overline g_N \P(\overline T_\ell^N \le t\wedge \overline T^N\mid X^N(0)),
$$
showing that $\P(\overline T_\ell^N \le t\wedge \overline T^N\mid X^N(0))\le \frac{X_\ell^N(0)}{\overline g_N}$ almost surely. Because of assumption (ii)  this proves that for all  $t>0$
\begin{equation}\label{negativem}
\P(\overline T_\ell^N= \overline T^N\le t) \to 0 \quad \mbox{as } N\to \infty.
\end{equation}
Next we consider the case of an $\ell>1$ with $\eta_\ell \in (0,1)$, and distinguish the three subcases
$$  \text{a) }m_\ell >m_1, \quad  \text{b) }m_\ell = m_1, \quad \text{c) } m_\ell < m_1 $$
As can be seen from their definition, the birth and death rates of the Galton--Watson processes $Y_\ell^{K,N}$ and $Z_\ell^{K,N}$ satisfy
\begin{equation}\label{ratesofYandZ}
\beta^\pm_{\ell,K} \to 1+(m_\ell-m_1)^+\, \mbox{ and } \,\,\delta^\pm_{\ell,K} \to 1+(m_1-m_\ell)^+ \quad \mbox{ as }K\to \infty. 
\end{equation}
Consequently, for $K$ large enough, in case a) both $Y_\ell^{K,N}$ and $Z_\ell^{K,N}$ are supercritical, in case~c) both $Y_\ell^{K,N}$ and $Z_\ell^{K,N}$ are subcritical, and in case~b) $Y_\ell^{K,N}$ is subcritical while $Z_\ell^{K,N}$ is supercritical.  Turning to~\eqref{caseA}, for given $t_0$, let $K$ be so large that the sup-distance of both $t\mapsto \eta_\ell +((\beta_{\ell,K}^+-\delta_{\ell,K}^-)t)^+$ and $t\mapsto \eta_\ell +((\beta_{\ell,K}^--\delta_{\ell,K}^+)t)^+$ to $h_\ell$ on $[0,t_0]$ is smaller than a given (small) threshold; this can be achieved by~\eqref{ratesofYandZ}.
Note that $Y_\ell^{K,N}$ and $Z_\ell^{K,N}$ have birth and death rates that are independent of $N$ and only their initial conditions depend on $N$. Hence, 
Lemma~\ref{lem:gw-supercritical} and Lemma~\ref{lem:gw-subcritical}
apply, and due to the assumption~\eqref{asinh2},  the logarithmic frequencies  $\frac{\log Y_\ell^{K,N}(t\log N)}{\log N}$ and $\frac{\log Z_\ell^{K,N}(t\log N)}{\log N}$, $t\in [0,t_0]$ then both are, with high probability as $N\to \infty$, close to $h_\ell$ in the sup-distance.  The claimed convergence~\eqref{caseA} then follows from the sandwiching relation~\eqref{sandwich}, thus completing the proof of assertion (A).

Consider now the assertion (B). Here, $Z_\ell^{K,N}$ and $Y_{\ell}^{K,N}$ are completely independent of $N$ (since the initial condition is always $1$), and therefore we will remove $N$ from their notation. Choosing $K$ so large that $Y_\ell^{K}$ is supercritical, let us show that, as $N\to\infty$,
\begin{equation}\label{survivalK}
\begin{aligned}
  \P_1 \big(Y_\ell^{K}(\sqrt{\log N}) \geq \log N \big) 
    \sim \P_1 \big(Y_\ell^{K}(\sqrt{\log N}) \geq 1 \big)
    \sim 1 - \frac{\delta_{\ell,K}^+}{\beta_{\ell,K}^-}.
\end{aligned}
\end{equation}
Indeed, the second equivalence follows by Lemma~\ref{lem:gw-supercritical}, 2. 
For the first, note that the first probability above is not larger than the second and not smaller than
\[
\begin{aligned}
  & \P_{1} \Big(Y_\ell^{K} \text{ reaches } 2 \log N \text{ by time } \sqrt{\log N}, Y_\ell^{K}(\sqrt{\log N})
  \geq \log N \Big) \\
  \geq \; & \P_{1} \Big(Y_\ell^{K} \text{ reaches } 2 \log N \text{ by time } \sqrt{\log N} \Big) - \big(\delta^+_{\ell, K}/\beta^-_{\ell, K} \big)^{\log N}
\end{aligned}
\]
by \eqref{gamblersruin}.
Thus parts 3b) and 3c) of Theorem~\ref{thm:superbranch} finish the proof of \eqref{survivalK}.
In order to exploit the monotone coupling~\eqref{sandwich}  up to time $t^N$ we will use the fact 
\begin{equation}\label{positiveTbar}
  \P(t^N\le \overline T^N)\to 1 \mbox{ as } N\to \infty.  
\end{equation}
(Indeed \eqref{positiveTbar} follows by combining~\eqref{negativem} with Lemma~\ref{lem:linearbounds}, where the latter takes care of those components for which $\eta_j <1$, while~\eqref{negativem} shows that $\overline T_j^N$ is stochastically bounded away from $0$ as $N\to \infty$ for all those $j\ge2$ for which $\eta_j=1$.)

Using the same arguments for $Z_\ell^{K}$ as well as \eqref{sandwich} and~\eqref{positiveTbar}
it follows that
  \[
  \begin{aligned}
    1 - \frac{\delta_{\ell,K}^+}{\beta_{\ell,K}^-}
   & \leq \liminf_{N\to\infty} \P(B^N_\ell=1) \leq \limsup_{N\to\infty} \P(B^N_\ell=1) \\
   &  \leq \limsup_{N\to\infty}\P(X^N_\ell(t^N) \geq 1) 
   \leq     1 - \frac{\delta_{\ell,K}^-}{\beta_{\ell,K}^+},
\end{aligned}
  \]
  and since both left- and right-hand sides above converge to $\pi_\ell$ as $K\to\infty$, 
  we verify \eqref{survival}. 

Now note that, since $\{Y_\ell^{K} \text{ survives}\} \subset \{Z_\ell^{K} \text{ survives}\}$
and both events have asymptotically equal probability as $K\to\infty$, 
with probability tending to $1$ either both survive or both die out. 
By \eqref{sandwich} and \eqref{survivalK} (and its analogue for $Z_\ell^{K}$),
w.h.p., in the first case $B^N_\ell=1$ and in the second case $B^N_\ell=0$.
In each case, by Lemma~\ref{lem:gw-supercritical}, $\frac{\log^+(Y_\ell^{K}(t\log N))}{\log N}$ and $\frac{\log^+(Z_\ell^{K}(t\log N))}{\log N}$ approximate, as $N\to \infty$ (in probability uniformly in $t\geq 0$) two lines which, for $K\to\infty$, converge to either $\eta_\ell +  (m_\ell-m_1)t$ or to $0$, respectively. 
Together with \eqref{sandwich} again, this shows \eqref{straightline} and completes the proof of assertion (B).

Finally we turn to assertion (C). From the assertions~\eqref{caseA} and~\eqref{straightline} we infer that
\begin{equation}\label{comp_overTN_tildeTN}
    \P(\overline{T}^N=\widetilde{T}^N) \to 1 \mbox{ as } N\to \infty .
    \end{equation}
    where (with $\overline{T}^N_j$ defined in~\eqref{defhatTN}) 
$$
  \widetilde{T}^N:= \min\{\overline{T}^N_j \mid j \geq 2, m_j> m_1, B_j^N =1\}.
$$
Thus it is enough to show~\eqref{mainassertionC} for $\widetilde{T}^N$ (in place of $\overline T^N$). Also, it suffices to consider the case in which the set $J:= \{j \ge 2: m_j > m_1\}$ is non-empty (because otherwise both $\tau^N$ and $\widetilde{T}^N$ equal $\infty$). Now $\tau^N$ is the minimum of the first hitting times of the level 1 of the processes $t \mapsto B_j^N h_j(t)$, $j \in J$, while $\widetilde T^N$ is, with high probability as $N\to \infty$, equal to the minimum of the first hitting times of the level $\overline h_N$ of the processes $H_j^N$, $j \in J$. The claim that on the events $\{\tau^N < \infty\}$ the distance of $\tau^N$ and $\widetilde{T}^N$ converges to $0$ in probability thus follows from~\eqref{caseA} and~\eqref{straightline} (note that $\overline h_N \to 1$ as $N\to \infty$, and that on the events $\{\tau_N < \infty\}$ the times $\tau_N$ are uniformly bounded by a constant $t_0$ not depending on $N$). The proof of~\eqref{mainassertionC} is thus concluded by observing that, as $N\to \infty$, the event $\{\widetilde T^N=\infty\}$ occurs with high probability if and only if $\eta_j=0$ and $B_j^N=0$ for all $j \in J$, which is precisely the case if $\tau^N = \infty$.

The last statement in (C) follows from~\eqref{mainassertionC} together with the fact that $\P(B_j^N=1)\to 1$ as $N\to \infty$, provided that $\eta_j > 0$.
\end{proof}
While Lemma~\ref{lem:multitype-without-kinks} takes care of the phases between resident change times,  the next lemma treats the (short) ``competition phases'' around the resident change times. More specifically, Lemma~\ref{lem:multitype-sweep} will consider the case where one macroscopic component gets invaded by a fitter component starting from `almost' macroscopic size, while all
other components are mesoscopic. We will show that the time it takes until the first component becomes mesoscopic and the invading component becomes macroscopic is asymptotically negligible on the 
$\log N$-timescale and leaves the remaining $\log$-scaled mesoscopic type sizes asymptotically unchanged. Together with Lemma~\ref{lem:multitype-without-kinks} this reflects the well-known fact that the time required for a single advantageous mutation to go from a small fraction of a population to a big fraction close to one is negligible compared to the time which the mutant's offspring needs to reach a small fraction of the population.

\begin{lemma}[Change of resident]\label{lem:multitype-sweep}
{Let $X^N$ be an $\mathcal S^k_N$-valued process with generator~\eqref{genL}, with the sequence of (random) initial conditions $X_i^N(0)$, $i=1,\ldots,k$, and the vector $m\in \R_+^k$ of fitnesses satisfying \eqref{asinh1} and \eqref{asinh2}.}
Suppose that, for some $\ell_\star \in \{2,\ldots, k\}$,
\begin{enumerate}[\quad(a)]
    \item ${X^N_1(0)} \sim N$ (and consequently ${X^N_\ell(0)} = o(N)$ for $\ell =2,\ldots, k$) as $N\to \infty$,%\\[-.75em]
    \item $\eta_{\ell_\star}=1$ and $m_{\ell_\star}> m_1$,%\\[-.75em]
    \item $\max \big\{m_\ell: \ell \notin \{1,\ell_\star\} \text{ and } \eta_\ell = 1 \big\} < m_{\ell_\star}$ (where $\max \emptyset = -\infty$).
\end{enumerate}
With $\varepsilon_N := \frac 1{\sqrt {\log N}}$ 
define
\begin{equation}
\label{defcurlyT}
\mathscr{T}^N
     := \inf \big\{t\geq 0\mid X^N_{\ell_\star}(t)
     \geq N(1-\varepsilon_N) \big\}.
\vspace{-0.5em}
\end{equation}
  Then the following holds:
  \begin{enumerate}[\qquad(1)]
  \item
    $\mathscr{T}^N\to0$ in probability;
  \item
    $\sup_{t\in[0,\mathscr T^N]}|H^N_\ell(t)-\eta_{\ell}|\to0$
    in probability for all $\ell\in\{1,\ldots,k\}$.
  \end{enumerate}
\end{lemma}

\begin{proof} 

Without loss of generality we may assume that $\ell_\star = 2$.
Let us first show (1) in the case $k=2$, in which $X^N_2$ is Markovian.
Let $Y$ be a continuous-time Galton--Watson process started from $Y(0)=X^N_2(0)$ with individual birth and death rates $(1+m_2-m_1)$ and $1$ respectively, and set $c_2:=m_2-m_1$. 
Define $$\tau^N_0:=\inf\{t \geq 0 \colon\, Y(t \log N)\geq \varepsilon_NN\}, \quad \tau^N:= \inf\{t \geq 0 \colon\, Y(t \log N)\geq N(1-\varepsilon_N)\}.$$
Since $\tau_0^N \leq \tau^N \leq \inf\{ t\geq 0 \colon\, Y(t\log N) \geq N\}$, 
Theorem~\ref{thm:superbranch}\,3c)
implies that, with high probability, $\tau_0^N \leq \tfrac{2}{c_2}[1-\log(X^N_2(0))/\log N]$.
Moreover, since $\tau^N-\tau_0^N$ under the probability $\P_{X^N_2(0)}(\cdot \mid \tau_0^N < \infty)$ has the same distribution as $\tau^N$ under $\P_{\lceil \varepsilon_N N\rceil}$, we see that $\tau^N-\tau^N_0 \leq \frac{2}{c_2}\log (1/ \varepsilon_N )/ \log N$ w.h.p.\ under $\P_{\lceil  \varepsilon_N N\rceil}$. 

Define $\phi:\{0,\ldots,N\} \to (0,\infty)$ by $\phi(x)=1-(x \wedge \lceil N(1-\varepsilon_N)\rceil)/N$ and introduce the time-change
\[
S_t:= \int_0^t \frac{1}{\phi(Y_u)} du.
\]
Note that $\phi(x) \geq \varepsilon_N-N^{-1}\geq \tfrac12 \varepsilon_N$ for large $N$ so that $S_t$ is continuous, strictly increasing and $\lim_{t\to\infty} S_t = \infty$.
The relation between the generators of $X^N_2$ and $Y$ shows that
$X^N_2(t \wedge \mathscr{T}^N)_{t \geq 0}$ has the same distribution as $(Y_{\sigma_t \wedge \tau^N})_{t \geq 0}$
where $\sigma_t$ is the inverse of $S_t$
(see e.g.\ \cite[Section~6.1]{ethier2009markov}).
In particular, $\mathscr{T}^N$ is equal in distribution to $S_{\tau^N}$.
Since $\phi(x) \geq 1/2$ for $x \leq \varepsilon_N N$, and since $\eta_2 = 1$ by assumption, we conclude that
\begin{equation*}
\begin{aligned}
S_{\tau^N}  = \int_0^{\tau_0^N} \frac{1}{\phi(Y_u)} du +\int_{\tau_0^N}^{\tau^N} \frac{1}{\phi(Y_u)} d u
& \leq 2 \tau_0^N + 2 \frac{1}{\varepsilon_N} (\tau^N-\tau_0^N) \\
& \leq \frac{4}{c_2} \Bigg[1-\frac{\log {X^N_2(0)}}{\log N} \Bigg]
+ \frac{4}{c_2}\frac{\log(1/\varepsilon_N)}{\varepsilon_N \log N}
\to 0
\end{aligned}
\end{equation*}
in probability as $N\to \infty$.
 This shows (1) in the case $k=2$.

 For general $k$, we will first apply Lemma~\ref{lem:linearbounds} to deal with the coordinates $\ell \geq 3$ where $m_\ell \geq m_2$ (so $\eta_\ell<1$). 
Let $J := \{3 \leq \ell \leq k \colon\, m_\ell \geq m_2\}$, 
define $\varphi:\mathcal{S}^k_N \to \{0,1,\ldots, N\}$
by $\varphi(x_1, \ldots, x_k) = \sum_{j \in J} x_j$
and set $Y^N(t) := \varphi(X^N(t))$.
We will compare $Y^N$ to the second coordinate of a bivariate process
$\widehat{X}^{N}=(\widehat X^{N}_1,\widehat X^{N}_2)$ with generator as in \eqref{genL},  and $\widehat m := (\min\{m_\ell \mid \ell \notin J \}, \max\{m_\ell \mid \ell \in J\})$ in place of $m$.
Note that $\widehat{X}^N_{2} $  is Markovian.
We start $\widehat{X}^N$ from $\widehat X^{(N)}(0) = (N-Y^N(0), Y^N(0))$.
It is straightforward to verify \eqref{e:stochdomcond} with $A$ the generator of $X^N$, $E=\mathcal{S}^k_N$ and 
$B$ the generator of $\widehat{X}^N_2$, so Theorem~\ref{thm:stochdom} gives a coupling such that $Y^N$ is smaller than $\widehat{X}^N_2$ for all times.
Now, the sequence $\widehat X^{N}(0)$ satisfies 
$$\lim_{N\to \infty} \frac{\log(\widehat X_i^N(0))}{\log N} =: \widehat \eta_i, \quad i=1,2$$
with $\widehat \eta_2 < 1$. 
By Lemma~\ref{lem:linearbounds}, there exist $\eta^\ast<1$ and $t^\ast>0$ such that $Y^N(t)\leq N^{\eta^\ast}$ for all $t\in[0,t^\ast]$ with high probability.

Next, we will deal with the coordinates where $m_\ell<m_2$ by reducing to the case $k=2$.
Let $\widetilde{J}:= \{1 \leq \ell \leq k \colon\, m_\ell < m_2\} = J^c \setminus \{2\}$,
define 
$\widetilde{\varphi}(x) = \sum_{k \in \widetilde{J}} x_k$ and 
set $$\widetilde{Y}^N(t) := \widetilde{\varphi}(X^N(t))=N-Y^N(t)-X_2^N(t).$$
Using Theorem~\ref{thm:stochdom}, we can couple $\widetilde{Y}^N$ with the first coordinate of a bivariate process $\widetilde{X}^{N}=(\widetilde{X}^{N}_1,\widetilde{X}^{N}_2)$ with generator as in \eqref{genL} 
and $\widetilde{m} := (\max\{m_\ell\mid\ell \in \widetilde{J}\}, m_2)$ in place of $m$,
started from $ \widetilde{X}^{N}(0)= (\widetilde{Y}^N(0), N-\widetilde{Y}^N(0))$,
in such a way that $\widetilde{Y}^N$ is smaller than $\widetilde{X}^N_1$ for all times.
Since the conclusion of the case $k=2$ of the present lemma applies for $\widetilde{X}^N$ also with  $\varepsilon_N$ substituted by $\tfrac12\varepsilon_N$,
we infer that $\widetilde{\mathscr{T}}^N:= \inf\{t \geq 0\colon \, X^N_2(t) + Y^N(t) \geq N(1-\tfrac12\varepsilon_N)\}$
converges to zero in probability as $N\to\infty$. 
Finally, since $X^N_2(\widetilde{\mathscr{T}}^N) \geq N(1-\tfrac12\varepsilon_N)-N^{\eta^*} \geq N(1-\varepsilon_N)$
for large $N$, we conclude that $\mathscr{T}^N\leq \widetilde{\mathscr{T}}^N$ w.h.p. This finishes the proof of (1).
Now (2) follows from (1) and Lemma~\ref{lem:linearbounds}.
\end{proof}

\begin{remark}[The final state from Lemma~\ref{lem:multitype-sweep} as initial state for   Lemma~\ref{lem:multitype-without-kinks}]\label{concat} 
Recalling the definition of $\mathscr{T}^N$ and $\varepsilon_N$ as well as the assumption on $\ell_\star$ in Lemma~\ref{lem:multitype-sweep}, note that 
\begin{equation}\label{largetype_Renato}
X_{\ell_\star}^N(\mathscr T^N) \geq N\left(1-\frac{1}{\sqrt{\log N}}\right),
\end{equation}
and in particular $X^N_{\ell_\star}(\mathscr{T}^N) \sim N$ in probability as $N \to \infty$.
Consequently, 
\begin{equation}\label{smalltypes_Renato}
\max_{\ell \neq \ell_\star} X_\ell^N(\mathscr T^N) \le \frac{N}{\sqrt{\log N}}.
\end{equation}
Thus, thanks to the assumptions (a)-(c) of Lemma~\ref{lem:multitype-sweep}, the {reordered} family sizes 
$$
\begin{aligned}
  \big( & \widetilde X^N_1(0), \widetilde X^N_2(0),  \widetilde X^N_3(0), \ldots, \widetilde X^N_k(0)\big) \\ & \qquad
   := \big(X^N_{\ell_\ast}(\mathscr T^N), X^N_1(\mathscr T^N), \ldots, X^N_{\ell_\ast-1}(\mathscr T^N), X^N_{\ell_\ast+1}(\mathscr T^N), \ldots,  X^N_k(\mathscr T^N)\big)
\end{aligned}
$$
  obey the conditions (i)-(iii) required for an initial state in Lemma~\ref{lem:multitype-without-kinks}. To see this, note that~\eqref{largetype_Renato} implies that $\widetilde X_1^N(0)$ fulfills condition (i), while~\eqref{smalltypes_Renato} together with~\eqref{boundforhN} implies that $\widetilde X_2(0), \ldots, \widetilde X_k^N(0)$ obey condition (ii). Finally, assumption (c) together with assertion~(2) of Lemma~\ref{lem:multitype-sweep} directly translate into condition (iii) of Lemma~\ref{lem:multitype-without-kinks}.
\end{remark}

Our next goal is to finish the analysis in the case of finitely many types, i.e., 
to show convergence of the (rescaled heights of the) Moran model with generator \eqref{genL} 
to a corresponding system of interacting trajectories, which in this case stabilizes in finite time. 
To this end, we will string together consecutive applications of Lemmas~\ref{lem:multitype-without-kinks} and~\ref{lem:multitype-sweep}, dealing respectively with the (macroscopic) stretches of time where the resident is fixed, and the (mesoscopic) stretches of time where the resident changes.
In addition, we will allow for a ``stop and restart'' at the arrival times of new mutants. 
Thanks to the coupling of the mutant arrivals in the prelimiting Moran systems via the times $T_i$ (see Remark~\ref{repres}), at the time of a new mutation we are with high probability faced with the situation at which precisely one clonal subpopulation has size $1$, while all the other clonal subpopulations that are alive at this time have limiting non-zero logarithmic frequencies in probability as $N\to \infty$ 
due to Lemma~\ref{lem:multitype-without-kinks}.
On the other hand, at a resident change time  the sizes of {\em all}  the subpopulations that are currently alive will increase to $\infty$ in probability as $N\to \infty$.  Thus, the ``stop and restart'' at a mutant arrival time will reflect in the assumption that $X_k^N(0) = 1$ for all $N$ in Proposition~\ref{fixedk}, while the restart just after a resident change time will correspond to the assumption that $X_k^N(0) \to \infty$ in probability as $N\to \infty$. These two cases then correspond to the conditions (B) and (A) in the assumptions of Lemma~\ref{lem:multitype-without-kinks}.

For the rest of this subsection, let the initial states in addition to \eqref{asinh1} and
\eqref{asinh2} obey
\begin{enumerate}
    \item[(C1)] $X^N_1(0) \geq N(1-1/\sqrt{\log N})$ \, (so $\eta_1=1$ and $X^N_\ell(0) \leq N/\sqrt{\log N}$ for $\ell =2, \ldots, k$)\,;
    \item[(C2)] $\eta_2,\ldots, \eta_{k-1} \in (0,1]$, and if $\eta_\ell = 1$ for some $\ell\in \{2,\ldots, k-1\}$, then $m_\ell < m_1$ ;
    \item[(C3)] $\eta_k=0$ and  $m_k > m_1$.
\end{enumerate}
We also recall the definition of $t^N$ and $B_\ell^N$ from~\eqref{defBN}.
Using the terminology introduced in Section~\ref{sec-IT}, but now with $\{1,\ldots , k\}$ instead of 
$\{-k+1,\ldots 0\}$ as the index set of $\aleph$, let
$$\widehat H^N:=  \Big((\widehat H_\ell(t))_{t\ge 0}\Big)_{1\le \ell \le k}$$
be  the system of interacting trajectories with starting configuration   
\[
\aleph := \big((1,0), (\eta_2,m_2-m_1), \ldots , (0,(m_k-m_1)B^N_k)\big)
\] 
(and $\beth := \emptyset$, i.e.\ no mutation arriving after time~0). Let  $\nu\, (\ge 0)$ denote the number of resident changes in $\widehat H^N$, and let $\tau_1<\cdots < \tau_\nu$ be the times of these resident changes.
Note that, because of their dependence on $B^N_k$, the quantities $\nu$ and $\tau_i$ are random variables, which also depend on $N$. For the sake of readability, we suppress this dependence in our notation.
Note further that, while $\nu$ is random, it can only take one of two integer values, one for each case $B^N_k=0$ or $1$. In particular, $\nu$ is almost surely bounded (with a deterministic bound that depends on the parameters).

\begin{prop}\label{fixedk}
Assume conditions \eqref{asinh1}, 
\eqref{asinh2} and (C1)--(C3) as above. Then
 \begin{equation}\label{unifconv} 
 \sup_{1\le \ell\le k} \sup_{0 \leq t \leq t_0}|H^N_\ell(t)-\widehat H_\ell^N(t)|
 \xrightarrow{N\to\infty}0 \quad \mbox{in probability for all } t_0 >0
 \end{equation}
  then
\vspace{-0.5em}
if $X_k^N(0) \to \infty$  in probability as $N\to \infty$
\begin{equation}\label{largeBN}
\lim_{N\to\infty}\P(B^N_k = 1) = 
1,
\end{equation}
whereas if  $X_k^N(0) = 1$ for all $N$, then
\begin{equation}\label{firstBN}
\lim_{N\to\infty}\P(B^N_k = 1) = 
\frac{m_k-m_1}{1+m_k-m_1},
\end{equation}
Moreover, for each $N$ there exist two sequences of random times $\tau^N_i, \sigma^N_i$, $0 \leq i \leq \nu$
satisfying $\tau^N_{i-1} \leq \sigma^N_{i-1} \leq \tau^N_{i}$ almost surely for all $1 \leq i \leq \nu$ and, with high probability,
\[
0 = \tau^N_0 = \sigma^N_0 <\tau_1^N < \sigma_1^N  < \tau_2^N < \sigma_2^N< \cdots <  \tau_{\nu}^N< \sigma_{\nu}^N  < \tau_{\nu+1}^N := \infty
\]
such that, as $N\to\infty$,
 \begin{enumerate}
 \item 
 $\max_{1 \leq i \leq \nu} |\tau_i^N - \tau_i| \to 0$ in probability, \\[-.5em]
  \item
    $\max_{1 \leq i \leq \nu}(\sigma^N_i - \tau_i^N)\1_{\{\tau^N_i<\infty\}} \to0$ in probability,\\[-.5em]
    \item 
    $\displaystyle \P\Bigg(\min_{0 \leq i \leq \nu}\inf_{\sigma_{i}^N < t < \tau_{i+1}^N} X_{\rho_i^{}}^N(t) \ge N\Big(1-\frac {2\overline g_N}{N}\Big)  \Bigg) \to 1$,\\
  where
    $\rho_i=\rho_i^N$ denotes the index of the resident in~$\widehat H^N$ during the time interval $[\tau_{i-1}, \tau_i)$, \, $i=1,\ldots, \nu$.
  \end{enumerate}
  \end{prop}

\begin{figure}[htb]  \includegraphics[width=6.967cm,trim={0, 1cm, 0, 2cm},clip]{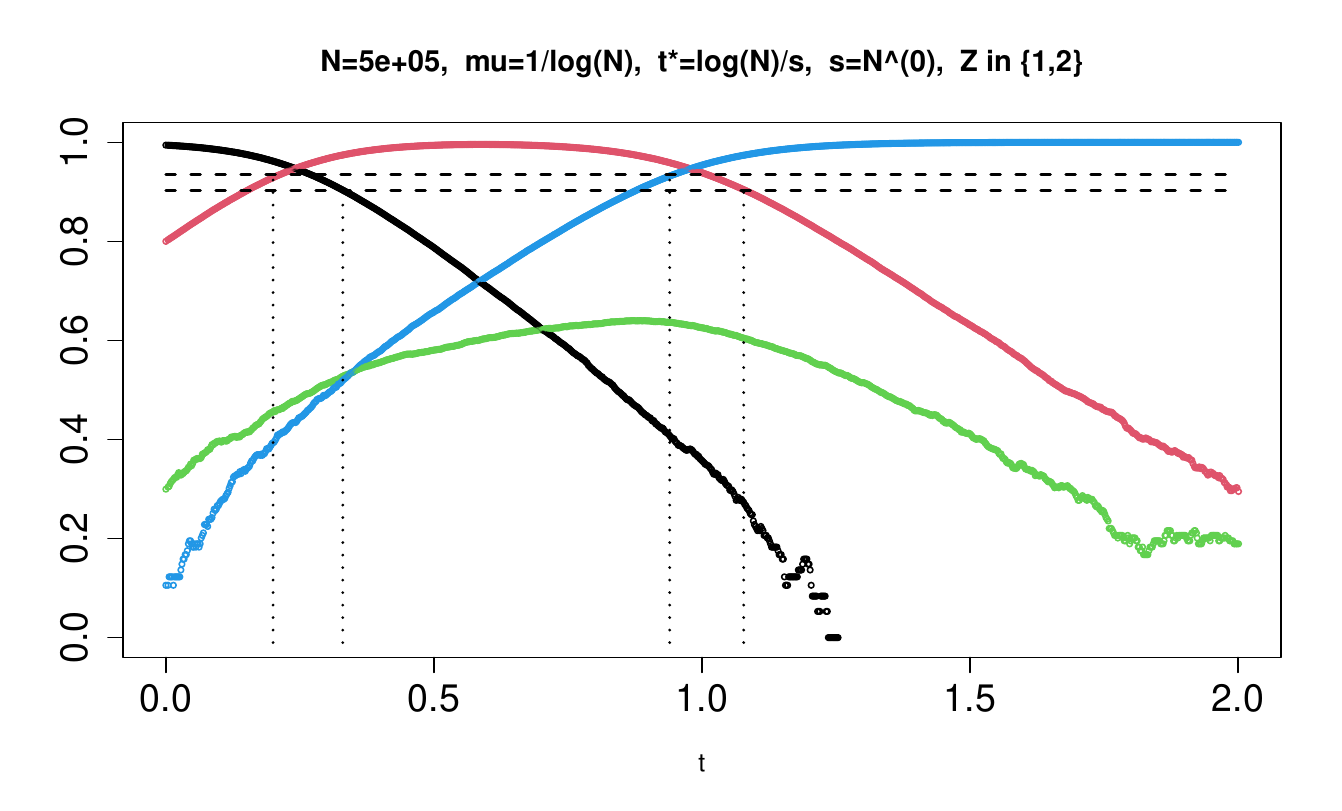}
  \includegraphics[width=6.967cm,trim={0, 1cm, 0, 2cm},clip]{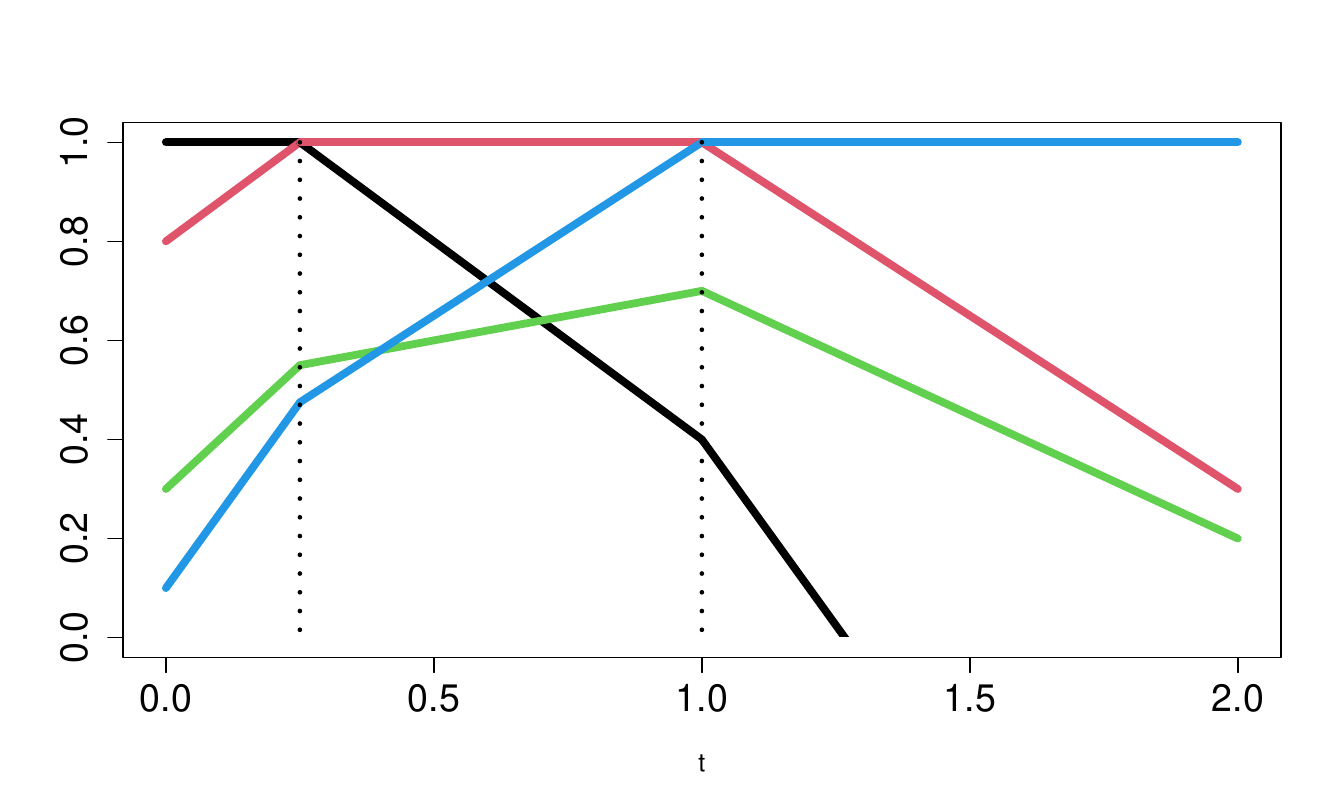}
\caption{\label{fig:THEproof-illustration}
  Illustration of Proposition~\ref{fixedk}.
  Left: Moran model, $N=500\,000$, without mutations, with initial state
    $(X^N_1(0),\ldots,X^N_4(0))
      = (N-\lfloor N^{0.8}\rfloor-\lfloor N^{0.3}\rfloor-\lfloor N^{0.1}\rfloor,
          \lfloor N^{0.8}\rfloor,\lfloor N^{0.3}\rfloor,\lfloor N^{0.1}\rfloor)$
  and $(m_1,\ldots,m_4)=(0,0.8,1,1.5)$; horizontals: $1-\frac{\log\log N}{2\log N}$ resp.\ $\overline h_N=1-\frac{\log\log N}{3\log N}$;
  verticals mark $\tau_1^N, \sigma_1^N, \tau_2^N, \sigma_2^N$ from left to right.
  Right: Corresponding PIT $\mathbb H(((0.1,1.5),(0.3,1),(0.8,0.8),(1,0)),\emptyset)$; dotted verticals give $\tau_1=0.25$ and $\tau_2=1$ respectively.
}
\end{figure}
  
  \begin{proof}
  Consider first the cases where either $X_k^N(0) =0$ for all $N$ or $X_k^N(0) \to \infty$ as $N\to \infty$. 
  In the first case $B_k^N=0$ deterministically, while in the second case a comparison with the branching process $Y_k^K$ defined in the proof of Lemma~\ref{lem:multitype-without-kinks} shows~\eqref{largeBN}. 
  Thus we may and will assume that, in these cases, $B^N_k$ is deterministically substituted by $0$ or $1$ in the definition of $\widehat{H}^N$.

In particular, $\nu$ is deterministic, and we may verify the assertions 1., 2. and 3. for each $1 \leq i \leq \nu$ separately. Let us prove the lemma in these cases by induction in $\nu$.  
With a view on~\eqref{defbarT}, define
\[ \tau_1^N
      := \inf\big\{t \geq 0\big| \max_{\ell\geq2}H_\ell^N(t)\geq \overline h_N\big\}.
\vspace{-0.5em}
\]
Part (C) of Lemma~\ref{lem:multitype-without-kinks} yields
  \begin{equation}\label{e:pr_lemfixedk_1}
  | \tau_1^N -\tau_1^{} | \1_{\{\tau_1^{} < \infty\}} \to 0 
  \ \text{ and } \ 
  | \1_{\{\tau_1^N = \infty\}}-\1_{\{\tau_1^{} = \infty\}} |
  \xrightarrow{N\to\infty} 0
  \quad \mbox{ in probability},
  \end{equation}
while parts (A) and (B) yield
\begin{equation}\label{e:pr_lemfixedk_2}
 \max_{1\le \ell\le k} \sup_{0\le t \leq \tau_1^N \wedge t_0}|H^N_\ell(t)-\widehat H_\ell^N(t)|
 \xrightarrow{N\to\infty}0 
 \quad \mbox{in probability},   
\end{equation}
which is the claimed convergence~\eqref{unifconv} restricted to $[0,\tau_1^N]$. 
This verifies the case $\nu=0$ since then $\tau^N_1 = \tau_1=\infty$ w.h.p.
Note that, under our assumptions, 
$\tau_1 < \infty$ exactly when $m_\ell>m_1$ for some $\ell = 2, \ldots, k-1$ 
or $B^N_k=1$; in particular, $\nu=0$ is not possible when $X^N_k(0) \to \infty$.

Assume thus that the statement is true for some $\nu_0 \geq 0$,
and let $\nu = \nu_0+1$. Then both $\tau_1$ and $\tau^N_1$  are finite w.h.p.,
and \eqref{unifconv} restricted to $[0,\tau^N_1]$ as well as claims 1. and 3. of the lemma for $i=1$ follow by \eqref{e:pr_lemfixedk_1}--\eqref{e:pr_lemfixedk_2} (note that $\rho_1=1$). 

Next we are going to define $\sigma_1^N$ on $\{\tau_1^N < \infty\}$.  
Thanks to the pairwise distinctness condition~\eqref{asinh1} and parts (A) and (B) of Lemma~\ref{lem:multitype-without-kinks}, in this case w.h.p. there are no two different types $\ell, \ell' \in \{2,\ldots, k\}$  with $H^N_\ell(\tau_1^N) \ge \overline h_N$, $H^N_{\ell'}(\tau_1^N) \ge \overline h_N$ and $m_\ell = m_{\ell'}$.  
This guarantees that
the assumptions of Lemma~\ref{lem:multitype-sweep} are satisfied with $X^N(\tau_1^N)$ in place of $X^N(0)$, i.e., with the time origin shifted to $\tau_1^N$.
Denote by $\ell_\star \in \{2,\ldots, k\}$ the (w.h.p.)  unique index for which $H^N_{\ell_\star}(\tau^N_1) \geq \overline{h}_N$  and $m_\ell<m_{\ell_\star}$ for any $\ell \neq \ell_\star$ such that $H^N_{\ell}(\tau^N_1) \geq \overline{h}_N$. With a view on \eqref{defcurlyT}, we define
      $$
      \sigma_1^N
      := \inf\big\{t \geq \tau^N_1 \big| X^N_{\ell_\star}(t) \geq N(1-\tfrac 1{\sqrt{\log N}})\big\}. 
      $$
  By the strong Markov property, we can apply Lemma~\ref{lem:multitype-sweep} with initial condition $X^N(\tau_1^N)$, obtaining
   \begin{equation}\label{firstclose}
   (\sigma^N_1 - \tau_1^N)\1_{\{\tau^N_1<\infty\}} \xrightarrow{N\to\infty} 0 \quad \mbox {in probability}
   \end{equation}
   which is the claimed assertion 2. for $i=1$, and also
 \begin{equation}\label{firstkink}
 \max_{1\le \ell \le k} \sup_{t \in [\tau^N_1, \sigma^N_1]}\big | H_\ell^N(t)- H_\ell^N(\tau_1^N)\big | \1_{\{\tau_1^N < \infty\}} 
 \xrightarrow{N\to\infty} 0 \quad \mbox{in probability.}
 \end{equation}
  Swapping the indices $1$ and $\ell_\star$, we obtain a new process $\widetilde{X}^N$ such that
  $\widetilde{X}^N_1 = X^N_{\ell_\star}$ and to which we can apply our induction hypothesis after shifting time by 
  $\sigma^N_1$,  yielding \eqref{unifconv} for $t \geq \sigma^N_1$ as well as further ordered random times $\tau^N_i, \sigma^N_i$ and assertions 1., 2. and 3. for $2 \leq i \leq \nu$.
  This finishes the induction step and the proof in the cases where either $X_k^N(0) =0$ for all $N$ or  $X_k^N(0) \to \infty$ as $N\to \infty$.

Consider now the case in which $X_k^N(0) = 1$ for all $N$, and recall that $t^N = 1/\sqrt{\log N}$.
Note that, with high probability, $t^N<\tau^N_1$ and, by Lemma~\ref{lem:multitype-without-kinks}(B), either $X^N_k(t^N)\geq \log N$ or $X^N_k(t^N)=0$, corresponding to $B^N_k=1$ or $B^N_k=0$. Also, the claimed convergence \eqref{firstBN} follows from part (B) of Lemma~\ref{lem:multitype-without-kinks}.
Lemma~\ref{lem:linearbounds} shows that
\[
\sup_{1 \leq \ell \leq k} \sup_{0 \leq t \leq t^N}|H^N_\ell(t) - \eta_\ell | \xrightarrow{N\to\infty} 0 \text{ in probability.}
\]
Since $B^N_k$ is measurable with respect to $(X^N(t))_{t \leq t^N}$,
we may apply the Markov property at time $t^N$ and use the proposition in one of the previously treated cases $X^N_k=0$ for all $N$ or $X^N_k \to \infty$ as $N\to\infty$ for the remaining time. This concludes the proof.
  \end{proof}
Proposition~\ref{fixedk} is illustrated by Figure~\ref{fig:THEproof-illustration}.

 The following asymptotic description
 of a selective sweep in the 2-type Moran model
under logarithmic scaling is a straightforward consequence of Proposition~\ref{fixedk} with $k=2$. 
\begin{cor}[Scaled sweep with two types]\label{cor:full-sweep}
  For $N\in\N$, $m_1 :=0$, $m_2:= s>0$, let $X^N=(X_1^N, X_2^N) $ be the Markov process on $\cS_N^2$ started at
  $X^N(0)=(N-1,1)$ with generator \eqref{genL} (where $k:=2$).
  Again, let $H^N_\ell$ be defined by \eqref{defH}, and let $$  h_1(t)
      :=        ((2-st)\wedge 1)^+
      = \begin{cases}
          1      & \text{if }t\in[0,\frac1s),\\
          1 - s(t-\frac1s) & \text{if }t\in[\frac1s,\frac2s),\\
          0      & \text{if }t\geq\frac2s,
        \end{cases}\quad \quad  h_2(t)=st\wedge1.
  $$
  Then, there is a sequence of events $E_N$ with probabilities tending to $\frac s{1+s}$ as $N\to \infty$ such that the
  following convergences hold in probability, uniformly in $t$ in compact subsets of $[0,\infty)$, 
  \begin{enumerate}
   \item
    $\displaystyle
      \big|H_1^N(t)-(\1_{E_N^c}+\1_{E_N}h_1(t))\big|
        \xrightarrow{N\to\infty} 0.
    $
  \item
    $\displaystyle
     \big|H_2^N(t)-\1_{E_N}h_2(t)\big|
        \xrightarrow{N\to\infty} 0.
    $
  \end{enumerate}
\end{cor}

%%%%%%%%%%%%%%%%%%%%%%%
\subsection{Adding one new mutation}\label{sec:newmut}   Let $X_1^N, \ldots, X_k^N$ be as in Proposition~\ref{fixedk}. Let $T$ be Exp$(\lambda)$-distributed and $A$ have distribution $\gamma$. Assume that $T$ and $A$ are independent of each other and of everything else. At time $T$, choose an individual uniformly at random from the Moran$(N)$-population and add the value~$A$ to its fitness.  Denoting the index of the family of the randomly picked individual by $\Lambda_N$ and assigning the index $k+1$ to a new family founded by this individual, we thus have a process $\widetilde X^N$ which up to time $T-$ coincides with $X^N$ and whose state at time $T$ is defined as
$$(\widetilde X_1^N(T),\ldots, \widetilde X_{\Lambda_N}^N(T), \ldots,\widetilde X_k^N(T)):= (X_1^N(T-),\ldots,X_{\Lambda_N}^N(T-)-1 , \ldots, X_k^N(T-)),$$ $$\widetilde X_{k+1}^N(T):= 1, \quad
M^N_{k+1} := m_{\Lambda_N}^{}+A.$$
For $t \ge T$, let $\widetilde X^N$ follow the dynamics \eqref{genL}, with $k+1$ in place of $k$. For  convenience we extend  $\widetilde{X}^N_{k+1}$ to the entire positive time axis by setting it to be $0$ for $t < T$. Let $\widetilde H^N$ be the process of logarithmic type frequencies of $\widetilde X^N$ defined as in \eqref{defH}.
 Let 
\begin{equation}
\widetilde B^N:=\1_{\{\widetilde X_{k+1}^{N}(T+t^N) \geq \log N\}}
\quad \text{ where } \quad
t^N = \tfrac{1}{\sqrt{\log N}},
\end{equation}
and re-define the system $\widehat H^N$ from Proposition~\ref{fixedk} by adding a trajectory $\widehat H_{k+1}^N$ that is $0$ for $t\le T$,  starts at time $T$ at height $0$ with slope $A\widetilde B^N$, and then interacts with the other trajectories of $\widehat H^N$ in the way described in Section~\ref{PIT}. 
Let 
\begin{equation}
\label{defrhoN}
\rho^N(T):= {\rm argmax}\{m_\ell\mid   1\le \ell \le k \mbox{ with } \widehat H^N_\ell(T) = 1 \},
\end{equation}
i.e.\ the index of the family which is resident at time $T$ in the PIT $\widehat H^N$.
\begin{lemma}\label{newbornmut}
We have as $N\to \infty$
 \begin{equation}\label{nextmalthusian}
  \P(M_{k+1}^N= m_{\rho^N(T)} + A) \quad \to 1,   
 \end{equation}
 \begin{equation}
 \label{convBN1}
 \P\big(\widetilde B^N=1 \mid A\big) \to \frac A{1+A} \quad \mbox{a.s.}
 \end{equation}
 \begin{equation}\label{unifconvnew} 
 \sup_{1\le \ell\le {k+1}} \sup_{0 \leq t \leq t_0}|\widetilde H^N_\ell(t)-\widehat H_\ell^N(t)|\to0 \quad \mbox{in probability for all } t_0>0.
 \end{equation}  
\end{lemma}
\begin{proof}
According to the statements 2. and 3. in Proposition~\ref{fixedk} we have
\begin{equation}\label{convLambda}
\P(\Lambda_N = \rho^N(T)) \to 1 \quad \mbox{ as } N\to \infty.
\end{equation}
(Recall that $\Lambda_N$ is the type of the mutant individual at time $T$ prior to its mutation,  and $\rho^N(T)$ is defined by~\eqref{defrhoN}.)
The  convergence \eqref{nextmalthusian} thus follows from the definition of $M^N_{k+1}$, and the  convergence~\eqref{convBN1} follows from~\eqref{firstBN}. Finally, \eqref{unifconvnew} follows from a twofold application of Proposition~\ref{fixedk}, first by restricting \eqref{unifconv} to $[0,T]$ and then by applying Proposition~\ref{fixedk} on the interval $[T,\infty)$ to $\widetilde X^N$ now with $k+1$ instead of $k$ types, and with the above described initial states   $\widetilde X^N(T)$.
\end{proof}
\subsection{Completion of the proof of Theorem~\ref{theorem-THE}}\label{sec:completion}  We now revert to the definition of $(\mathscr X^N, \mathscr M^N, \mathscr I^N)$ as in Remark~\ref{repres} in Section~\ref{sec-model} . Let $H^N$ be as in \eqref{defH}.
We define

\[
 B^N_i
  := \1_{\{\mathscr X^N_i(T_i+t^N) \geq \log N \}}
  \quad \text{ where } \quad
  t^N= \tfrac{1}{\sqrt{\log N}}.
\]
Let $\widehat H^N$ be the PIT with initial state $((1,0),(0,0),(0,0),\ldots)$, and with new trajectories born at times $T_i$ with initial slope $A_i B^N_i$.
For  $i=1,2,\ldots$, let
$\rho^N(T_i)$ be the type that is resident in $\widehat H^N$ at time~$T_i$, i.e. that index $J < i$ for which $\widehat H_J^N(T_i)=1$.
(Note that by construction $\rho^N(T_i)$ is a.s.\ well-defined.)
We define recursively
\begin{equation}
    \label{updateMN}
    \widehat M_i^N := \widehat M^N_{\rho^N(T_i)}+A_i, \qquad \widehat M^N_0 := 0.
\end{equation}
We now state a ``quenched'' version of 
Theorem~\ref{theorem-THE}.
\begin{prop}\label{convprob}
Conditionally given $(T_i, A_i)_{i \in \N}$, for all $i=1,2,\ldots$,
\begin{equation}\label{convprobHN}
\sup_{0\le \ell< i} \sup_{0\le t\le T_i}|H^N_\ell(t)-\widehat H_\ell^N(t)|\to0 \quad \mbox{in probability} \quad \mbox{as } N\to \infty
\end{equation} and
\begin{equation}
\label{convprobMN}
    \P((M_0^N,\ldots, M_i^N) = (\widehat M_0^N,\ldots, \widehat M_i^N)) \to 1 \qquad \mbox{as } N\to \infty. 
    \end{equation}
\end{prop}
\begin{proof}
    This follows from Lemma~\ref{newbornmut} by induction over $i$.
\end{proof}
For all $i=1,2,\ldots$, let $B_i$ be mixed Bernoulli with random parameter $\frac {A_i}{1+A_i}$, i.e.
$$\P(B_i =1\mid A_i) = \frac {A_i}{1+A_i}.$$
Let $\mathscr H = (H_i)_{i\in \N_0}$ be the PIT$(\lambda, \gamma)$ as defined in Section~\ref{PIT}.
For  $i=1,2,\ldots$, let
$\rho(T_i)$ be the resident type in $\mathscr H$ at time $T_i$ (as introduced in Definition~\ref{defrchPIT}),
and let $M_i$ be defined as in~\eqref{defMi}.
\begin{prop}\label{prop-updateM} For all $i=1,2,\ldots$  and all $t_0 > 0$, as $N\to \infty$,
\begin{align}
    \label{convBN}
      (B_1^N, \ldots,B_i^N)
       &\xrightarrow{\,\,d\,\,} (B_1, \ldots, B_i),
    \\[.5em]
    \label{convHN}
      \big(\widehat H^N(t)\big)_{0\le t \le t_0}
       &\xrightarrow{\,\,d\,\,} \big( H(t)\big)_{0\le t \le t_0} \, \, 
          \mbox{\rm as random elements of }\big(\mathcal D([0,t_0], [0,1])\big)^{\N_0},
    \\[.5em]
    \label{convrhoN}
      (\rho^N(T_1), \ldots,\rho^N(T_i))
       &\xrightarrow{\,\,d\,\,} (\rho(T_1), \ldots,\rho(T_i)),
    \\[.5em]
    \label{convMN}
      (\widehat M_0^N,\ldots, \widehat M_i^N)
       &\xrightarrow{\,\,d\,\,} (M_0,\ldots, M_i).
    \end{align}
Moreover, for each $i$ the above convergences occur jointly.
\end{prop}
\begin{proof}
\eqref{convBN} follows by induction from \eqref{convBN1}. The convergence \eqref{convHN} is a consequence of \eqref{convBN} and the definitions of $H$ and $\widehat H^N$. The convergence \eqref{convrhoN} follows  from \eqref{convBN} together with the construction of the PIT and the fact that the $T_i$ have a continuous distribution. Finally, \eqref{convMN} results from \eqref{convrhoN} together with the update rules \eqref{defMi} and \eqref{updateMN}.
\end{proof}
Assertion~\eqref{convHNdist} of Theorem~\ref{theorem-THE} now follows by combining~\eqref{convprobHN} and~\eqref{convHN}, while~\eqref{convMNdist} results from combining~\eqref{convprobMN} with~\eqref{convMN}. 
The convergence~\eqref{convGN} follows from~\eqref{convrhoN} together with~\eqref{convLambda} and an induction argument.
To complete the proof of Theorem~\ref{theorem-THE} it remains to show~\eqref{convFN}. Recall that there the use of the $M_2$-topology is due to the fact that the average fitness at times of a resident change can take any value between the fitness of the former and the fitness of the new resident.
    Denote by $\rho(t)$ the resident in the system \[ \mathbb H((1,0),((T_i,A_i))_i) \]
    at time $t$ and let $t$ not be a resident change time.
    Then with (quenched) probability tending to~$1$, the inequality $\sum_{i \neq \rho(t)} X_i^N(t) \leq N/\log N$ holds and by~\eqref{convHN} it is also true that
    $$
      \Big(1-\frac1{\log N}\Big)M^N_{\rho(t)}
        \leq \overline F^N(t)
          \leq M^N_{\rho(t)} + \frac1{\log N}\max_{i\leq\mathscr I^N(t)}M^N_i.
    $$
    That is, conditionally given $(T_i,A_i)_i$, $d_{M_2}(\overline F^N,\widehat F^N)\to0$ in probability, where $\widehat F^N$ denotes the resident fitness in the PIT$((1,0),\beth^N)$ with $\beth^N:=(T_i,A_iB_i^N)_i$.
    Further, by Proposition~\ref{prop-updateM}, $\widehat F^N\xrightarrow{d}F$ with respect to the Skorokhod $J_1$-topology, which is stronger than the $M_2$-topology. The desired convergence thus holds,
    conditionally given $(T_i,A_i)_i$.
    Finally, by triangular inequality and dominated convergence,
    $$
      d_{M_2}(\overline F^N,F) \to 0
    $$
    in probability, without conditioning.\qed

 \section{Speed of adaptation in the PIT: Proof of Theorems~\ref{theorem-speed} and~\ref{theorem-speedCLT}
 }\label{sec-speedproof}

This section is  devoted to the proofs of the results stated in Section~\ref{sec-speedresult}.

\subsection{Proof of Lemma~\ref{lemFM}}\label{sec-prooflem}
a) Since by definition $F$ is constant between resident change times, the assertion~\eqref{represF} is equivalent to
\begin{equation}\label{sumV}
F(R_\ell) = \sum_{j=1}^\ell V_{\rho(R_j)}(R_j-), \quad \ell = 1,2,\ldots
\end{equation}
We will prove~\eqref{sumV} by induction over $\ell$. For $\ell =1$, we observe that $T_{\rho(R_1)} < R_1$, hence $\rho(T_{\rho(R_1)}) = 0$ and 
$$F(R_1) = M_{\rho(R_1)} = M_{\rho(T_{\rho(R_1)})} + A_{\rho(R_1)} = M_0 + V_{\rho(R_1)}(T_{\rho(R_1)})= 0 + V_{\rho(R_1)}(R_1-),$$
with the first two equalities being due to \eqref{defMi} and \eqref{finc}, and the last equality resulting from the kinking rule  in Definition~\ref{PITdyn} (since by definition of $R_1$ there is no trajectory reaching height 1 from below before time $R_1$ and hence $V_{\rho(R_1)}$ remains constant between $T_{\rho(R_1)}$ and $R_1-$).
For $\ell > 1 $ let $J_\ell^{<}:= \{j: 0< R_j < T_{\rho(R_{\ell})}\}$. By~\eqref{finc}, \eqref{defMi} and the induction hypothesis (which says  that~\eqref{sumV} is valid for $j=1,\ldots, \ell-1$ in place of $\ell$) we then have a.s. the chain of equalities
\begin{eqnarray*}
F(R_{\ell})-F(R_{\ell-1}) &=& M_{\rho(R_{\ell})}-F(R_{\ell-1})\\ &=& F(T_{\rho(R_{\ell})})+ A_{\rho(R_\ell)} -F(R_{\ell-1})\\ &=&\sum_{j\in J_\ell^{<}}V_{\rho(R_j)}(R_j-)+ V_{\rho(R_{\ell})}(T_{\rho(R_{\ell})})- \sum_{j=1}^{\ell-1}V_{\rho(R_j)}(R_j-)\\ &=&  V_{\rho(R_{\ell})}(R_{\ell}-),
\end{eqnarray*}
where again the last equality is due to the kinking rule in Definition~\ref{PITdyn}. This completes the induction step for proving~\eqref{sumV}. \\
b) The kinking rule in Definition~\ref{PITdyn} together with~\eqref{represF} shows that  for all $i\in \N$, as long as $H_i > 0$, every jump of $V_i$ corresponds to a jump of $F$. More precisely,
for all $t$ with $H_i(t) > 0$,
\begin{equation}\label{Vjumps}
V_i(t-)-V_i(t) = F(t)-F(t-).
\end{equation}
Thus~\eqref{Vjumpsum} results by summing~\eqref{Vjumps} over the resident change times between $t$ and~$t'$. $\Box$
\subsection{Renewals in the PIT}\label{sec-renewal}
We can view $(H_i(t), V_i(t)) _{i = 0,1,\ldots}$ as the state at time $t$
of a Markovian system of particles whose dynamics (apart from the birth of particles  given by  the  Poisson process $(T_i,A_iB_i)_{i\in \N}$)
 is deterministic and follows the interactive dynamics introduced in Definition~\ref{PITdyn}. We note that an immediate corollary of~\eqref{represF} is
\begin{equation}\label{estfitnessinc}
  F(t) \le \sum_{i \colon T_i \leq t} A_iB_i, \qquad t \ge 0.
\end{equation}

\begin{remark}\label{FalongL} The solitary resident change times $L_n$ specified in Definition~\ref{srch} initiate  idle periods of the particle system, with the next resident still waiting for its birth.   Since the trajectories $i$ for which $V_i(L_n) \leq 0$ never become resident after time $L_n$, we may forget about  them and observe that  $\mathbb H(((1,0)), \Psi)$ has the same distribution as $\mathbb H(((1,0)), \Psi_n)$  , where (as in Section~\ref{PIT}) $\Psi= ((T_i,A_i\cdot B_i))_{i \in \mathbb N}$, and 
$$\Psi_n :=((T_{i_n+i-1}-L_n,A_{i_n+i-1}\cdot B_{i_n+i-1}))_{i \in \mathbb N},$$where $i_n := \min\{j \in \N: T_j > L_n\}$. Thus  the $L_n$ form {\em regeneration} (or {\em renewal}) {\em times} for the PIT. Intuitively, the restrictions of the PIT to the intervals $[L_n, L_{n+1})$ can be seen as  i.i.d. ``clusters of trajectories'', whose concatenation renders the PIT. In particular, with $L_0:=0$,
\begin{equation}
F(L_n)= \sum_{\ell=1}^n  (F(L_\ell)-F(L_{\ell-1})), \quad n=1,2\ldots,
\end{equation}
and the random variables $\big(L_n-L_{n-1}, F(L_n)-F(L_{n-1})\big)$, $n=1,2,\ldots$, are independent copies of $(L_1, F(L_1))$.
\end{remark}

\begin{lemma}[Cluster lengths have finite moments] \label{clusterlength} The first solitary resident change time $L_1$ obeys
\[
\E[e^{\alpha L_1}]<\infty \quad \text{ for some $\alpha>0$.}
\]
In particular,
$\E[L_1^\nu]  < \infty$ for all $\nu\in \N.$
\end{lemma}
\begin{proof}   1. Let $i \in \N$ be such that 
\begin{equation}\label{solitarycond}
A_iB_i>0 \mbox{ and there is no } i'\neq i \mbox{ with } B_{i'} > 0  \mbox{ and }T_{i'} \in \big[T_i-\tfrac 2{A_i}, T_i+\tfrac 2{A_i}\big].
\end{equation}
We claim that as a consequence, the trajectory born at  time $T_i$ becomes resident not later than $T_i+\frac 2{A_i}$, and moreover that this resident change is solitary. 
To this purpose we first observe that any trajectory whose height $H_k(T_i)$ is strictly positive must have been born at some time $T_k < T_i-\frac 2{A_i}$ and hence must have at time $T_i$ a slope 
\begin{equation}\label{slopebound}
V_k(T_i) < \frac {A_i}2.
\end{equation} 
This is true because $t \mapsto V_k(t)$ is non-increasing on $[T_k,\infty)$ (which is clear by Definition~\eqref{defdyn}) and becomes non-positive as soon as $H_k(t)$ has reached height 1.
Let
$$
  S
   := \sup\big( \{T_i\} \cup \{R_\ell\mid \ell \in \N \mbox{ such that } \max_{t\le R_\ell} H_i(t)  < 1 \} \big). 
$$
On the event $\{S=T_i\}$ there are no resident changes after $T_i$  until the trajectory born at time~$T_i$ reaches height 1. Hence this trajectory keeps its initial slope $A_i$, reaches height 1 at time $T_i+\frac 1{A_i}$ and at this time  kinks the slopes of all the trajectories  whose height was positive at time $T_i$ to a negative value. 

On the event $\{S>T_i\}$, put $k:= \rho(S)$. Observing that $V_k(S) =0$  we obtain from  \eqref{Vjumpsum} and  
\eqref{slopebound}
$$F(S) -F(T_i) = V_k(T_i) -V_k(S) \le \frac {A_i}2.$$ Likewise, observing that $V_i(T_i)=A_i$, we obtain   from \eqref{Vjumpsum} 
$$V_i(S)- A_i = F(T_i)-F(S)  \ge -\frac {A_i}2,$$
hence $V_i(S)  \ge \frac {A_i}2$. Consequently, the trajectory born at time~$T_i$ keeps a slope of at least $\frac {A_i}2$ until it becomes resident at some time $R\le T_i+\frac 2{A_i}$. Thus, all trajectories that were born before time $T_i-\frac 2{A_i}$ and at time $R$ have height in $(0,1]$  are kinked to a negative slope at time $R$, and by assumption no contending trajectories except $H_i$  are born in the time interval $[T_i-\frac 2{A_i}, T_i+\frac 2{A_i}]$. Hence $R$ is the time of a solitary resident change. An
illustration of this step is available in Figure~\ref{figure-renewal}.

\begin{figure}
\begin{tikzpicture}[scale=3]

  % axes
    \draw[->,color=black] (1.3, 0) -- (5.45, 0) node[right] {time};
    \draw[->,color=black] (1.3, 0) -- (1.3, 1.05) node[above] {};

  % marks on y-axis
    \draw[black] (1.27,1) node[left] {\footnotesize $1$};
    \draw[black] (1.27,0) node[left] {\footnotesize $0$};
 
  % initial horizontal
   % \draw[domain=0.9:1.85, dotted, variable=\x, black] plot ({\x}, {1});
    
  % marks on x-axis
    \draw[blue] (2.45,-0.05) node[below] {\footnotesize $T_i-2/A_i$};
  %  \draw (3,-0.05) node[below] {\footnotesize $t+\frac{2}{K}$};
    %\draw[blue] (3.38, -0.05) node[below] {\footnotesize $ T_i=R$};
     \draw[blue] (3.38, -0.05) node[below] {\footnotesize $ T_i$};
    \draw[brown] (3.65,-0.05) node[below] {\footnotesize $S$};
   % \draw (3.7,-0.05) node[below] {\footnotesize $R$};
    \draw[blue] (3.985,-0.04) node[below] {\footnotesize $ R$};
    \draw[blue] (4.4,-0.05) node[below] {\footnotesize $T_i+2/A_i$};

  % vertical lines, from left to right
    \draw[dashed,blue,thick] (2.4, 0) -- (2.4,1) node[above] {};
    \draw[dashed, brown,thick] (1.65, 0) -- (1.65,1) node[above] {};
    \draw[dashed, blue,thick] (3.4, 0) -- (3.4,1) node[above] {};
    \draw[brown,thick] (3.65, 0) -- (3.65,1) node[above] {};
    \draw[blue,thick] (3.985, 0) -- (3.985,1) node[above] {};
    \draw[dashed,blue,thick] (4.4, 0) -- (4.4,1) node[above] {};

  % gray mutant
    %\draw[domain=0.9:1.85, variable=\x, gray] plot ( {\x}, {0.5*(\x-0.9)+0.525} );
    %\draw[gray] (1.1,0.83) node[right] {\footnotesize $0.5$};
    \draw[domain=1.3:3.65, variable=\x, gray,thick] plot ( {\x}, {1} );
    \draw[ domain=3.65:3.985, variable=\x, gray,thick] plot ( {\x}, {1-0.5*(\x-3.65)} );
    \draw[domain=3.985:4.4, variable=\x, gray,thick] plot ( {\x}, {1-0.1667-2*(\x-3.985)} );
    \draw[gray] (4,0.27) node[right] {\footnotesize $-2$};
    \draw[gray] (3.6,0.85) node[right] {\footnotesize $-0.5$};

  % brown mutant    
    %\draw[brown] (1.7,0.2) node[right] {\footnotesize $1$};
    \draw[brown,thick] (2.63,0.67) node[right] {\footnotesize $0.5$};
    \draw[brown,thick] (4.1,0.77) node[right] {\footnotesize $-1.5$};
    %\draw[domain=1.75:1.85, variable=\x,brown] plot ( {\x}, {\x-1.75} );
    \draw[domain=1.65:3.65, variable=\x,brown,thick] plot ( {\x}, {0.5*(\x-1.85)+0.1} );
    \draw[domain=3.65:3.985, variable=\x, brown,thick] plot ({\x}, {1});
    \draw[domain=3.985:4.652, variable=\x, brown,thick] plot ({\x}, {1-1.5*(\x-3.985)});

  % blue mutant
    \draw[blue,thick] (3.75,0.22) node[below] {$2=A_iB_i$};
    \draw[blue,thick] (3.82,0.65) node[below] {\footnotesize $1.5$};
    \draw[domain=3.4:3.65, smooth, variable=\x, blue,thick] plot ( {\x}, {2*(\x-3.4)});
    \draw[domain=3.65:3.985, smooth, variable=\x, blue,thick] plot ( {\x}, {1.5*(\x-3.65)+0.5});
    \draw[domain=3.985:5.45, variable=\x, blue,thick] plot ({\x}, {1});
    
  % the three phases
  %  \draw (2.4,-0.3)  node[below] {\footnotesize \text{(a)}};
   % \draw (3.35,-0.3) node[below] {\footnotesize \text{(b)}};
   % \draw (4.35,-0.3) node[below] {\footnotesize \text{(c)}};

\end{tikzpicture}

\begin{tikzpicture}[scale=3]

  % axes
    \draw[->,color=black] (1.3, 0) -- (5.45, 0) node[right] {time};
    \draw[->,color=black] (1.3, 0) -- (1.3, 1.05) node[above] {};

  % marks on y-axis
    \draw[black] (1.27,1) node[left] {\footnotesize $1$};
    \draw[black] (1.27,0) node[left] {\footnotesize $0$};
 
  % initial horizontal
   % \draw[domain=0.9:1.85, dotted, variable=\x, black] plot ({\x}, {1});
    
  % marks on x-axis
    \draw[blue] (2.45,-0.05) node[below] {\footnotesize $T_i-2/A_i$};
  %  \draw (3,-0.05) node[below] {\footnotesize $t+\frac{2}{K}$};
    \draw[blue] (3.38, -0.05) node[below] {\footnotesize $ T_i=S$};
    \draw[blue] (3.935,-0.04) node[below] {\footnotesize $ R$};
    \draw[blue] (4.4,-0.05) node[below] {\footnotesize $T_i+2/A_i$};

  % vertical lines, from left to right
    \draw[dashed,blue,thick] (2.4, 0) -- (2.4,1) node[above] {};
    %\draw[dashed, brown] (1.65, 0) -- (1.65,1) node[above] {};
    \draw[dashed, blue,thick] (3.4, 0) -- (3.4,1) node[above] {};
   % \draw[brown] (3.65, 0) -- (3.65,1) node[above] {};
    \draw[blue,thick] (3.9, 0) -- (3.9,1) node[above] {};
    \draw[dashed,blue,thick] (4.4, 0) -- (4.4,1) node[above] {};

  % gray mutant
    %\draw[domain=0.9:1.85, variable=\x, gray] plot ( {\x}, {0.5*(\x-0.9)+0.525} );
    %\draw[gray] (1.1,0.83) node[right] {\footnotesize $0.5$};
    \draw[domain=1.3:3.9, variable=\x, gray,thick] plot ( {\x}, {1} );
    \draw[domain=3.9:4.4, variable=\x, gray,thick] plot ( {\x}, {1-2*(\x-3.9)} );
    \draw[gray] (4,0.27) node[right] {\footnotesize $-2$};
   % \draw[gray] (3.6,0.85) node[right] {\footnotesize $-0.5$};

  % blue mutant
    \draw[blue] (3.69,0.22) node[below] {\scriptsize $2=A_iB_i$};
    %\draw[blue] (3.82,0.65) node[below] {\footnotesize $1.5$};
    \draw[domain=3.4:3.9, smooth, variable=\x, blue,thick] plot ( {\x}, {2*(\x-3.4)});
    \draw[domain=3.9:5.45, variable=\x, blue,thick] plot ({\x}, {1});
    
  % the three phases
  %  \draw (2.4,-0.3)  node[below] {\footnotesize \text{(a)}};
   % \draw (3.35,-0.3) node[below] {\footnotesize \text{(b)}};
   % \draw (4.35,-0.3) node[below] {\footnotesize \text{(c)}};

\end{tikzpicture}
\caption{
Illustration of part 1 of the proof of Lemma~\ref{clusterlength}. Top: case $\{ S > T_i \}$. Between times $T_i-2/A_i$ and $T_i+2/A_i$ there is no birth time apart from $T_i$, while we have $V_i(T_i)=A_iB_i=2$. All slopes of trajectories that are still positive at time $T_i$ are at most $A_i/2$, and hence the $i$-th trajectory reaches height 1 at time $R \leq T_i + 2/A_i$ at latest, kinking all other trajectories with current heights in $(0,1]$ to a negative slope. In the picture, the only trajectory still having a positive slope at time $T_i$ is the brown one, and $S$ is the time when this trajectory reaches height 1. The slope of the brown trajectory in $[T_i,S)$ equals $0.5$, and thus at time $R$, the blue trajectory is kinked to slope $2-0.5=1.5 \geq A_i/2$. The gray trajectory corresponds to the mutant who is resident at time $T_i$ (this is the last resident before the brown one). \\
Bottom: case $\{ S=T_i \}$. Now the brown mutant is absent, so that the blue trajectory suffers no kink before reaching height 1, and the
previous resident before the blue one is the gray one. Note that here, the time when the blue trajectory reaches height 1 is $R=T_i+1/A_i$.}
\label{figure-renewal}
\end{figure}

2. Let $i_0:= \min \{ i\in \N \mid i \mbox{ has property } \eqref{solitarycond}\}$. We claim that $i_0 < \infty$ a.s.\
and that $\E[e^{\alpha R_0}]< \infty$ for some $\alpha>0$,
where~$R_0$ is the time at which the trajectory born in $T_{i_0}^{}$ becomes resident. To see this, consider the Poisson point process $\Phi := \sum_{i\in \N}\delta_{(T_i, A_iB_i)}$. Let $a_0 > 0$ be such that 
$\gamma([a_0, \infty)) > 0$. For $n \in \N$ we define the sets $C_n$, $D_n \subset \R_+\times \R_+$ and the events $E_n$ by
$$C_n :=\big[\tfrac{5n+2}{a_0}, \tfrac{5n+3}{a_0}\big]\times[a_0, \infty), 
\quad 
D_n:=\Big(
\Big[\tfrac{5n}{a_0}, \tfrac{5n+5}{a_0}\Big] 
\times \R_+\Big)
\setminus C_n,
$$
\[
  E_n
   := \{\Phi(C_n) = 1\}\cap \{\Phi(D_n) = 0\}.
   \numberthis\label{Endef}
\]
The events $E_n$ are independent and have a probability that does not depend on $n$. Due to our choice of~$a_0$ this probability is positive.
Therefore, the random variable $K := \min\{n \mid \1_{E_n} =1\}$ is a geometric random variable with a positive parameter. This implies that $\E[e^{\alpha' K}]<\infty$ for some $\alpha'>0$. Since by construction $R_0 \le \frac{5(K+1)}{a_0}$, it is enough to take $\alpha = a_0 \alpha'/5$.

3. Because of step 1, the resident change time $R_0$ found in step 2 is solitary. Obviously, $L_1 \le R_0$, and thus 
$\E[e^{\alpha L_1}]<\infty$ with $\alpha>0$ as in step 2.
\end{proof}

\subsection{Proof of Theorem~\ref{theorem-speed}}\label{sec-mainrenewalproofs}
Theorem~\ref{theorem-speed} is a direct consequence of the following proposition, which in turn relies on the just proved key Lemma~\ref{clusterlength}. 
\begin{prop}\label{renrewprop}
a) \, Almost surely, $\lim\limits_{t\to \infty}\tfrac{F(t)}t$ \, exists, and equals $\overline v :=  \frac{\E[F(L_1)]}{\E[L_1]}$.

b)\, \,$\overline v  \le \lambda \E[A_1B_1]$.

c)   $\overline v < \infty$ if and only if   $\int_0^\infty a \, \gamma(da) < \infty$.
\end{prop}
\begin{proof} In view  of Remark~\ref{FalongL}, 
\begin{equation}\label{Fhat}
\widehat F(t):= \sum_{i\ge 1}\1_{\{L_i\le t\}}(F(L_i)-F(L_{i-1})) =\sum_{i=0}^{\infty} F(L_i) \mathds 1_{[L_i,L_{i+1})}(t), \qquad t\geq 0,
 \end{equation} is a renewal reward process, and  thanks to Lemma~\ref{clusterlength} assertion a) is a quick consequence of the law of large numbers.   For convenience of the reader we recall the argument.  For $t \ge 0$ let $n(t)$ be such that $L_{n(t)} \le t < L_{n(t)+1}$. Then
\begin{equation}\label{sandwichF}
\frac{F(L_{n(t)})/n(t)}{L_{n(t)+1}/n(t)} \le \frac{F(t)}t \le \frac{F\left(L_{n(t)+1}\right)/n(t)}{L_{n(t)}/n(t)} 
\end{equation}
Since 

$\bullet \quad n(t)\to \infty$ a.s.\ as $t\to \infty$, 

$\bullet \quad L_n$ is a sum of i.i.d.\ copies of $L_1$ which has finite expectation by Lemma~\ref{clusterlength},

$\bullet \quad F(L_n)$ is a sum of i.i.d.\ copies of $F(L_1)$, 

\smallskip \noindent
both the left and the right hand side of \eqref{sandwichF} converge a.s.\ to $\frac{\mathbb E[F(L_1)]}{\mathbb E[L_1]}$. This proves assertion a).

To show assertion b) we first observe that the strong law of large numbers for renewal processes gives the a.s.\ convergence
$\frac 1t  \sum_{i: T_i < t} A_iB_i \to  \lambda \E[A_1B_1]  \mbox{ as } t\to \infty.$
Combining this with~\eqref{estfitnessinc} results in assertion b).

We now turn to the proof of c). From the definition of $\overline v$ and Lemma~\ref{clusterlength} it follows that $\overline v < \infty$ if and only if $\E[F(L_1)] < \infty$. On the other hand, the finiteness of $\int a \gamma(da)$ clearly is equivalent to the finiteness of $\E[A_1B_1] = \int a\frac a{a+1} \gamma(da)$. In view of the proposition's part b)  it thus only remains to show that $\E[F(L_1)]$ is infinite provided $\gamma$ has infinite expectation. This, however, follows from the estimate
\[
  \E[F(L_1)]
   \ge \E[A_1\1_{\{A_1B_1\geq1\}\cap\{T_1<1\}\cap\{T_2\geq2\}}].
\]
\end{proof}

\subsection{Proof of Propositions~\ref{prop-1case} and~\ref{highmut}}
\begin{proof}[Proof of Proposition~\ref{prop-1case}.] 
Let $T$ be the time at which the first contending mutation appears. The time $T$ has an exponential distribution whose parameter is $\lambda \frac c{1+c}$, the intensity of the birth process of contending mutations.  The first contending mutation becomes resident at time  $R:=T +\frac 1c$, and all contending mutations that are born in the time interval $(T, R)$ are kinked to slope $0$ at time $R$. This means that $R$ is the first solitary resident change time $L_1$ specified in Definition~\ref{srch}. This time has expectation $$
  \E[L_1]
    = \E[T] + \frac 1c = \frac 1{\frac{\lambda c}{1+c} }+ \frac 1c
    = \frac{1+c+\lambda}{c\lambda},
$$
and the ``renewal reward'' $F(L_1)$ has the deterministic value $c$. Thus, the assertion of Proposition \ref{prop-1case}  follows directly from Proposition~\ref{renrewprop}~a).
\end{proof}
\begin{proof}[Proof of Proposition~\ref{highmut}.]
 Recalling Definition~\ref{defdyn},
  consider the system $\mathbb H(\aleph,\beth)$ where $\aleph=((1,0),(0,b))$ and $\beth=((\frac ib,b))_{i\geq1}$. There, at time $0$ immediately a line starts with slope $b$ and, just as that hits $1$, the next line starts with slope $b$ and so on. In this system, the resident fitness will always jump up by $b$ at times $i/b$, $i\in\N$, and thus equals
  $b\lfloor bt\rfloor$ at any time $t$. 
  This system describes a best case scenario for the PIT$(\lambda, \gamma)$ in this proposition, in the sense that the resident fitness of the PIT$(\lambda, \gamma)$ obeys $F_{\lambda}(t) \le b\lfloor bt\rfloor$.  Since $\P(F_\lambda(\frac1b)=b) = 0$, we obtain for all $t$ that almost surely $F_\lambda(t)$ is bounded from above by the left-continuous version of $t\mapsto b\lfloor bt\rfloor$, i.e.\ $F_\lambda(t) \leq b(\lceil bt\rceil - 1)$.
  
  For a lower bound let $\Psi_\lambda$ be a Poisson point process of intensity $\lambda dt\otimes\gamma$, fix $\varepsilon\in(0,\frac b2)$ and note that the probability of the event
  \[
    E_\lambda
      := \{\Psi_\lambda\cap ([0,\varepsilon)\times[b-\varepsilon,b])=\emptyset\},
      \qquad\text{i.e.\ }
      e^{-\lambda\varepsilon\gamma([b-\varepsilon,b])},
  \]
  tends to $0$ as $\lambda\to\infty$.
  Now, note that outside of $E_\lambda$ there is at least one mutant line born before time $\varepsilon$ of slope at least $b-\varepsilon$. Hence, the first change of resident will be at the latest at time $\varepsilon+\frac1{b-\varepsilon}$ and will add fitness of at least $(\varepsilon+\frac1{b-\varepsilon})^{-1}$. At that moment, all other contenders will be kinked to a slope of at most $\varepsilon<b-\varepsilon$. From there, we can iterate and obtain
  $$
    F_\lambda(t)
      \geq \Big(\varepsilon+\frac1{b-\varepsilon}\Big)^{-1}
            \bigg(
              \Big\lceil\Big(\varepsilon+\frac1{b-\varepsilon}\Big)^{-1}t\Big\rceil
               - 1 \bigg)
  $$
  on an event of probability $\P(E_\lambda^c)^{\lfloor (\varepsilon+\frac1{b-\varepsilon})^{-1}t\rfloor}\to1$. Since $(\varepsilon+\frac1{b-\varepsilon})^{-1}\uparrow b$, as $\varepsilon\downarrow0$, the proposition holds.
\end{proof}

\subsection{Proof of Theorem~\ref{theorem-speedCLT}}\label{sec3_4}

1. In order to apply the result of Appendix~\ref{s:RenRew} to the renewal reward process $\widehat F$ defined in \eqref{Fhat} with $F(0)=0$, 
we have to check that, under our assumption that $\int_0^\infty a^2\gamma(\d a)<\infty$,
\begin{equation}\label{finsecmom}
\E[F(L_1)^2]< \infty.
\end{equation}
In order to exploit the independence properties of the Poisson process $(T_i, A_iB_i)_{i\ge 1}$ we work with the random variable $K$ defined in the proof of  Lemma~\ref{clusterlength} and set out to show that
\begin{equation}\label{finsecmomK}
 \E\Big[F\Big(\tfrac{5(K+1)}{a_0}\Big)^2\Big] < \infty.
\end{equation}
In view of $L_1 \le 5(K+1)/a_0$,  the  representation~\eqref{sumV} and the estimate~\eqref{estfitnessinc} we have
\begin{equation}\label{estimateofF}
F(L_1) \le F\Big(\tfrac{5(K+1)}{a_0}\Big) \le \sum_{n=0}^K{X_n}
\end{equation}
where
$$
  X_n
   := \sum_{i\ge 1} \1_{\big\{\tfrac{5n}{a_0} \le T_i < \tfrac{5(n+1)}{a_0}\big\}} A_iB_i,
   \quad n \ge 0.
$$ 
Thus for proving~\eqref{finsecmom} it suffices to show that the second moment of the r.h.s. of~\eqref{estimateofF} is finite.
By definition of $K$ and from the second moment assumption on $\gamma$, 
$$
  \E[X_K^2] 
    = \E[ A_1^2 \mid A_1 B_1 \geq a_0]
    = \int_{a_0}^\infty a^2\tfrac a{a+1} \gamma(\d a) \bigg /  \int_{a_0}^\infty \tfrac a{a+1} \gamma(\d a) < \infty.
$$
We know from the proof of Lemma~\ref{clusterlength} that $\E[K^2]< \infty$. Hence the finiteness of the second moment of the r.h.s. of~\eqref{estimateofF} is guarenteed if we can show that
\begin{equation} \label{condsecmom}
\E[X_n^2 \mid n< K] = c <\infty
\end{equation}
with $c$ not depending on $n$. For this we use the terminology from  the proof of Lemma~\ref{clusterlength}. Both $\Phi(C_n)$ and $\Phi(D_n)$ are Poisson random variables  with parameters that depend only on $\lambda, \gamma$ and $a_0$, let us put $\alpha_C:=\E[\Phi(C_n)]$ and $\alpha_D:=\E[\Phi(D_n)]$. 
Recalling from~\eqref{Endef} that
\[
  E_n^c
   = \{\Phi(C_n) \neq 1\} \cup \{\Phi(D_n) \neq 0\},
\]
note that
$
  \E[X_n^2 \mid n< K]
    = \E[X_n^2 \mid \cap_{k=1}^n E_k^c]
    = \E[X_n^2 \mid E_n^c]
$
since $X_n$ is independent of $\1_{E_k}$ for $k \neq n$.
We write $X_n = X_{C,n} + X_{D,n}$, with
$$
  X_{C,n}
   := \sum_{i:(T_i,A_i)\in C_n} A_iB_i,
  \qquad
  X_{D,n}
   := \sum_{i:(T_i,A_i)\in D_n} A_iB_i
$$
The random variables $X_{C,n}$ and $X_{D,n}$  are measurable w.r.t.\ the random point measures $\Phi\big |_{C_n^{}}$ and~$\Phi\big |_{D_n^{}}$, respectively. Conditioning these random point measures under the event $E_n^c$ affects only the number of their points in the sets $C_n$ and $D_n$ and not the distribution of the points' locations. Recalling that $\gamma^*(da) = \tfrac a{1+a} \gamma(da) \big /  \int \tfrac {a'}{1+a'} \gamma(da')$, let $Y_C$ and $Y_D$ be random variables with distribution $\gamma^* \big |_{[a_0, \infty)} \big / \gamma^*([a_0,\infty))$ and $\gamma^*$, respectively. The above considerations imply
\begin{equation}\label{Xnsquare}
\begin{aligned}
  \E[X_n^2 \, |\, E_n^c]
    & \le 2 \big(\E[X_{C,n}^2 \, |\, E_n^c] + \E[X_{D,n}^2 \, |\, E_n^c]\big) \\
    & \le 2 \big(\E[\Phi(C_n)^2 \, |\, E_n^c] \, \E[Y_C^2]
       + \E[\Phi(D_n)^2 \, |\, E_n^c] \,\E[Y_D^2]\big).
\end{aligned}
\end{equation}
Our second moment assumption on $\gamma$ implies that both $\E[Y_C^2]$ and $\E[Y_D^2]$ are finite. 
Thanks to the assumption $\gamma([a_0, \infty)) > 0$ we have $\alpha_C > 0$.
Hence
$$\P(E_n^c) \ge \P(\Phi(C_n) \neq 1)) = 1-\alpha_Ce^{-\alpha_C} =:\beta_C > 0.$$
Consequently,
$$\E[\Phi(C_n)^2+ \Phi(D_n)^2\mid E_N^c] \le \frac 1{\beta_C}\left(\alpha_C(\alpha_C+1)+\alpha_D(\alpha_D+1)\right)< \infty.$$
This shows that the r.h.s. of~\eqref{Xnsquare} is finite and does not depend on $n$, thus showing~\eqref{condsecmom} and completing the proof of~\eqref{finsecmom}.

2. 
The quantity
\begin{equation}\label{defsigma}
    \sigma^2 := \frac{\E[(F(L_1)-\overline{v} L_1)^2]}{\E[L_1]}
\end{equation}
is finite by \eqref{finsecmom} and Lemma~\ref{clusterlength},
and positive since the random variable $F(L_1) - \overline{v} L_1$ is not almost-surely constant.
Then Theorem~\ref{thm:FCLTRenRew} applied to the renewal reward process $\widehat{F}$ implies
\begin{equation}\label{CLTEKM}
  \bigg(\frac{\widehat{F}(nt)-\overline v nt}{\sigma \sqrt{n}} \bigg)_{t \geq 0} 
    \overset{d}{\longrightarrow} W \quad\mbox{as}\quad n\to \infty.
\end{equation}
It is plain that $\widehat{F}(t) \le F(t)$ for all $t\ge 0$.
On the other hand, considering
\[\widetilde F(t)= \sum_{i=0}^{\infty} F(L_{i+1}) \mathds 1_{[L_i,L_{i+1})}(t), \]
we have $\widetilde F(t) \geq F(t)$ for all $t$.
In order to conclude, it suffices to show that, for any $M >0$,
\[\sup_{ t \in [0,M]} \frac{\widetilde F(nt) - \widehat F(nt)}{\sqrt n} \overset{n\to\infty}{\longrightarrow} 0 \qquad \text{in probability}, \numberthis\label{differenceto0inP} \]
since this will imply that the Skorokhod distance between diffusive rescalings of $\widehat{F}$ and $F$
will go to zero in probability and hence \eqref{CLTEKM} will be valid with $F$ in place of $\widehat{F}$ as well.
To that end, denote by $N_t = \sup \{n \in \N \colon\, L_n \leq t\}$ (with $\sup \emptyset = 0$) the number of SRC times up to time $t$, and note that $\widehat{F}(t)=F(L_{N_t})$, $\widetilde{F}(t) = F(L_{N_t+1})$.
By \cite[Theorem~2.5.10]{EKM97},
\begin{equation}\label{prCLTNt}
\lim_{t \to \infty} \frac{N_t}{t} = \frac{1}{\E[L_1]} \quad \text{almost surely,}
\end{equation}
and $\E[L_1] \in (0, \infty)$ by Lemma~\ref{clusterlength}.
Now, for $M,\eps>0$,
\begin{equation*}
\begin{aligned}
  & \P\Big(\sup_{t \in [0,M]} \widetilde{F}(nt)-\widehat{F}(nt) > \varepsilon \sqrt{n}\Big)\\
  \leq \, & \P(N_{nM}>2 nM/E[L_1]) + \P\big(\exists k \leq 2nM/E[L_1] \colon\, F(L_{k+1}) - F(L_k) > \eps \sqrt{n} \big) \\
  \leq \, & \P(N_{nM}>2 nM/E[L_1]) + (2nM/E[L_1]+1)\P(F(L_1) > \varepsilon \sqrt{n}).
\end{aligned}
\end{equation*}
The first term in the r.h.s.\ above goes to zero as $n \to \infty$ by \eqref{prCLTNt}, and the second term
also goes to zero because $F(L_1)$ is square-integrable.
This concludes the proof.
\hfill $\qed$

\section{Fixation of mutations in the PIT and heuristics for the speed of adaptation}\label{sec-construction}\label{sec:fixandGLh}
Complementing Sections~\ref{PIT} and~\ref{sec-speedresult}, Section~\ref{compfix} states properties of the genealogy of mutations of the PIT.    
 \subsection{Fixation of mutations in the PIT}\label{compfix}
Recalling Definition~\ref{genmut} of the genealogy of mutations in the PIT, a mutation is said to {\em fix} (or {\em to reach fixation}) if it is ancestral to all mutations in the far future. 
Clearly, only contending mutations have a chance to fix (recall the notion of contenders from Remark~\ref{remark-discarding0slopes}).  
We consider three attributes of contending mutations: 
\\
        \phantom{UASR}R: \ becoming resident, \\
        \phantom{SRR}UA: \ becoming {\em ultimately ancestral}, i.e.\ eventually reaching fixation, \\
        \phantom{UAR}SR: \ becoming solitary resident.

\begin{lemma}\label{compfixlemma}For contending mutations in the PIT the following implications are valid:
$$\mbox{SR} \Rightarrow \mbox{UA} \Rightarrow \mbox{R}. $$ 
For neither of the two implications, the converse
is true in general.
If, however, fitness advantages are deterministic and fixed, then R implies SR.
\end{lemma}
\begin{proof}
    The second implication is clear. For the first one, assume that mutation $i$ becomes resident at some time $r$.  If there is no trajectory $H$ in $\mathscr H$ with $v_H(r) > 0$, then no mutation that happened before $r$ will become resident after time~$r$, and all the mutations happening after time $r$  will be descendants of $i$. The fact that the converse of the implications is not true in general is shown by Figure~\ref{threemutants}: There, the mutation corresponding to the blue trajectory is UA but not SR (but its green child is SR), while the mutation corresponding to the red trajectory is R but not UA.
    
    For the last assertion, observe that if fitness advantages are deterministic and fixed, then every resident change is solitary; cf.\ the proof of Proposition~\ref{prop-1case} in Section~\ref{sec-mainrenewalproofs}. 
\end{proof}
\begin{remark}(Fixation and solitary resident changes)

  \begin{enumerate}[a)]
    \item As a consequence of Lemma~\ref{compfixlemma}, the event whether the mutation born at time~$T_i$ goes to fixation is measurable with respect to the past of the first solitary resident change after $T_i$. However, this event is not measurable with respect to the past of $T_i$. Indeed, trajectories born after (but close to) time $T_i$ with initial slopes higher than $A_i$ may become resident before the $i$-th trajectory (and in that case, the $i$-th trajectory never becomes resident).
    \item It may well happen that a mutation becomes ultimately ancestral even though the clonal subpopulation belonging to this mutation becomes extinct before its first UA descendant becomes resident. 
    For an example, see Figure~\ref{fig-limitingexample}.
    Here, mutation 6 (orange) becomes resident at a solitary resident change time ($R_3$). Hence mutation 6 as well as its parent, mutation 3 (red), go to fixation. However, the type 3 subpopulation goes extinct before time $R_3$ due to its interference with the type 4 subpopulation (blue), which in turn is outcompeted by type~6.
\item By definition, every SR-mutation is ancestral to all mutations born after the time at which the SR-mutation became resident (and thus in particular is a UA-mutation). Conversely,
every UA-mutation $i$ is ancestral to any SR-mutation that becomes resident after $T_i$. (Indeed, assume that $j$ is an SR-mutation becoming resident at time $t$, and consider a mutation $i$ born before time $t$ that is non-ancestral to $j$.
Since all trajectories with positive height at time $t$, except the one belonging to $j$, have negative slope at time $t$, the offspring of $i$ will be extinct either before time $t$ or by some finite time after $t$, showing that $i$ cannot be UA.) 
 \item  
    As ensured by  Lemma~\ref{clusterlength}, the expected number of mutations between two subsequent SR-mutations is finite. Arguing as in item c), we thus see that (with probability 1) any mutant that has infinitely many descendants is ancestral to some SR-mutation (and therefore is UA). Hence the set of UA-mutations is the set of all ancestors of SR-mutations, and thus  constitutes the unique infinite path within the tree $\mathscr G$ introduced in Definition~\ref{genmut}. 
  \end{enumerate}
\end{remark}

\subsection{Heuristics for the speed of adaptation}\label{speedheur}

\begin{figure}
  \includegraphics[width=6.45cm,trim={0, 1.25cm, 0, 2cm},clip]
   {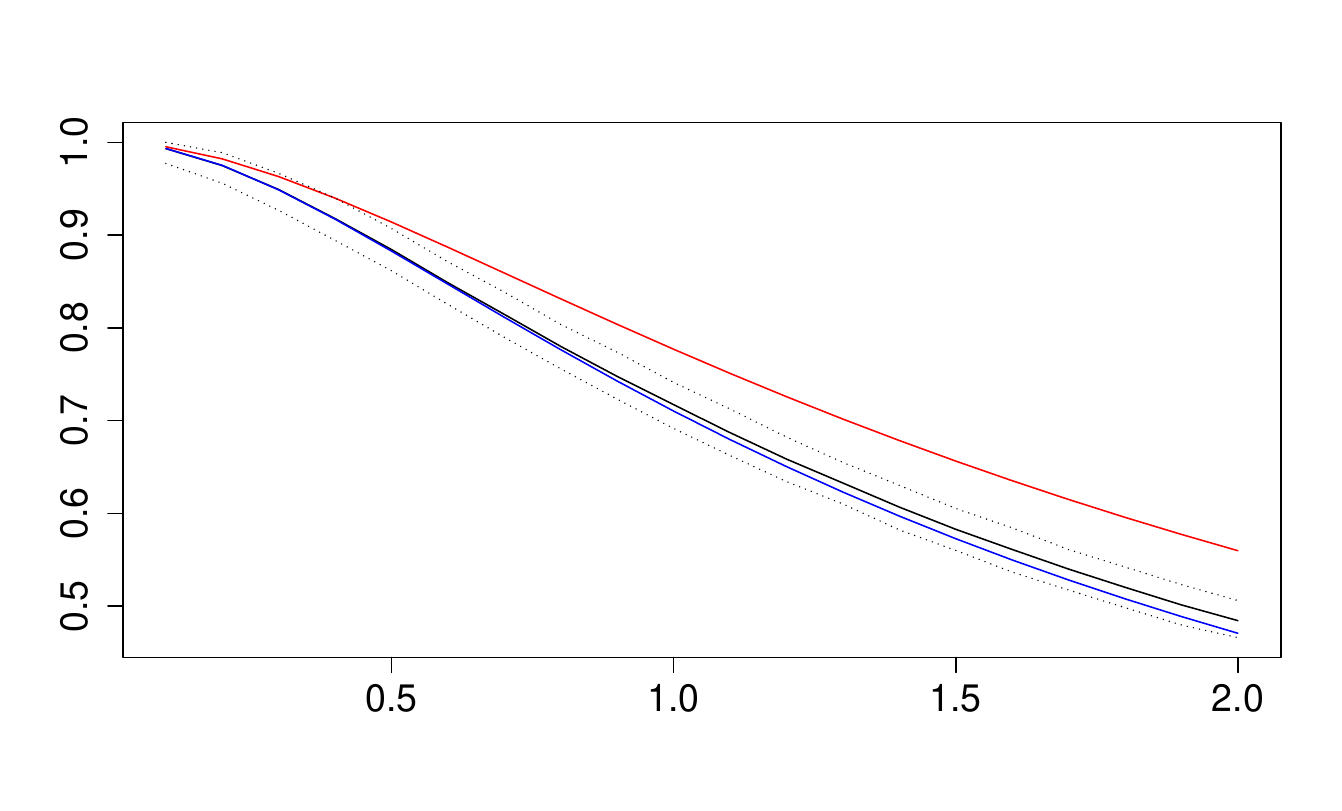}
  \includegraphics[width=6.45cm,trim={0, 1.25cm, 0, 2cm},clip]
   {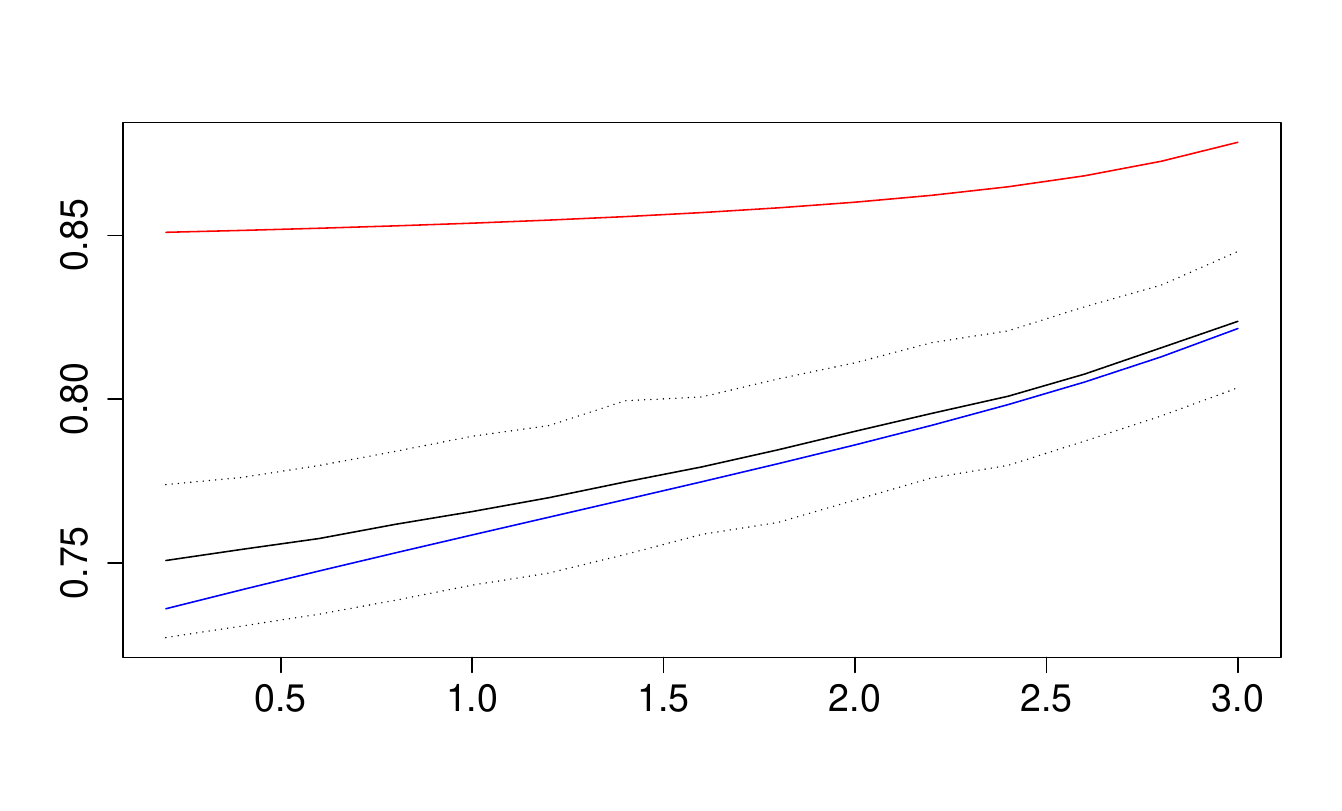}
  \\[-.3em]
  {\tiny$\lambda^*$\hspace{6.25cm}$\sigma$}
   \caption{\label{fig-GLh-sim}
     Simulations of speed of adaptation via $F(1000)/\sum_{i:T_i^*\leq1000}A_i^*$
     with $1000$ iterations.
     Black line: mean, black dotted: $95\%$ confidence interval, red: $\overline v_{\textnormal{GL}}/\lambda^*$, blue: $\overline v_{\textnormal{rGL}}/\lambda^*$.
     Left: For each parameter $\lambda^*\in\{0.1,0.2,\ldots,1.9,2\}$ simulation with
     $\gamma^*=\text{Exp}(1/\lambda^*)$. As $\lambda^*$ increases, expected total sum of increments
     remains constant, while effects of clonal interference increase (more mutations,
     longer fixation times) and hence the speed declines.
     Right: $\lambda^*=1$ and for each $\sigma\in\{0.2,0.4,\ldots,2.8,3\}$ we choose
     $\gamma^*=\textnormal{unif}([3-\sigma,3+\sigma])$. With $\sigma$, the variance of $\gamma^*$
     (i.e.\ $\sigma^2/3$) increases while its mean remains constant.
   }
\end{figure}

Gerrish and Lenski~\cite{GL98} proposed a heuristic for predicting the speed of adaptation which can be formulated and discussed in our  framework as follows. 

Consider a contender born at time $T^*_i$ with fitness increment $A^*_i$, and let $E_i$ be the event that its  trajectory is not kinked by a previous resident change. On the event $E_i$, this contender becomes solitary resident  if and only if between times $T^*_i$ and
$T_i^*+ (A_i^*)^{-1}$ there is no birth of another contender whose fitness increment is larger than $A^*_i$.  In other words, given the event $E_i$ {\em and} given $A_i^*=a$, the contender becomes solitary resident with probability 
\begin{align*}
  \pi_{\text{GL}}(a)
   &= \exp\Big(-\frac{\lambda^*}{ a}\gamma^*((a,\infty))\Big), \quad a > 0,
\end{align*}
where we recall $\lambda^*$ and $\gamma^*$ from Remark~\ref{remark-discarding0slopes}.

Retaining only such mutations (and neglecting the relevance of the events $E_i$) leads to the following prediction of the speed, called Gerrish--Lenski heuristics and abbreviated as GLh:
\begin{equation}\label{GLh}
    \overline v_{\text{GL}}:= \lambda^* \int a\, \pi_{\text{GL}}^{}(a)\,\gamma^*(\d a).
\end{equation}
Ignoring negative effects from the past by assuming $E_i$ naturally constitutes an overestimation of the speed, as confirmed in Figure~\ref{fig-GLh-sim}.

The \emph{refined} Gerrish--Lenski heuristics (rGLh) introduced by Baake et al.\ \cite{BGPW19} takes into account not only the future but also the past,
  by the following consideration:\\
  Denote by $K_i=\{j<i\mid T^\ast_j<T^\ast_i<T^\ast_j+(A^\ast_j)^{-1}, A^\ast_j\geq A^\ast_i\}$. This is the set of mutations $j$ born prior to $i$ that would kink the $i$th trajectory to a negative slope before it reaches $1$ and hence prohibit its becoming a resident -- provided that no further interference occurs. The rGLh now suggests the event $\{|K_i|=0\}$ as an approximation of $E_i$ leading to an (estimated) retainment probability given $A^\ast_i=a$ of
\begin{align*}
  \pi_{\text{rGL}}(a)
   &= \pi_{\text{GL}}(a)
        \cdot \exp\Big(-\lambda^*\int_{[a,\infty)}\frac1{b} \gamma^*(\d b)\Big),
\end{align*}
and the prediction $\overline v_{\text{rGL}}$ for the speed is as in~\eqref{GLh}, now with $\pi_{\text{rGL}}$ in place of $\pi_{\text{GL}}$.

While Figure~\ref{fig-GLh-sim} confirms that the rGLh gives a generally much more accurate estimate than the GLh, in most cases it \emph{under}estimates the speed of adaptation.
Indeed there exist instances of configurations where out of three consecutive mutations, the first and the last one contribute to the eventual increase of the population fitness and the middle one does not, in spite of the fact that only the middle one would be retained according to the refined Gerrish--Lenski heuristics, see Figure~\ref{threemutants} for an example. This may (at least partially) explain this underestimation.
It is conceivable that a more thorough analysis of the ``clusters of trajectories'' addressed in Remark~\ref{FalongL}, which takes into account also higher order interactions than the rGLh, leads to a further refinement of the Gerrish--Lenski heuristics.

\begin{figure}\label{fig-kickevenmore}
\scalebox{0.85}{
\begin{tikzpicture}
  \draw[scale=3.7, black, ->] (0.9, 0) -- (5.05, 0) node[right] {time};
  \draw[scale=3.7, black, ->] (0.9, -0.1) -- (0.9, 1.2) node[above] {};
  \draw[scale=3.7,black,thick] (0.87,1) node[left] {$1$};
    \draw[scale=3.7,black,thick] (0.87,0) node[left] {$0$};
  \draw[scale=3.7, blue,thick] (1,-0.1) node[below] {$ T_1$};
  \draw[scale=3.7, blue,thick] (1.5,0.5) node[above] {$a$};
  \draw[scale=3.7, red,thick] (1.8,-0.1) node[below] {$ T_2$};
   \draw[scale=3.7, red,thick] (1.85,0.1) node[above] {$b$};
    \draw[scale=3.7, red,thick] (2.83,0.75) node[above] {$b-a$};
       \draw[scale=3.7, black!40!green,thick] (2.4,0.1) node[above] {$c$};
           \draw[scale=3.7, black!40!green,thick] (3.65,0.65) node[above] {$a+c-b$};
  \draw[scale=3.7, blue,thick] (2.1,-0.08) node[below] {$ T_1+\smfrac{1}{a}$};
    %\draw[scale=3.7, green] (2.3,-0.4) node[below] {$ T_3$};
    \draw[scale=3.7, black!40!green,thick] (2.38,-0.1) node[below] {$ T_3$};
    \draw[scale=3.7, red,thick] (2.62,-0.08) node[below] {$T_2+\smfrac{1}{b}$};
 \draw[scale=3.7, red,thick] (3.4,-0.1) node[below] {$t$};
    \draw[scale=3.7, black!40!green,thick] (3.95,-0.08) node[below] {$L_1$};
      \draw[scale=3.7, domain=0.9:2, smooth, variable=\x, black,thick] plot ({\x}, {1});
  \draw[scale=3.7, domain=1:2, smooth, variable=\x, blue,thick] plot ({\x}, {\x-1});
  \draw[scale=3.7, domain=1.8:2, smooth, variable=\x, red,thick] plot ({\x}, {1.5*(\x-1.8)});
   \draw[scale=3.7, domain=2:2.46, dashed, variable=\x, black,thick] plot ({\x}, {0.99});
   \draw[scale=3.7, domain=2.48:3.13, dashed, variable=\x, black,thick] plot ({\x}, {1-1.5*(\x-2.47)});
      \draw[scale=3.7, domain=2.47:5, dashed, variable=\x, red,thick] plot ({\x}, {0.99});
   \draw[scale=3.7, domain=2:2.46, dashed, variable=\x, red,thick] plot ({\x}, {1.5*(\x-1.8)});
  \draw[scale=3.7, domain=2:3, smooth, variable=\x, black,thick] plot ({\x}, {-1*(\x-2)+1});
   \draw[scale=3.7, domain=2:3.4, smooth, variable=\x, red,thick] plot ({\x}, {0.5*(\x-2)+0.3});
    \draw[scale=3.7, domain=2:3.46, smooth, variable=\x, blue,thick] plot ({\x}, {1});
     \draw[scale=3.7, domain=2.35:3.4, smooth, variable=\x, black!40!green,thick] plot ({\x}, {0.8*(\x-2.35)}); 
     \draw[scale=3.7, domain=2.47:2.605, dashed, variable=\x, black!40!green,thick] plot ({\x}, {-0.7*(\x-2.47)+0.096}); 
          \draw[scale=3.7, domain=3.4:3.93, smooth, variable=\x, black!40!green,thick] plot ({\x}, {0.3*(\x-3.4)+0.84}); 
%     
      %    \draw[scale=3.7, domain=2.47:2.87, dashed, variable=\x, green] plot ({\x}, {0.17-(0.5*(\x-2.47))}); 
 %\draw[scale=3.7, domain=2.47:4.1, smooth, variable=\x, green] plot ({\x}, {0.5*(\x-2.47)+0.17}); 
  \draw[scale=3.7, domain=3.4:3.93, smooth, variable=\x, blue,thick] plot ({\x}, {-0.5*(\x-3.4)+1});
   \draw[scale=3.7, domain=3.93:4.85, smooth, variable=\x, blue,thick] plot ({\x}, {-0.8*(\x-3.93)+0.735});
    %\draw[scale=3.7, domain=5.2:5.4, dashed, variable=\x, blue] plot ({\x}, {-0.5*(\x-3.4)+1});
    \draw[scale=3.7, domain=3.4:3.93, smooth, variable=\x, red,thick] plot ({\x}, {1});
     % \draw[scale=3.7, domain=2.47:6.2, dashed, variable=\x, red] plot ({\x}, {0.99});
       %  \draw[scale=3.7, domain=2:2.47, dashed, variable=\x, black] plot ({\x}, {0.99});
        \draw[scale=3.7, domain=3.93:5, smooth, variable=\x, black!40!green,thick] plot ({\x}, {1});
            \draw[scale=3.7, domain=3.93:5, smooth, variable=\x, red,thick] plot ({\x}, {-0.3*(\x-3.93)+1});
        \draw[scale=3.7, black] (1.6,1.18) node[below] {$0$};
        \draw[scale=3.7, blue] (2.7,1.18) node[below] {$a$};
        \draw[scale=3.7, red] (3.7,1.18) node[below] {$b$};
        \draw[scale=3.7, black!40!green] (4.3,1.18) node[below] {$a+c$};
\end{tikzpicture}
}
\caption{A realisation of the PIT in which the second of three consecutive contending mutations (path shown by red solid line) does not contribute to the eventual increase in population fitness -- i.e.\ its increment $b$ is not contained in the final fitness $a+c$, which is composed of only the increments of the first (blue) and third (green) mutation.
In contrast to this,
the refined Gerrish--Lenski heuristics (see Sec.~\ref{speedheur}) would not take into account the first kink of the red trajectory (since $a<b$) and would rather see the first trajectory as being killed by the second one (when used for determining whether the first one is retained). Hence, this heuristics would  predict a final fitness of $b$. 
This would lead to the continuation of the second trajectory by the dashed red line, and thus also to a killing of the third trajectory according to the rGLh. Above the height line 1 we display the values of the resident fitness $F(t)=M_{\rho(t)}$ during each residency interval.
}
\label{threemutants}
\end{figure}

\section{Possible model extensions}\label{modext}

\subsection{General type space}\label{sec:THE-genar}

  Instead of understanding a type in terms of its fitness and time of arrival, one could think of types in a more abstract manner, i.e.\ as elements of a (measurable) \emph{type space} $(\Theta,\cA)$. Mutation occurring in an individual $i$ would then assign a new (random) type $\vartheta_i$ to it, distributed as $\mu(\vartheta_j,\cdot)$, where $\vartheta_j$ is type of parent individual $j$ and $\mu:\Theta\times\cA\to[0,1]$ is a probability kernel. Then, between mutations, the evolution of the  clonal subpopulations in the corresponding generalized Moran model could be described by the generator
  \[
    Lf(x)
      = \frac1N\sum_{i\neq j}x_ix_j(1+c(\vartheta_i,\vartheta_j)^+)(f(x+e_i-e_j)-f(x)),
  \]
  where $c\in\R^{\Theta\times\Theta}$ can be viewed as a \emph{competition matrix}.
  (Note that taking $\Theta=[0,\infty)$, $\mu(\theta,\cdot)=\delta_\theta\ast\gamma$ and
  $c(\theta,\vartheta)=\theta-\vartheta$ recovers the Moran model in Section~\ref{sec-model}.)
  It is conceivable that this generalized model might be used to incorporate \emph{slowdown effects} that produce strict concavity in population fitness as observed in the Lenski
  experiment (see Fig. 2 in \cite{WRL13}).

  We postulate that with similar methods based on Lemmas~\ref{lem:multitype-without-kinks} and~\ref{lem:multitype-sweep} one should arrive at a corresponding scaling limit result -- possibly even when allowing $c$ to vary over time.
  However, the coupling used for the quenched convergence result would have to become much more involved. Also, in the limiting system new challenges might arise, such as cyclic effects
  providing infinitely many resident changes from finitely many mutations, possibly even in
  finite time; similarly to \cite[Examples 3.2, 3.5 and 3.6]{BPT23} and \cite[Example 3.6]{CKS21}; Figure~\ref{figure-cyclicity} for an illustration. We defer more detailed discussions to future work.

\begin{figure}
\begin{tikzpicture}[scale=2]

  % axes
    \draw[->,black] (0, 0) -- (1.6, 0) node[right] {$t$};
    \draw[->,black] (0, 0) -- (0, 1.05) node[above] {};

  % marks on y-axis
    \draw[black] (-0.03,1) node[left] {\footnotesize $1$};
    \draw[black] (-0.03,0) node[left] {\footnotesize $0$};
 
  % marks on x-axis
    \draw[black] (1,-0.03) node[below] {\footnotesize $1$};

  % black mutant
    \draw[domain=0:0.5, variable=\x, black,thick] plot ( {\x}, {1-\x} );
    \draw[domain=0.5:1, variable=\x, black,thick] plot ( {\x}, {0.5 + (\x-0.5)} );
    \draw[domain=1:1.5, variable=\x, black,thick] plot ( {\x}, {1} );
    \draw[domain=1.5:1.55, variable=\x, black,thick] plot ( {\x}, {1 - (\x-1.5)} );
    \draw[domain=1.55:1.6, variable=\x,dotted, black,thick] plot ( {\x}, {1 - (\x-1.5)} );

  % brown mutant    
    \draw[domain=0:0.5, variable=\x,brown,thick] plot ( {\x}, {1} );
    \draw[domain=0.5:1, variable=\x,brown,thick] plot ( {\x}, {1 - (\x-0.5)} );
    \draw[domain=1:1.5, variable=\x,brown,thick] plot ( {\x}, {0.5 + (\x-1)} );
    \draw[domain=1.5:1.55, variable=\x,brown,thick] plot ( {\x}, {1} );
    \draw[domain=1.55:1.6, variable=\x,brown,dotted] plot ( {\x}, {1} );

  % blue mutant
    \draw[domain=0:0.5, variable=\x,blue,thick] plot ( {\x}, {0.5+\x} );
    \draw[domain=0.5:1, variable=\x,blue,thick] plot ( {\x}, {1} );
    \draw[domain=1:1.5, variable=\x,blue,thick] plot ( {\x}, {1 - (\x-1)} );
    \draw[domain=1.5:1.55, variable=\x,blue,thick] plot ( {\x}, {0.5 + (\x-1.5)} );
    \draw[domain=1.55:1.6, variable=\x,blue,dotted] plot ( {\x}, {0.5 + (\x-1.5)} );

\end{tikzpicture}
\begin{tikzpicture}[scale=2]

  % axes
    \draw[->,black] (0, 0) -- (2.1, 0) node[right] {$t$};
    \draw[->,black] (0, 0) -- (0, 1.05) node[above] {};

  % marks on y-axis
    \draw[black] (-0.03,1) node[left] {\footnotesize $1$};
    \draw[black] (-0.03,0) node[left] {\footnotesize $0$};
 
  % marks on x-axis
    \draw[black] (1,-0.03) node[below] {\footnotesize $1$};
    \draw[black] (2,-0.03) node[below] {\footnotesize $2$};

  % black mutant
    \draw[domain=0:0.5, variable=\x, black,thick] plot ( {\x}, {1-0.5*\x} );
    \draw[domain=0.5:(2/3), variable=\x, black,thick] plot ( {\x}, {0.75 + 1.5*(\x-0.5)} );
    \draw[domain=(2/3):1, variable=\x, black,thick] plot ( {\x}, {1} );
    \draw[domain=1:1.5, variable=\x, black,thick] plot ( {\x}, {1-0.5*(\x-1)} );
    \draw[domain=1.5:(5/3), variable=\x, black,thick] plot ( {\x}, {0.75 + 1.5*(\x-1.5)} );
    \draw[domain=(5/3):2, variable=\x, black,thick] plot ( {\x}, {1} );
    \draw[domain=2:2.05, variable=\x, black,thick] plot ( {\x}, {1-0.5*(\x-2)} );
    \draw[domain=2.05:2.1, variable=\x, dotted, black,thick] plot ( {\x}, {1-0.5*(\x-2)} );

  % brown mutant    
    \draw[domain=0:0.5, variable=\x,brown,thick] plot ( {\x}, {1} );
    \draw[domain=0.5:(2/3), variable=\x,brown,thick] plot ( {\x}, {1 - (\x-0.5)} );
    \draw[domain=(2/3):1, variable=\x,brown,thick] plot ( {\x}, {5/6 + 0.5*(\x-2/3)} );
    \draw[domain=1:1.5, variable=\x,brown,thick] plot ( {\x}, {1} );
    \draw[domain=1.5:(5/3), variable=\x,brown,thick] plot ( {\x}, {1 - (\x-1.5)} );
    \draw[domain=(5/3):2, variable=\x,brown,thick] plot ( {\x}, {5/6 + 0.5*(\x-5/3)} );
    \draw[domain=2:2.05, variable=\x,brown,thick] plot ( {\x}, {1} );
    \draw[domain=2.05:2.1, variable=\x,brown,dotted] plot ( {\x}, {1} );

  % blue mutant
    \draw[domain=0:0.5, variable=\x,blue,thick] plot ( {\x}, {0.5+\x} );
    \draw[domain=0.5:(2/3), variable=\x,blue,thick] plot ( {\x}, {1} );
    \draw[domain=(2/3):1, variable=\x,blue,thick] plot ( {\x}, {1 - 1.5*(\x-(2/3))} );
    \draw[domain=1:1.5, variable=\x,blue,thick] plot ( {\x}, {0.5+(\x-1)} );
    \draw[domain=1.5:(5/3), variable=\x,blue,thick] plot ( {\x}, {1} );
    \draw[domain=(5/3):2, variable=\x,blue,thick] plot ( {\x}, {1 - 1.5*(\x-(5/3))} );
    \draw[domain=2:2.05, variable=\x,blue,thick] plot ( {\x}, {0.5+(\x-2)} );
    \draw[domain=2.05:2.1, variable=\x,blue,dotted] plot ( {\x}, {0.5+(\x-2)} );

\end{tikzpicture}
\begin{tikzpicture}[scale=2]

  % axes
    \draw[->,black] (0, 0) -- (2.1, 0) node[right] {$t$};
    \draw[->,black] (0, 0) -- (0, 1.05) node[above] {};

  % marks on y-axis
    \draw[black] (-0.03,1) node[left] {\footnotesize $1$};
    \draw[black] (-0.03,0) node[left] {\footnotesize $0$};
 
  % marks on x-axis
    \draw[black] (1,-0.03) node[below] {\footnotesize $1$};
    \draw[black] (2,-0.03) node[below] {\footnotesize $2$};

  % black mutant
    \draw[domain=0:0.5, variable=\x, black,thick] plot ( {\x}, {1-\x} );
    \draw[domain=0.5:1, variable=\x, black,thick] plot ( {\x}, {0.5 + (\x-0.5)} );

    \draw[domain=1:1.25, variable=\x, black,thick] plot ( {\x}, {1} );
    \draw[domain=1.25:1.5, variable=\x, black,thick] plot ( {\x}, {1 - 2*(\x-1.25)} );

    \draw[domain=1.5:1.625, variable=\x, black,thick] plot ( {\x}, {0.5 + 4*(\x-1.5)} );
    \draw[domain=1.625:1.75, variable=\x, black,thick] plot ( {\x}, {1} );

    \draw[domain=1.75:1.8125, variable=\x, black,thick] plot ( {\x}, {1 - 8*(\x-1.75)} );
    \draw[domain=1.8125:1.875, variable=\x, black,thick] plot ( {\x}, {0.5 + 8*(\x-1.8125)} );

    \draw[domain=1.875:1.90625, variable=\x, black,thick] plot ( {\x}, {1} );
    \draw[domain=1.90625:1.9375, variable=\x, black,thick] plot ( {\x}, {1 - 16*(\x-1.90625)} );

  % brown mutant    
    \draw[domain=0:0.5, variable=\x,brown,thick] plot ( {\x}, {1} );
    \draw[domain=0.5:1, variable=\x,brown,thick] plot ( {\x}, {1 - (\x-0.5)} );

    \draw[domain=1:1.25, variable=\x,brown,thick] plot ( {\x}, {0.5 + 2*(\x-1)} );
    \draw[domain=1.25:1.5, variable=\x,brown,thick] plot ( {\x}, {1} );

    \draw[domain=1.5:1.625, variable=\x,brown,thick] plot ( {\x}, {1 - 4*(\x-1.5)} );
    \draw[domain=1.625:1.75, variable=\x,brown,thick] plot ( {\x}, {0.5 + 4*(\x-1.625)} );

    \draw[domain=1.75:1.8125, variable=\x,brown,thick] plot ( {\x}, {1} );
    \draw[domain=1.8125:1.875, variable=\x,brown,thick] plot ( {\x}, {1 - 8*(\x-1.8125)} );

    \draw[domain=1.875:1.90625, variable=\x,brown,thick] plot ( {\x}, {0.5 + 16*(\x-1.875)} );
    \draw[domain=1.90625:1.9375, variable=\x,brown,thick] plot ( {\x}, {1} );

  % blue mutant
    \draw[domain=0:0.5, variable=\x,blue,thick] plot ( {\x}, {0.5+\x} );
    \draw[domain=0.5:1, variable=\x,blue,thick] plot ( {\x}, {1} );

    \draw[domain=1:1.25, variable=\x,blue,thick] plot ( {\x}, {1 - 2*(\x-1)} );
    \draw[domain=1.25:1.5, variable=\x,blue,thick] plot ( {\x}, {0.5 + 2*(\x-1.25)} );

    \draw[domain=1.5:1.625, variable=\x,blue,thick] plot ( {\x}, {1} );
    \draw[domain=1.625:1.75, variable=\x,blue,thick] plot ( {\x}, {1 - 4*(\x-1.625)} );

    \draw[domain=1.75:1.8125, variable=\x,blue,thick] plot ( {\x}, {0.5 + 8*(\x-1.75)} );
    \draw[domain=1.8125:1.875, variable=\x,blue,thick] plot ( {\x}, {1} );

    \draw[domain=1.875:1.90625, variable=\x,blue,thick] plot ( {\x}, {1 - 16*(\x-1.875)} );
    \draw[domain=1.90625:1.9375, variable=\x,blue,thick] plot ( {\x}, {0.5 + 16*(\x-1.90625)} );

  % accummulation lines
    \draw[dotted] (1.9375,1) -- (2,1);
    \draw[dotted] (1.9375,0.5) -- (2,0.5);
    \draw[dotted] (2,0) -- (2,1);

\end{tikzpicture}
\caption{
Illustration of cyclicity. In each case, $\Theta=\{0,{\color{brown}1},{\color{blue}2}\}$,
$h_0(0)=1, h_{\color{brown}1}(0)=1$, $h_{\color{blue}2}(0)=\frac12$.
Left: $c({\color{brown}1},0)=c({\color{blue}2},{\color{brown}1})=c(0,{\color{blue}2})=1$.
Mid: $c({\color{brown}1},0)=0.5,c({\color{blue}2},{\color{brown}1})=1,c(0,{\color{blue}2})=1.5$.
Right: $c(t,{\color{brown}1},0)=c(t,{\color{blue}2},{\color{brown}1})=c(t,0,{\color{blue}2})=2^{-\lfloor\log_2(2-t)\rfloor}\1_{\{t<2\}}$.
}\label{figure-cyclicity}
\end{figure}

\subsection{Moderate and nearly strong selection}\label{sec-weakerselection} 

\begin{figure}
\includegraphics[width=4.7cm,trim={0.5cm, 1cm, 0.8cm, 2cm},clip]{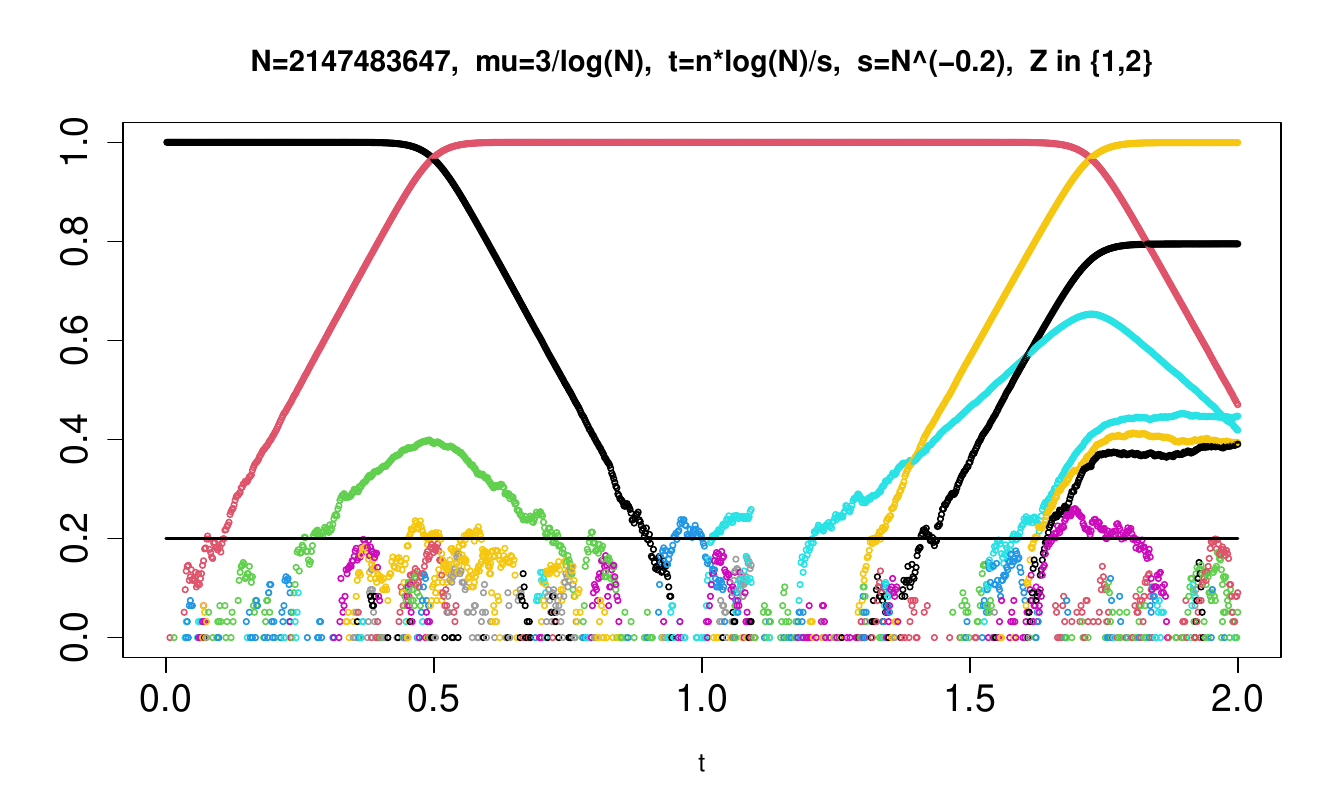}
\includegraphics[width=4.7cm,trim={0.5cm, 1cm, 0.8cm, 2cm},clip]{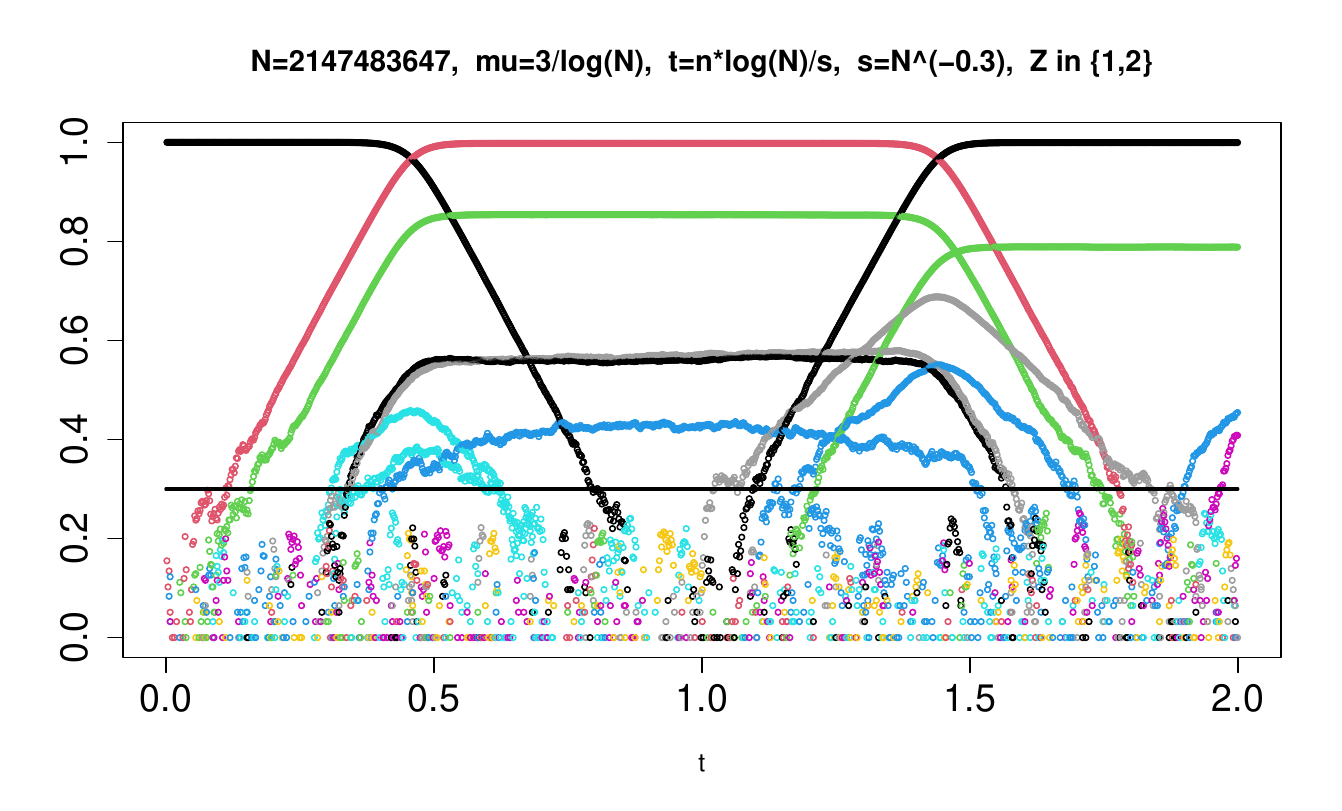}
\includegraphics[width=4.7cm,trim={0.5cm, 1cm, 0.8cm, 2cm},clip]{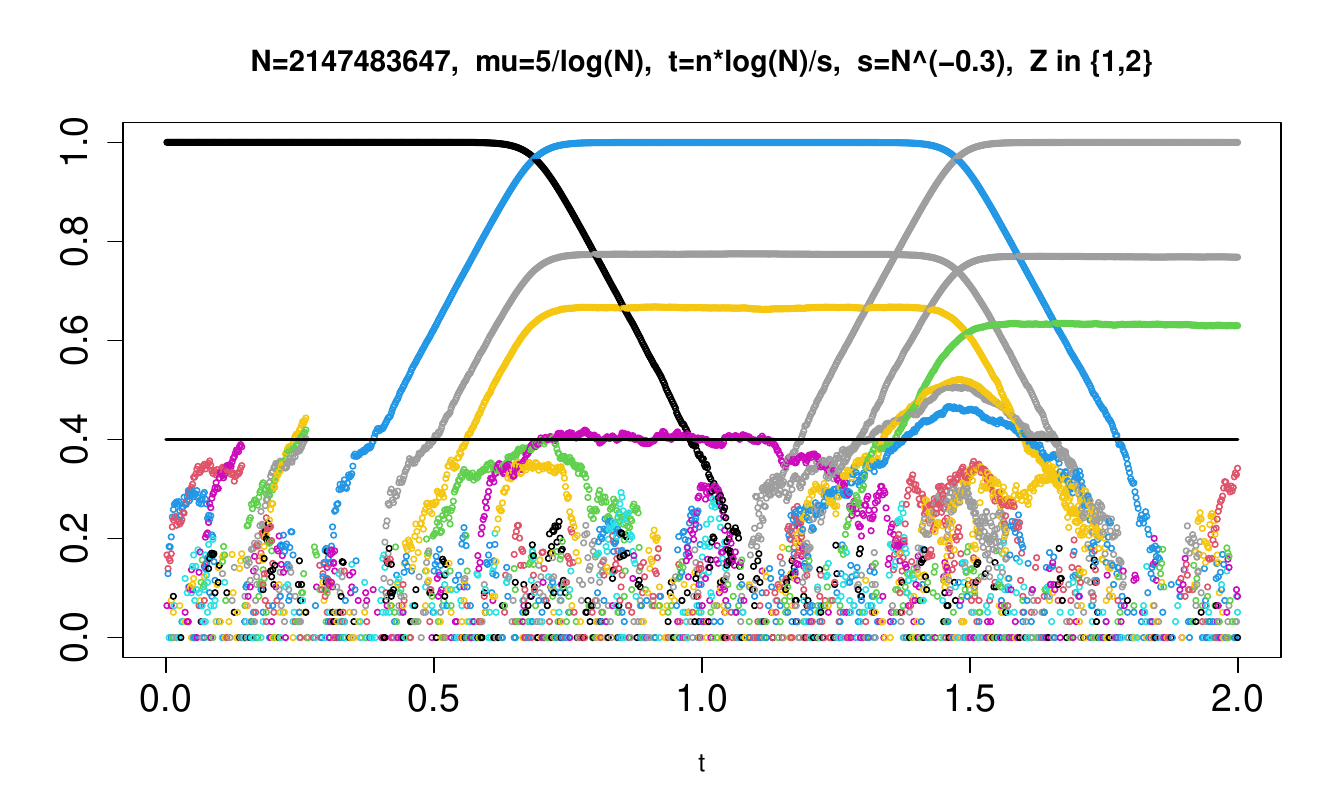}
\caption{\label{fig:moderate-selection}
  Simulation of a Cannings model with mutation and selection, $N\approx2.1\cdot10^{9}$, under moderate selection, i.e.\ $s_N=N^{-b}$, on timescale $N^b\log N$.
  Left: $b=0.2$, middle: $b=0.3$, right: $b=0.4$, each indicated by the horizontal line.
}
\end{figure}

In mathematical population genetics \emph{weak selection} classically refers to the scaling regime where fitness increments are of the order of~$1/N$.
In contrast, as we already mentioned, \emph{strong selection} means that $s_N=s$ does not scale with $N$. In the latter regime, in the proof of  Theorem~\ref{theorem-THE}, we exploited that
the frequency of all mutations, including non-contending ones, arising in finite time
stays finite in the scaling limit. This is never true for $s_N \to 0$ as $N\to \infty$; then only an asymptotically vanishing fraction $\Theta(s_N)$ of mutants survives drift. This makes the analysis more involved since 
then the supercriticality of the branching processes that approximate the clonal subpopulations tends to $0$ as $N\to \infty$.

As can be seen from the formula for the fixation probability of a single mutant with selective advantage $s_N$ in the Moran model (see e.g. \cite[Theorem 6.1]{durrett2008probability}), the probability that such a mutant becomes a contender should be  of order $s_N$ provided that $\frac 1N \ll s_N \ll 1$. Thus a mutation rate $1/\log N$ will lead to the arrival of finitely many contending mutations in a time interval of length $s_N^{-1}\log N$; this is also the (order of) time that a contender which ever becomes resident takes to reach residency. 

One interesting regime is that of~\emph{moderate selection}, where $s_N \asymp N^{-b}$ for some $0<b<1$. Recent results show that Haldane's formula for the probability that a mutation becomes contending applies not only in Moran models but also in Cannings models for the case of \emph{moderately weak} selection ($s_N \asymp N^{-b}$, \mbox{$\tfrac 12 <b<1$)} (\cite{B21ii}) as well as for the case of \emph{moderately strong} selection ($s_N \asymp N^{-b}$, \mbox{$0<b<\tfrac 12$)} (\cite{B21i}), and thus the probability that a given mutation becomes contending is of order $s_N$. Simulations based on a Cannings model (see Figure~\ref{fig:moderate-selection}) indicate that moderate selection yields a similar limiting process as the PIT, however, piecewise linear trajectories now start and end at height $b$ instead of $0$. That is, we conjecture that contending mutant subpopulations reach size $N^b$ in $o(s_N^{-1}\log N)$ time and decaying subpopulations of size $o(N^b)$ go extinct in $o(s_N^{-1}\log N)$ time.

The regime that is intermediate between moderate and strong selection, where $s_N$ tends to zero as a slowly varying function of $N$, is also interesting to study. 
We call this the regime of \emph{nearly strong selection}.  Here the limiting process of the logarithmic clonal subpopulation sizes on that timescale should again be similar to the PIT, with trajectories of the contending mutants born at height 0.

We defer the precise investigation of these regimes to future work.

\begin{appendix}

\section{Supercritical branching}
\label{s:superbranch}

In this section we consider a continuous-time binary Galton--Watson process $Z=(Z_t)_{t \geq 0}$ with individual birth and death rates $b\geq0$, $d\geq0$ respectively, satisfying $s:= b-d>0$. We denote by $\P_z$ the law of $Z$ started at $z$, 
and we abbreviate $\{Z\not\to 0\} := \{Z_t \geq 1, \forall t \geq 0\}$, $\{Z \to 0\} := \{Z \not \to 0\}^c$.

\begin{theorem}\label{thm:superbranch}
Let $Z$ be as above. Then:
\begin{enumerate}
\item[1)] For $z \in \N$, $\P_z(Z \not \to 0) = 1- (d/b)^z \in (0,1]$.

\item[2)]
The family of random variables $(\Xi_z)_{z \in \N}$ given by
\[
  \Xi_z
   := \sup_{t \geq 0} \big| \log^+(Z_{t}) - (\log z + s t) \1_{\{Z \not \to 0\}} \big|
   \quad \text{under} \quad \P_{z}
\]
is tight. In particular, $\P_z(\Xi_z<\infty)=1$ for all $z \in \N$.
Moreover, there exists a constant $C \in (0,\infty)$ such that
$
\lim_{z \to \infty} \P_z( \Xi_z > C) = 0.
$

\item[3)] Let $T_0 := \inf\{t \geq 0 \mid Z_t=0\}$ and, for $L > 0$, $T_L := \inf\{t \geq 0 \mid Z_t \geq L\}$.
Let $z_L \in \N$. Then:
\begin{enumerate}
\item[a)]If $t_L \to \infty$ then $\P_{z_L}(T_0 \geq t_L) \sim \P_{z_L}(Z \not \to 0)$ as $L \to \infty$.
\item[b)] $\P_{z_L}(T_L < \infty) \sim \P_{z_L}(Z \not \to 0)$ as $L \to \infty$.
\item[c)] If $z_L = o(L)$ then $\lim_{L \to \infty} \frac{T_L}{\log(L/z_L)}=\frac 1s$ in probability under $\P_{z_L}(\cdot|Z \not \to 0)$.\\
If $z_L \equiv 1$ this also holds almost surely. 
\end{enumerate}
\end{enumerate}
\end{theorem} 

\begin{proof}
When $z=1$, 1) follows from \cite[Theorem~3.4.1]{athreya1972branching},
and 3b) follows from the fact that $\{Z \to 0\} = \cup_{L \geq 1} \{T_L = \infty\}$ almost surely. It is also well-known that $(Z_t e^{-st})_{t\geq 0}$ is a martingale that almost-surely 
converges to a random variable $W$ such that $\{W>0\}=\{Z\not\to0\}$. 

For $z \geq 1$, $Z$ has the same distribution as $\sum_{k=1}^{z} Z^{(k)}$ where $Z^{(k)}$ are i.i.d.\ and distributed as $Z$ under $\P_1$, see e.g.\ \cite[Eq.~(10) on p.~105]{athreya1972branching}.  
So $\P_z(Z \not \to 0) = 1-[1-\P_1(Z \not \to 0)]^z$, implying 1). 
For 3a) and 3b), it is enough to assume that $z_L \to z \in \N$ or $z_L \to \infty$.
If $z_L \to z$, note for 3b) that
\[
\bigcap_{k=1}^{z_L} \{T^{(k)}_{L/z_L} = \infty\}\subset \{T_L = \infty \} \subset\{Z \to 0\}
\quad \text{ almost surely,}
\]
where $T^{(k)}_L$ is the analogue of $T_L$ for $Z^{(k)}$,
and $\P_1(T_{L/z_L}=\infty)^{z_L} \to (d/b)^z$ by the case $z=1$.
For 3a), note that $T_0 = \max_{1 \leq k \leq z_L} T^{(k)}_0$, so the family of distributions of $T_0$ under $\P_{z_L}(\cdot \mid Z \to 0)$ is tight. 
Thus
\[
\P_{z_L}(T_0 > t_L) = \P_{z_L}(T_0 > t_L, Z \to 0) + \P_{z_L}(Z \not \to 0) \sim \P_{z_L}(Z \not \to 0)
\]
since the first term after the equality converges to $0$ and the second is bounded away from $0$.

If $z_L \to \infty$, note for 3a) and 3b) that both $\P_{z_L}(T_L < \infty)$
and $\P_{z_L}(T_0>t_L)$ are not smaller than $\P_{z_L}(Z \not \to 0)$ which converges to $1$ as $L\to \infty$.

Let us next show 2).
On $\{Z\to 0\}$, $\log^+( \sup_{t \geq 0} Z_{t})$ is almost surely bounded. 
On \mbox{$\{Z\not\to0\}$}, 
\begin{equation}
\label{e:prlemBranchSupercrt1}
  \log Z_t - \log z - s t
    =  \log\bigg( \frac{1}{z} \sum_{k=1}^{z} Z^{(k)}_t e^{-st} \bigg).
\end{equation}
Now note that $\{Z\not \to 0\}= \bigcup_{k=1}^z\{Z^{(k)} \not \to 0\}$ and that, on $\{Z^{(k)} \not \to 0\}$, 
$Z^{(k)}_t e^{-st}$ almost surely is positive, c\`adl\`ag,
has positive left limits and converges to a positive limit. 
Hence $0<\inf_{t \geq 0}Z^{(k)}_te^{-st} \leq \sup_{t \geq 0} Z^{(k)}_te^{-st}<\infty$
on $\{Z^{(k)} \not \to 0\}$, and on $\{Z^{(k)} \to 0\}$ the last inequality also clearly holds.
This shows that $\Xi_z$ is almost-surely finite (and hence tight) for each $z \in \N$.
To finish the proof of 2), it is enough to obtain the constant $C>0$ mentioned therein.
But since the summands inside the last $\log$ in \eqref{e:prlemBranchSupercrt1} are i.i.d.,
it will be provided by the strong law of large numbers once we show that
\begin{equation}\label{e:prlemBranchSupercrt2}
  0 < \E_1\Big[\inf_{t \geq 0} Z_t e^{-st} \Big]
  \quad \text{ and } \quad
  \E_1\Big[\sup_{t \geq 0} Z_t e^{-st} \Big] < \infty.
\end{equation}
The first inequality follows from $\P_1(\inf_{t \geq 0} Z_t e^{-st} > 0)\geq \P_1(W>0)>0$. 
For the second, note that the martingale $Z_t e^{-st}$ is bounded in $L^2$ (see \cite[Eq.(5), p.~109]{athreya1972branching}),
so it follows from Doob's $L^2$ inequality (see e.g.\ \cite[Theorem~II.1.7]{RY99}).
Finally, 3c) follows from 2).
\end{proof}

%%%%%%%%%%%%%%%%%%%%%%%%%%%%%%%%%%%%%%%%%%%%%%%%%%%%%%%%%%%%%%%%%%%%%%%%%%%%%%%%%%%%%%%%%%%%%%%%%%%%%%%%%%%%%%%%%%%%%
\section{Stochastic domination}
\label{s:stochdom}
In this section we provide the couplings required in the proofs of Lemmas~\ref{lem:linearbounds},\ref{lem:multitype-without-kinks} and \ref{lem:multitype-sweep}, combining results from \cite{KKO77} and \cite{Ma87}. Since we feel that these arguments are of independent interest,  we state and prove, for two Markov chains $X$, $Y$ in continuous time, a comparison result in terms of an ordered coupling  between~$Y$ and the mapped process $(\varphi(X_t))_{t\ge 0}$ under the assumption of a ``monotone intertwining'' of $\varphi$ and the jump rates of $X$ and $Y$.

Specifically, let $E, F$ be countable sets, $F$ equipped with a partial order $\leq$, and let $\varphi:E\to F$. 
Let $X=(X_t)_{t\geq 0}$, $Y=(Y_t)_{t\geq 0}$ be continuous-time c\`adl\`ag Markov jump processes on $E, F$ with bounded generators $A, B$, respectively.
Here we will say that $Y$ is \emph{monotone} if, for any bounded non-decreasing $g:F\to\R$, $Bg$ is also non-decreasing.
We will write $\P^X_x$ for the law of $X$ started from $x$, $\E^X_x$ for the corresponding expectation, and analogously for $Y$.

\begin{theorem}\label{thm:stochdom}
Assume that $Y$ is monotone and that, for all bounded non-decreasing $g:F \to \R$,
\begin{equation}\label{e:stochdomcond}
A (g \circ \varphi)(x) \leq B g( \varphi(x)) \quad \forall x \in E_0
\end{equation}
where $E_0 \subset E$. Denote by $\tau_0 := \inf \{ t \geq 0 \colon\, X_t \notin E_0\}$
the first time when $X$ exits $E_0$.
Then, for all $x \in E_0$ and $y \in F$ with $\varphi(x) \leq y$,
there exists a coupling $Q$ of $(\varphi(X_t))_{t \geq 0}$ 
under $\P^X_x$ and of $Y$ under $\P^Y_y$ such that
$
Q(\varphi(X_t) \leq Y_t \,\, \forall t \in [0,\tau_0]) = 1,
$
where we interpret $[0,\infty]=[0,\infty)$.
The analogous result holds with the inequalities reversed.
\end{theorem}

\begin{proof}
We will only prove the theorem for the inequalities as first stated; 
the proof for the reversed inequalities is analogous.

Let us first reduce to the case $E_0=E$. 
If $E_0 \subsetneq E$, let $A(x,y)$ denote the matrix entries corresponding to the operator $A$.
If $\varphi$ is not surjective, we enlarge $E$ to $\widehat{E} := E \cup (F \setminus \varphi(E))$ where the union is disjoint,
and extend $\varphi$ to $\widehat{E}$ by setting $\varphi(x)=x$ for $x \notin E$.
Define $\widehat{X}$ to be the Markov jump process on $\widehat{E}$ with generator $\widehat{A}$ 
given by $\widehat{A}(x,y) = A(x,y)\1_E(y)$ if $x \in E_0$, and $\widehat{A}(x,y) = B(\varphi(x), \varphi(y))/\#\varphi^{-1}(\varphi(y))$
otherwise. 
One may verify that 
\eqref{e:stochdomcond} is valid for $\widehat{A}$ in place of $A$ and all $x \in \widehat{E}$,
and it is clear that $X$ and $\widehat{X}$ are equal in distribution up to their first exit of $E_0$. 

From here on we assume $E_0=E$, implying $\tau_0 = \infty$.
In this case, the first step is to use \cite[Theorem~3.5]{Ma87} (with the strong stochastic ordering; see Definition~2.4 therein)
to conclude that, for any $t>0$ and any $x \in E$, $y \in F$ with $\varphi(x) \leq y$,
\begin{equation}
\label{e:prstochdom1}
\begin{aligned}
& \text{
$\varphi(X_t)$ under $\P^X_x$ is stochastically dominated by $Y_t$ under $\P^Y_y$. 
}
\end{aligned}
\end{equation}
First of all, note that our assumptions on $Y$ imply that
its generator $B$ is monotone in the sense discussed in Definition~3.2 in \cite{Ma87}, i.e., 
for any $t>0$ and $y_1 \leq y_2 \in F$,
\begin{equation}
\label{e:prstochdom2}
\text{$Y_t$ under $\P^Y_{y_1}$ is stochastically dominated by $Y_t$ under $\P^Y_{y_2}$.}
\end{equation}
Indeed, this follows from \cite[Theorem~2.2]{Lig85} 
and the fact that the semigroup for $Y$, $\exp(tB)$, has e.g.\ the representation given right before Definition~3.2 in \cite{Ma87}.
To verify that our assumptions imply those of Theorem~3.5 in \cite{Ma87},
note first that $f$, $E'$, $E$ therein correspond to our $\varphi$, $E$, $F$, respectively.
Then note that the mapping $\Phi(\varphi)$ from $\ell_1(E)$ to $\ell_1(F)$ defined before Theorem~3.5 acts by multiplication to the left. 
Its adjoint mapping of multiplication to the right (from $\ell_\infty(F)$ to $\ell_\infty(E)$)
is defined such that $u \Phi(\varphi) \cdot v = u \cdot \Phi(\varphi)v$, i.e.,  $\Phi(\varphi)v(x) := \mathrm{e}_x \Phi(\varphi) \cdot v = v(\varphi(x))$,
where $\mathrm{e}_x$ is the indicator function of $\{x\}$, $x \in E$.
Finally, note that, according to the ordering $\leq_{\rm{st}}$ (cf.\ Definition~2.4 and Proposition~3.1 therein),
$A \Phi(\varphi) \leq_{\rm{st}} \Phi(\varphi) B$ if and only if $A \Phi(\varphi)\1_\Gamma(x) \leq \Phi(\varphi) B \1_\Gamma(x)$ for all $x \in E$ and all 
increasing sets $\Gamma \subset F$; since in this case $\1_\Gamma$ is non-decreasing, this follows from \eqref{e:stochdomcond} (and is actually equivalent to it). 

To finish the proof, we will verify the conditions of \cite[Theorem~4]{KKO77}.
We write $Z=(Z_t)_{t \geq 0}$ with $Z_t:= \varphi(X_t)$.
For $n \geq 2$, $t^n=(t_1, \ldots, t_n) \in [0,\infty)^n$ with $t_1<\cdots<t_n$ and
$z^{n-1} = (z_1,\ldots, z_{n-1}) \in F^{n-1}$,
define the kernel
\[
p_{t^n}(z^{n-1}, B) = \P^X_x(Z_{t_n} \in B \mid \cap_{i=1}^{n-1} \{Z_{t_i} = z_i\} ), \quad B\subset F,
\]
and let $q_{t^n}(y^{n-1}, B)$ denote the analogous kernel for $Y$ in place of $Z$.
The conditions of \cite[Theorem~4]{KKO77} will be verified if we show that,
for any $t^n$, $z^{n-1}$, any $y^{n-1}$ with $z_i \leq y_i$ for $1 \leq i \leq n-1$, and any 
non-decreasing $g:F \to \R$, 
\begin{equation}
\label{e:prstochdom3}
\int g(u) p_{t^n}(z^{n-1}, d u) \leq \int g(u) q_{t^n}(y^{n-1}, d u).
\end{equation}
To this end, note first that, since $Y$ is Markovian,
\begin{equation}
\label{e:prstochdom4}
\int g(u) q_{t^n}(y^{n-1}, du) = \E^Y_{y_{n-1}}[g(Y_{s_n})]
\end{equation}
where $s_n := t_n-t_{n-1}$.
On the other hand, by the Markov property,
\[
\begin{aligned}
  \E^X_x\big[ \1_{\cap_{i=1}^{n-1}\{Z_{t_i}=z_i\}} g(Z_{t_n})\big]
    &= \E^X_x\big[\1_{\cap_{i=1}^{n-1}\{Z_{t_i}=z_i\}} \E^X_{X_{t_{n-1}}}[g(Z_{s_n}) ]\big] \\
    &\leq \P^X_x\big(\cap_{i=1}^{n-1}\{Z_{t_i}=z_i\}\big) \E^Y_{z_{n-1}}\big[g(Y_{s_n})\big]\\
    &\leq \P^X_x\big(\cap_{i=1}^{n-1}\{Z_{t_i}=z_i\}\big) \E^Y_{y_{n-1}}\big[g(Y_{s_n})\big],
\end{aligned}
\]
where for the first inequality we used \eqref{e:prstochdom1} at time $s_{n}$
and for the second inequality we used \eqref{e:prstochdom2}.
Together with \eqref{e:prstochdom4}, this shows \eqref{e:prstochdom3}.
To conclude, note that the kernels $p$, $q$ plus the initial states determine all finite-dimensional distributions of $Z$, $Y$, and thus completely characterize their distributions in Skorokhod space (see e.g.\ \cite[Section~14]{Bil68}). Thus the construction in \cite[Theorem~4]{KKO77} provides the desired coupling.
\end{proof}

%%%%%%%%%%%%%%%%%%%%%%%%%%%%%%%%%%%%%%%%%%%%%%%%%%%%%%%%%%%%%%%%%%%%%%%%%%%%%%%%%%%%%%%%%%%%%%%%%%%%%%%%%%%%%%%%%%%%
\section{A functional CLT for renewal reward processes}
\label{s:RenRew}
In this section we provide a functional central limit theorem for renewal reward processes, thus  completing the proof of Theorem~\ref{theorem-speedCLT} that was given in Section~\ref{sec3_4}.

Let $(X_n, \tau_n)$, $n \in \N$, be an i.i.d.\ sequence of $\R\times(0,\infty)$-valued random variables.
We assume that $X_1$ and $\tau_1$ are both square-integrable.
Define
\[
T_n := \tau_1 + \cdots + \tau_n, \qquad S_n = X_1 + \cdots + X_n,
\]
and set
\[
N_t := \sup \{ n \in \N \colon\, T_n \leq t\}, \qquad Z_t := S_{N_t},
\]
where in the above we take $\sup \emptyset = 0$.
By the SLLN for sums of i.i.d.\ random variables,
\[
\lim_{n \to \infty} \frac{T_n}{n} = \E[\tau_1] =: \theta \quad \text{ and } \quad 
\lim_{n \to \infty} \frac{S_n}{n} = \E[X_1] =: \mu \quad \text{almost surely,}
\]
and an interpolation argument shows that (see e.g.\ \cite[Theorems~2.5.10 and 2.5.14]{EKM97}),
\[
\lim_{t \to \infty} \frac{N_t}{t} = \frac{1}{\theta} \quad \text{ and } \quad
\lim_{t \to \infty} \frac{Z_t}{t} = \frac{\mu}{\theta} =: v \quad \text{almost surely.} 
\]
Here we will prove a functional central limit theorem for $Z_t$, as stated next.

\begin{theorem}\label{thm:FCLTRenRew}
Assume that $\sigma := \sqrt{\E[(X_1 - v \tau_1)^2]/\theta} > 0$. Then
\[
\Big(\frac{Z_{nt} - n tv}{\sigma \sqrt{n}} \Big)_{t \geq 0} \quad \overset{d}{\longrightarrow} \quad W
\]
where $W=(W_t)_{t \geq 0}$ is a standard Brownian motion and ``$\overset{d}{\longrightarrow}$'' denotes convergence in distribution as $n \to \infty$ in the space of càdlàg functions from $[0,\infty)$ to $\R$ equipped with the Skorokhod $J_1$-topology.
\end{theorem}

\begin{proof}
We adapt the proof of Theorem~1.4(b) in \cite{HHSST14}.
First note that, by the Donsker--Prokhorov invariance principle (see e.g.\ \cite[Theorem 1.2(c) in Chapter 5]{ethier2009markov}) for sums of i.i.d.\ random variables, 
\[
\widehat{W}^{(n)} = \big(\widehat{W}^{(n)}_t\big)_{t \geq 0}
 := \bigg( \frac{1}{\sigma \sqrt{\theta} \sqrt{n}}\sum_{k=1}^{ \lfloor n t \rfloor}(X_k - v \tau_k) \bigg)_{t \geq 0}
 \overset{d}{\longrightarrow} W.
\]
Consider the random time change $\varphi_n(t) := N_{nt}/n$.
Let us show that
\begin{equation}\label{e:prFCLTRenRew1}
  \lim_{n \to \infty} \sup_{t \in [0,M]} \Big| \varphi_n(t) - \frac{t}{\theta}\Big|
    = 0 \quad \text{in probability for any $M>0$.}
\end{equation}
Indeed, since $T_n > t$ if and only if $N_t < n$,
given $\delta, \varepsilon>0$,
there are $\delta', \varepsilon'> 0$ such that, for large $n$,
\[
\P\bigg(\sup_{t \in [\delta, \infty)} \Big|\frac{\varphi_n(t)}{t} - \frac{1}{\theta} \Big|> \varepsilon \bigg)
\leq \P\big( \exists k \geq \delta'n \colon\, |T_k/k - \theta| \geq \varepsilon' \big) \underset{n \to \infty}{\longrightarrow} 0
\]
by the SLLN for $T_n$. 
On the other hand, taking $\delta < \theta\varepsilon/2$, we obtain $\varepsilon''>0$ such that
\[
  \P\bigg(\sup_{t \in [0, \delta]} \Big|\varphi_n(t) - \frac t\theta \Big|> \varepsilon \bigg)
    \leq \P\Big( \frac{N_{n\delta}}{n \delta} \geq \frac{1}{\theta}+\varepsilon''\Big).
    \numberthis\label{ndelta}
\]
Indeed, for all $t \leq \delta$, $|\varphi_n(t) - t/\theta| \leq \varphi_n(t) + t/\theta \leq N_{n \delta}/n + \delta/\theta$, thus $\sup_{t \in [0, \delta]} |\varphi_n(t) - t/\theta |$ satisfies the same inequality. On the other hand, in the event in the l.h.s.\ of \eqref{ndelta}, 
\[ \sup_{t \in [0, \delta]} \Big|\varphi_n(t) - \frac t\theta \Big| > \eps > \delta(2/\theta+\varepsilon'') \] for some $\varepsilon''>0$
by the assumption on $\delta$. This implies that $N_{n \delta}/{n \delta} > 1/\theta + \varepsilon''$, as asserted. 
By the SLLN for $N_t$, the r.h.s.\ of~\eqref{ndelta} converges to $0$ as $n\to\infty$.
This shows \eqref{e:prFCLTRenRew1}. In particular, $\varphi_n$ converges in probability with respect to the Skorokhod topology to the linear function $t \mapsto t/\theta$. Using a time-change argument as in Section~17 of \cite{Bil68} (see in particular (17.7)--(17.9) and Theorem~4.4 therein), we conclude that $t \mapsto \widehat{W}^{(n)}_{\varphi_n(t)}$
converges to a Brownian motion time-changed by $t \mapsto t/\theta$, or equivalently, to a Brownian motion multiplied by $1/\sqrt{\theta}$. 
To compare with $Z_t$, note that
\[
  \Big|\frac{Z_{nt} - nt v}{\sigma \sqrt{n}} - \sqrt{\theta}\widehat{W}^{(n)}_{\varphi_n(t)}   \Big|
    = \frac{|v|}{\sigma}\cdot \frac{| T_{N_{nt}} - nt |}{\sqrt{n}}
    \leq \frac{|v|}{\sigma}\cdot \frac{T_{N_{nt}+1} - T_{N_{nt}}}{\sqrt{n}}
\]
so that, for any $M,\varepsilon>0$, there is an $\varepsilon'>0$ such that
\[
\begin{aligned}
  \P\bigg( 
   & \sup_{t \in [0,M]}\Big| \frac{Z_{nt} - nt v}{\sigma \sqrt{n}} - \sqrt{\theta}\widehat{W}^{(n)}_{\varphi_n(t)}  \Big| \geq \varepsilon\bigg) \\[.5em]
   &\leq \P(N_{nM}> 2nM/\theta) + \P\big( \exists k \leq 2nM/\theta+1 \colon\, \tau_k \geq \varepsilon' \sqrt{n} \big) \\[.5em]
   &\leq \P(N_{nM}> 2nM/\theta) + (2nM/\theta+1) \P(\tau_1 \geq \varepsilon' \sqrt{n}).
\end{aligned}
\]
The first term above converges to $0$ as $n\to\infty$ by the LLN for $N_t$, while the second converges to $0$
since $\tau_1$ is square-integrable.
This shows that the Skorokhod distance between 
$\sqrt{\theta}\widehat{W}^{(n)}_{\varphi_n(t)} $ and $(Z_{nt} - nt v)/(\sigma \sqrt{n})$ converges to zero in probability,
concluding the proof.
\end{proof}

\end{appendix}

\begin{acks}[Acknowledgements] We thank Jason Schweinsberg for inspiring discussions and the Hausdorff Institute of Mathematics for its hospitality during the Junior Trimester Program ``Stochastic modelling in the life science: From evolution to medicine'' in 2022.  We also thank Joachim Krug and Su-Chan Park for pointing us to the reference \cite{guess1974limit}. We are grateful to two anonymous reviewers for valuable suggestions and comments.
\end{acks}

\begin{funding}
The third author was partially supported by CNPq grants 313921/2020-2, 406001/2021-9 and FAPEMIG grants APQ-02288-21, RED-00133-21.

Funding acknowledgements by the fourth author: This paper was supported by the János Bolyai Research Scholarship of the Hungarian Academy of Sciences.  Project no.\ STARTING 149835 has been implemented with the support provided by the Ministry of Culture and Innovation of Hungary from the National Research, Development and Innovation Fund, financed under the STARTING\_24 funding scheme.
\end{funding}

\bibliographystyle{imsart-nameyear}
\bibliography{literature_revision2}

\end{document}